\documentclass[12pt]{article}

\usepackage{latexsym} 
\usepackage{amsmath, amsfonts} 
\usepackage[dvips]{epsfig}

\usepackage[margin=1in]{geometry} 
\usepackage{amssymb}

\usepackage{xcolor} 

\allowdisplaybreaks[4]

\newcommand{\q}{\quad}
\newcommand{\p}{\partial}
\newcommand{\e}{\epsilon}
\newcommand{\n}{\nabla}
\newcommand{\om}{\int_\Omega}
\newcommand{\omt}{\iint_{\Omega_t}}

\newcommand{\dv}{{\rm div}}
\newcommand{\cl}{{\rm curl}}

\newcommand{\ef}{\eqref}
\newcommand{\wt}{\widetilde}

\newcommand{\be}{\begin{equation}}
\newcommand{\ee}{\end{equation}}

\newtheorem{thm}{Theorem}[section]

\newtheorem{lem}[thm]{Lemma}
\newtheorem{prop}[thm]{Proposition}
\newtheorem{rmk}[thm]{Remark}
\newtheorem{defn}[thm]{Definition}

\numberwithin{equation}{section}
\makeatletter
\renewcommand{\@seccntformat}[1]{
  \csname the#1\endcsname.\hspace{0.5em}
}
\makeatother

\begin{document}

\begin{center}\large  
\textbf{Low Mach number limit of the full Navier-Stokes equations with no-slip boundary conditions and large temperature variations }    
\end{center}  
\begin{center}      
{Xiaoyu Gu\textsuperscript{1}, Song Jiang\textsuperscript{1}, Qiangchang Ju\textsuperscript{1}, Yaobin Ou\textsuperscript{2,*}}
\vspace{1.6mm}\\    
\textsuperscript{1} \small Institute of Applied Physics and Computational Mathematics, Beijing 100088, China\\ 
\textsuperscript{2} \small School of Mathematics, Renmin University of China, Beijing 100872,   China \\ 
\textsuperscript{*} \small Corresponding author.\\  
\small{Emails:   
guxiaoyu@ruc.edu.cn, jiang@iapcm.ac.cn, ju\_qiangchang@iapcm.ac.cn, ou@ruc.edu.cn}
\end{center}

\noindent{\bf Abstract:}{  In this paper, we rigorously justify the low Mach number limit of the full compressible Navier-Stokes equations  in  a three-dimensional smooth bounded domain, subject to the no-slip boundary condition for the velocity and  well-prepared initial data. In the presence of heat conduction and large temperature variations, we establish  the uniform estimates for strong solutions on a time interval independent of the Mach number by constructing an  energy functional in a Sobolev framework.  
The no-slip boundary condition cannot  provide the vorticity boundary relation used in the slip-boundary setting. 
To overcome this difficulty, a divergence-free Stokes lifting of the tangential  boundary trace  for the time derivative of the  modified velocity is employed  to define a test function vanishing on the boundary, which  enables  the transfer of {\it a priori} estimates from the velocity to the temperature.
}


\section{Introduction and main results} 
We consider the full compressible Navier-Stokes equations in a smooth bounded domain $\Omega\subset\mathbb R^3$:
\begin{equation}\label{ns}\left\{\begin{aligned}
&\rho_t+\dv(\rho\mathbf{u})=0,  \\
&\rho(\mathbf{u}_t+\mathbf{u}\cdot\nabla \mathbf{u})+\nabla p=\mu\Delta \mathbf{u} +(\mu+\lambda)\nabla\dv \mathbf{u} , \\
&\rho(e_t+\mathbf{u}\cdot\nabla e)+p\dv \mathbf{u}-\kappa\Delta \Theta=2\mu|D(\mathbf{u})|^2+\lambda (\dv \mathbf{u})^2.
\end{aligned}\right.\end{equation}  
Here, $\rho,  \mathbf{u}, e,\Theta,$ and $p$ denote the  density,  velocity field, internal energy, temperature  and pressure,  respectively, and $D(\mathbf{u})=\frac12(\nabla \mathbf{u}+\nabla \mathbf{u}^{\mathsf T})$. 
The constants $\mu$ and $ \lambda$ denote the shear and bulk viscosity coefficients, respectively, and satisfy the physical conditions $\mu>0$ and $2\mu+3\lambda\geq0 $. The constant $\kappa>0$ denotes the thermal conductivity.
 We assume that the fluid is an ideal polytropic gas governed by the equations of state
\be\label{state}
p=R\rho \Theta,\q e=C_v \Theta,
\ee
where $R>0$ is the specific gas constant and $C_v>0$ is the specific heat at constant volume.

Let $\e\in(0,1]$ denote the Mach number, the ratio of a characteristic flow velocity to a characteristic sound speed.
In this paper, we investigate the low Mach number limit of the compressible flow regime for $\e \in(0,1]$.  
To this end, we introduce the following change of variables:
$$
t \rightarrow \e t,~ x \rightarrow  x,~  \mu  \rightarrow   \e \mu, ~ \lambda  \rightarrow   \e \lambda, ~ \kappa  \rightarrow   \e \kappa,~  
 \mathbf{u} \rightarrow \e \mathbf{u},~ \rho \rightarrow  \rho,~  \Theta \rightarrow  \Theta.
$$  
Thus, the system \ef{ns} takes the dimensionless form 
\begin{equation}\label{non-ns}\left\{\begin{aligned}
&\rho^\e_t+\dv(\rho^\e\mathbf{u}^\e)=0,   \\
&\rho^\e(\mathbf{u}^\e_t+\mathbf{u}^\e\cdot\nabla \mathbf{u}^\e)+\frac{1}{\e^2}\nabla p^\e=\mu\Delta \mathbf{u}^\e+(\mu+\lambda)\nabla\dv \mathbf{u}^\e,   \\ 
&C_v\rho^\e(\Theta_t^\e+\mathbf{u}^\e\cdot\nabla \Theta^\e)+p^\e \dv\mathbf{u}^\e=\kappa\Delta\Theta^\e +\e^2[2\mu|D(\mathbf{u}^\e)|^2+\lambda(\dv\mathbf{u}^\e)^2].
\end{aligned}\right.\end{equation}

A formal asymptotic expansion suggests that, as $\e\to0$, the solutions to \eqref{non-ns} converge to a limit $(r,\omega, \pi )$ satisfying the following low Mach number system (see Chapter 1 of \cite{Lions}):
\begin{equation}\label{ins0}\left\{\begin{aligned}
&r_t+\dv(r\omega)=0,   \\
&r(\omega_t+\omega\cdot\nabla\omega)+ \nabla \pi =\mu\Delta \omega +(\mu+\lambda)\nabla\dv \omega , \\ 
& \gamma C_v \dv\omega = \dv\bigg [\kappa \n\bigg(\frac{1}{r} \bigg ) \bigg],
\end{aligned}\right.\end{equation}
where $\gamma=1+R/C_v$ denotes the ratio of specific heats.   
The purpose of the low Mach number approximation is to verify that the compression due to pressure variations can be neglected.
When heat conduction is absent, or when the limiting temperature is spatially constant, the right-hand side of $\ef{ins0}_3$ vanishes.
Thus, the divergence constraint reduces to
$$
\dv \omega =0,       
$$  
and the corresponding low Mach number limit is referred to as the incompressible limit. 
There is an extensive literature on the incompressible limit of compressible flows under various assumptions.
However, the general low Mach number model \ef{ins0} filters acoustic compressibility while retaining the effects of temperature variations and heat conduction through a nontrivial divergence constraint. This motivates the study of the low Mach number limit for the full compressible Navier-Stokes system in the presence of large temperature variations and heat conduction.

For isentropic flows, the incompressible limit was investigated extensively 
in the literature  \cite{Ebin, KM,Ukai,Sch, Beirao,LM, DGLM, Sec, Danchin,  DM,MRS,OY,MSWZ} and the references therein. 
In this case, the pressure depends only on the density. 
The changes of variables such as
$$ \rho^\e=1+\e\rho_1^\e \quad\text{or}\quad \rho^\e=e^{\e\rho_1^\e}, $$
can be introduced to obtain a skew-symmetric leading-order operator of the form 
$\e^{-1}(\dv,\nabla)$, which   yields the   $\e$-uniform estimates for strong solutions. 
Then the convergence can be demonstrated by   compactness arguments for well-prepared initial data,  and by the oscillatory group method for the ill-prepared initial data.

The non-isentropic problem is  more delicate when heat conduction or entropy production is present, owing to the coupling of the singular pressure term with both the density and the temperature. 
For the non-isentropic Euler system $(\mu=\lambda=\kappa=0)$, one may refer to    \cite{Ala04,Bresch,MS,Sch86} for extensive studies of the incompressible limit.    
For the non-isentropic Navier-Stokes system with $\kappa=0$, Kim and Lee \cite{KL} and Jiang and Ou \cite{JO} established the incompressible limit in $\mathbb R^3$ and   bounded domains, respectively, allowing large variations of both density and temperature. 
When the  heat conduction is present but   temperature variations are small, the authors in \cite{DJO,RO} treated the slip-boundary case by   energy methods, whereas   the no-slip case was studied in  \cite{ROzamp} using   interior   and boundary estimates.   
Moreover, the low Mach number limits of the weak solutions to  non-isentropic Navier-Stokes system with ill-prepared initial data were verified by Feireisl and  Novotný \cite{F,FN2007,FN2013} in $\mathbb R^3,\mathbb T^3 ,$ or an unbounded domain with a compact boundary. 
By contrast, when $\kappa>0$ and the initial temperature variations are large, the leading penalization operator is no longer skew-symmetric, preventing a direct energy argument and making $\e$-uniform estimates substantially more delicate. 
In $\mathbb R^d$ or $\mathbb T^d$, Alazard \cite{Ala} employed a localization in frequency space to establish uniform estimates for the solutions of the full Navier-Stokes equations with ill-prepared initial data and justified the low Mach number limit. 
 Jiang et al. \cite{JJLX} proved the low Mach number limit for the magnetohydrodynamic equations with ill-prepared initial data in $\mathbb R^3$  by energy estimates.  
In bounded domains, the boundary gives rise to additional difficulties in the control of boundary terms and higher-order normal derivatives.
 Ju and Ou \cite{JO2022} exploited the skew-symmetric structure associated with
$$(\mathbf{v}^\e,q^\e):= \left(\mathbf{u}^\e-\frac{\kappa}{2}\nabla\theta^\e,q^\e\right), $$
together with the vorticity-slip boundary condition, to derive $\e$-uniform {\it a priori} estimates and hence established the convergence for the well-prepared initial data. Recently, Sun \cite{Sun} subsequently treated the Navier-slip case within the framework of conormal Sobolev spaces for the  ill-prepared initial data. 
The analyses in \cite{JO2022,Sun} rely essentially on slip-type boundary conditions. However, the boundary  effects in the case of no-slip boundary condition are stronger, which   makes the uniform control of high-order normal derivatives more difficult. 
Therefore, the low Mach number problem for the full Navier-Stokes system with no-slip boundary conditions and large temperature variations remains open.

In this paper, we study the low Mach number limit of \ef{non-ns}  with no-slip boundary conditions  and large temperature variations in a smooth bounded domain $\Omega\subset\mathbb R^3$.  
We establish the $\e$-uniform estimates for strong solutions on a  time interval independent of $\e$, and consequently justify the convergence from \ef{non-ns}   to  \ef{ins0}, as $\e\rightarrow 0$.   
Large temperature variations destroy the skew-symmetric structure of the  singular terms for hyperbolic-parabolic system.
Under the no-slip boundary condition, the lack of vorticity boundary condition prevents us from obtaining high-order normal derivative estimates. 
Although the modified velocity restores the  skew-symmetric structure for  singular acoustic part, it remains strongly coupled with the temperature through  both its evolution equation and its tangential boundary trace. 
To overcome these difficulties, we exploit two skew-symmetric structures of the full compressible Navier-Stokes system: one coupling the temperature with  weighted unknowns, which  yields the uniform estimates for the time derivative of temperature, while the other coupling the pressure with the modified velocity, which is used to control the time derivative of  velocity and higher-order spatial derivatives. 
A key ingredient is the use of a divergence-free Stokes lifting of the boundary trace of the time derivative of modified velocity, which enables us to define a test function that vanishes on the boundary.

In what follows, we assume $R=C_v=1$ for brevity.
From  $\ef{non-ns}_1$ and $\ef{non-ns}_3$,  
one gets the pressure equation
\be 
\label{pres}
p_t^\e+\mathbf{u}^\e\cdot\nabla p^\e+ 2p^\e\dv\mathbf{u}^\e =  \kappa \Delta\Theta^\e + \e^2 [2\mu|D(\mathbf{u}^\e)|^2+\lambda(\dv\mathbf{u}^\e)^2].
\ee

Under the large temperature variations,  we introduce the change of variables (cf. \cite{JO2022}):  
\be\label{change1} 
p^\e=1+\e q^\e, \q \Theta^\e=1+\theta^\e.  
\ee 
By virtue of \ef{state}, one has $1+\e q^\e=\rho^\e(1+\theta^\e )$. 
Then we  adapt the compressible Navier-Stokes equations \ef{non-ns} and the pressure equation \ef{pres} into the  ``overdetermined'' system:
\begin{equation}\label{ns1}\left\{\begin{aligned}
&\rho^\e_t+\dv(\rho^\e\mathbf{u}^\e)=0  &{\rm in}\ \Omega_t, \\
&\rho^\e(\mathbf{u}^\e_t+\mathbf{u}^\e\cdot\nabla \mathbf{u}^\e)+\frac{1}{\e}\nabla q^\e=\mu\Delta \mathbf{u}^\e+(\mu+\lambda)\nabla\dv \mathbf{u}^\e &{\rm in}\ \Omega_t, \\
&\rho^\e(\theta_t^\e+\mathbf{u}^\e\cdot\nabla \theta^\e)+(1+\e q^\e)\dv\mathbf{u}^\e =\kappa\Delta\theta^\e+\e^2[2\mu|D(\mathbf{u}^\e)|^2+\lambda(\dv\mathbf{u}^\e)^2] &{\rm in}\ \Omega_t, \\
&\frac{1}{2}(q_t^\e+\mathbf{u}^\e\cdot\nabla q^\e)+ q^\e\dv\mathbf{u}^\e+\frac{1}{\e}\dv\mathbf{u}^\e-\frac{\kappa}{2\e}\Delta\theta^\e =\frac{\e}{2} [2\mu|D(\mathbf{u}^\e)|^2+\lambda(\dv\mathbf{u}^\e)^2] &{\rm in}\ \Omega_t, 
\end{aligned}\right.\end{equation}
where  $\Omega_t:=\Omega\times(0, t]$ and each of $\ef{ns1}_{1,3,4}$  coincides with the others. 
Assume that $(\rho^\e, \mathbf{u}^\e,q^\e,\theta^\e)$ satisfies the initial condition: 
\be\label{ini}     
(\rho^\e, \mathbf{u}^\e, q^\e, \theta^\e)|_{t=0}= (\rho_0^\e, \mathbf{u}_0^\e, q_0^\e, \theta_0^\e) :=(\rho_0^\e, \mathbf{u}_0^\e, (\rho_0^\e\Theta_0^\e-1)/\e, \Theta_0^\e-1)  \q {\rm in}\ \Omega,   
\ee    
and the boundary conditions:   
\be\label{bdry}   
\mathbf{u}^\e=0, \q \frac{\partial \theta^\e}{\partial \mathbf{n}}=0  \q {\rm on}~ \partial \Omega_t,
\ee  
where $\partial\Omega_t:=\partial\Omega\times(0, t]$ and $\mathbf{n}$ denotes the unit outward normal vector to $\Omega$.
Due to the presence of the term $-\frac{\kappa}{2\e}\Delta \theta^\e$ in $\ef{ns1}_{4}$,  the singular differential operators of $O(\frac 1 \e)$ in $\ef{ns1}_{2,4}$ are not skew-symmetric.  
Define the modified velocity  (see \cite{Ala,JO2022}) by
$$\mathbf{v}^\e:=\mathbf{u}^\e-\frac \kappa 2 \nabla \theta^\e.$$
Then the equations  \ef{ns1}, \ef{ini}, and \ef{bdry}  can be  written equivalently as  follows:
\begin{equation}\label{ns0}\left\{\begin{aligned} &\rho^\e_t+\dv(\rho^\e\mathbf{u}^\e)=0  &{\rm in}\  \Omega_t, \\ &\rho^\e(\mathbf{v}^\e_t+\mathbf{u}^\e\cdot\nabla \mathbf{v}^\e)+\frac{1}{\e}\nabla q^\e-\mu\Delta \mathbf{v}^\e-\xi\nabla\dv \mathbf{v}^\e-\beta\nabla\Delta\theta^\e =f^\e  &{\rm in}\  \Omega_t, \\
&\frac{1}{2} (q_t^\e+\mathbf{u}^\e\cdot\nabla q^\e)+\frac{1}{\e}\dv\mathbf{v}^\e=g^\e  &{\rm in}\  \Omega_t, \\
&\rho^\e(\theta_t^\e+\mathbf{u}^\e\cdot\nabla \theta^\e)-\frac{\kappa}{2} \Delta\theta^\e +\dv\mathbf{v}^\e=h^\e  &{\rm in}\  \Omega_t, \\ 
&\mathbf{v}^\e\cdot\mathbf{n}=0 ,\q  \mathbf{v}^\e  =-\frac{\kappa}{2} \n_\tau\theta^\e, \q \frac{\partial \theta^\e}{\partial \mathbf{n}}=0  &{\rm on}\  \partial\Omega_t, \\
&(\rho^\e, \mathbf{v}^\e, q^\e, \theta^\e)|_{t=0}= (\rho_0^\e, \mathbf{u}_0^\e-\frac{\kappa}{2}\n\theta_0^\e, q_0^\e, \theta_0^\e)(x) & {\rm in}\ \Omega,  \end{aligned}\right.\end{equation}
where   $\n_\tau\theta^\e|_{\p\Omega_t}:=  \n\theta^\e-  \frac{\p\theta^\e}{\p\mathbf{n}}\mathbf{n}$, $\xi:=\mu+\lambda+\frac{\kappa}{2}$, $\beta:= \frac{\kappa}{2}(2\mu+\lambda-\frac{\kappa}{2})$, and 
\begin{equation}\label{fgh}\begin{aligned}
f^\e&:=\frac{\kappa}{2} [\nabla \rho^\e(\theta^\e_t+\mathbf{u}^\e\cdot\nabla\theta^\e)+\rho^\e\nabla \mathbf{u}^\e\nabla\theta^\e-\nabla h^\e], \\ g^\e&:=-q^\e\dv\mathbf{u}^\e+\frac \e 2 [2\mu|D(\mathbf{u}^\e)|^2+\lambda(\dv \mathbf{u}^\e)^2 ], \\ h^\e&:=-\e q^\e\dv\mathbf{u}^\e+\e^2 [2\mu|D(\mathbf{u}^\e)|^2+\lambda(\dv \mathbf{u}^\e)^2 ].
\end{aligned}\end{equation} 
Note that the singular terms of $O(\frac{1}{\e})$ in the equations $\ef{ns0}_{2,3}$ for $(\mathbf{v}^\e,q^\e)$  are skew-symmetric.  
However, the equation for $\mathbf{v}^\e$ contains a third-order derivative term of the temperature $-\beta \n \Delta \theta^\e$, which gives rise to   essential difficulties in obtaining  uniform estimates.

\begin{rmk} 
The boundary conditions for $\mathbf{v}^\e$ in 
$\ef{ns0}_5$  are  compatible.
Indeed, $\mathbf{v}^\e\cdot \mathbf{n}|_{\p\Omega_t}=0$ can be obtained by $\frac{\p\theta^\e}{\p\mathbf{n}}|_{\p\Omega_t}=0$ and $  \mathbf{v}^\e |_{\p\Omega_t}=-\frac{\kappa}{2} \n_\tau\theta^\e$. 
We emphasize that
the tangential boundary trace of $\mathbf{v}^\e$ plays a key role in
establishing uniform estimates for $\p_t\mathbf{v}^\e$.
\end{rmk}

We first state the following local well-posedness result for \ef{non-ns}-\ef{ini}-\ef{bdry}. It can be derived by arguments similar to those employed in the proof of local well-posedness for the compressible Navier-Stokes system with Dirichlet boundary conditions (see, e.g. \cite{CCK,MSWZ,MN80}).

\begin{thm}[Local existence] \label{le}  
Let $\epsilon \in (0, 1]$ be a fixed number and $\Omega \subset \mathbb{R}^3$ be a simply connected, bounded domain with smooth boundary $\partial\Omega$. Assume that the initial data $\rho_0^\epsilon\in H^2(\Omega), \mathbf{u}_0^\epsilon \in H^3(\Omega),$ and  $\Theta_0^\epsilon \in H^4(\Omega)$ satisfy $\rho_0^\epsilon \geq m > 0$ and $ \Theta_0^\epsilon \geq m$ for a given constant $m>0$, and 
$$ 
( \rho_t^\epsilon(0),   \mathbf{u}_t^\epsilon(0) ) \in H^{1}(\Omega), \quad  \Theta_t^\epsilon(0) \in H^{2}(\Omega), \q ( \rho_{tt}^\epsilon(0),   \mathbf{u}_{tt}^\epsilon(0) ,\Theta_{tt}^\epsilon(0)) \in L^2(\Omega).
$$  
Let the following compatibility conditions be satisfied:  
$$
\mathbf{u}_0^\epsilon  = \mathbf{u}_t^\epsilon(0) =  0, \q \frac{\partial\Theta_0^\epsilon}{\partial \mathbf{n}} = \frac{\partial\Theta_t^\epsilon(0)}{\partial \mathbf{n}}  = 0 \quad \text{on} ~ \partial\Omega.  
$$  
Then there exists a positive constant $T^\epsilon = T^\epsilon(\rho_0^\epsilon, \mathbf{u}_0^\epsilon, \Theta_0^\epsilon, m, \epsilon)$, such that the initial-boundary value problem \ef{non-ns}-\ef{ini}-\ef{bdry} admits a unique strong solution $(\rho^\epsilon, \mathbf{u}^\epsilon, \Theta^\epsilon)$, satisfying 
\be\nonumber
\begin{aligned}
&\rho^\epsilon > 0,\quad 1 + \theta^\epsilon > 0 \quad \text{ in } \Omega \times (0, T^\epsilon), \\
&(\partial_t^k \rho^\epsilon, \partial_t^k \mathbf{u}^\epsilon) \in C([0, T^\epsilon], H^{2-k}(\Omega)),  \quad \partial_t^k \mathbf{u}^\epsilon \in L^2(0, T^\epsilon; H^{3-k}(\Omega)),\q  k=0,1,2,\\
&   \Theta^\epsilon \in C([0, T^\epsilon], H^{3 }(\Omega)) \cap L^2(0, T^\epsilon; H^{4}(\Omega)),  ~ \Theta^\epsilon _t \in C([0, T^\epsilon], H^{1 }(\Omega)) \cap L^2(0, T^\epsilon; H^{2}(\Omega)), \\
&\Theta^\epsilon_{tt} \in C([0, T^\epsilon], L^{2}(\Omega)) \cap L^2(0, T^\epsilon; H^{1}(\Omega)). 
\end{aligned} 
\ee  
\end{thm}

\begin{rmk} 
For brevity, we write $\rho_t^\e(0)$ for the quantity 
$$   \rho_t^\e|_{t=0} := -\dv(\rho_0^\e \mathbf{u}_0^\e), $$
  obtained by setting $t=0$ in equation $\eqref{non-ns}_1$. The quantities $\mathbf{u}_t^\e(0)$  and $\Theta_t^\e(0)$
  are defined by setting $t=0$ in   $\eqref{non-ns}_{2,3}$.  In addition, we can define $\mathbf{u}_{tt}^\e(0),\rho_{tt}^\e(0)$,  and $\Theta_{tt}^\e(0)$ inductively.  
\end{rmk}

\begin{defn}   
For $t\in (0,1]$, we define the energy functional as 
\be\nonumber\begin{aligned}
M^\e (t)&:=\|(\rho^\e, q^\e)\|^2_{L^\infty (0, t;H^2(\Omega))}+\|\mathbf{u}^\e\|^2_{L^\infty (0, t;H^1(\Omega))}+\|\theta^\e\|^2_{L^\infty (0, t;H^2(\Omega))}\\
&\q +\|(\rho^\e)^{-1}\|^2_{L^\infty ((0, t)\times\Omega)} +\|(\rho_t^\e, \mathbf{u}_t^\e, q_t^\e)\|^2_{L^\infty (0, t;L^2(\Omega))}+\|\theta_t^\e\|^2_{L^\infty (0, t;H^1(\Omega))}
\\
&\quad + \|(\mathbf{u}^\e, \nabla\theta^\e)\|^2_{L^2 (0, t;H^3(\Omega))}+\|(\mathbf{u}_t^\e,  \nabla\theta_t^\e)\| ^2_{L^2(0, t;H^1(\Omega))}.
\end{aligned}\ee
\end{defn}

We have  the following uniform estimates theorem. 

\begin{thm}[Uniform estimates] \label{uethm}
Let $\Omega\subset \mathbb{R}^3$ be a simply connected, bounded domain with smooth boundary $\p\Omega$. Suppose that  
\be\begin{aligned} \label{data}
&\|(\rho_0^\e, \mathbf{u}_0^\e,  q_0^\e )\|_{H^2(\Omega)}   + \|\theta_0^\e \|_{H^3(\Omega)}+\|(\rho_t^\e , \mathbf{u}_t^\e ,  q_t^\e )(0)\|_{L^2(\Omega)}   + \|\theta_t^\e(0) \|_{H^1(\Omega)}    \\
&+ \|\e ( \n\mathbf{u}_{t}^\e ,  \n q_{t}^\e )(0)\|_{L^2(\Omega)}   + \|(  \e^2 \mathbf{u}_{tt}^\e ,  \e^2  q_{tt}^\e,  \e  \theta_{tt}^\e )(0)\|_{L^2(\Omega)}   \leq M_0,  
\end{aligned}  
\ee 
$\rho_0^\e\geq M_0^{-1}>0$, and  $\Theta_0^\e \geq A>0$  for certain positive constants $M_0$ and $A$ independent of $\e$. Assume that $\beta>0$. Then there exist positive constants  $\e_0=\e_0(\Omega, M_0,  A)$,  $T=T(\Omega, M_0,  A)$ and $C=C(\e_0, T)$ independent of $\e$ such that the unique solution $(\rho^\e, \mathbf{u}^\e, q^\e, \theta^\e)$ of \ef{ns1}-\ef{ini}-\ef{bdry} exists on $\Omega\times(0, T]$, which is obtained in Theorem \ref{le}, and satisfies the uniform bound: 
\be\label{bound} 
M^\e(t)\leq C, \q  {\rm for~all}~ \e\in (0, \e_0]~ {\rm and}~ t\in(0, T], 
\ee
and   
\be \nonumber
\Theta^\e (x, t) \geq C^{-1}, \q (x, t)\in \Omega\times [0, T].
\ee
\end{thm}

Consequently, the compactness arguments imply the subsequent convergence results.  
\begin{thm}[Low Mach number limit] 
Let $\Omega\subset \mathbb{R}^3$ be a simply connected, bounded domain with smooth boundary $\p\Omega$.
Suppose that $(\rho^\e, \mathbf{u}^\e, q^\e, \theta^\e)$ is the solution to \ef{ns1} in $\Omega\times(0, T]$,  which satisfies the uniform boundedness \ef{bound}. Then as $\e\rightarrow 0$, there exists a subsequence (still denoted by the same notation) such that the convergence results hold  
$$\mathbf{u}^\e   \rightharpoonup  \omega \q {weakly-*}~{\rm in }~L^\infty(0, T;H^1(\Omega)), $$
$$ 
(\rho^\e,   \theta^\e) \rightharpoonup ( r, \theta )\q  {weakly-*}~{\rm in }~L^\infty(0, T;H^2(\Omega)), 
$$    
with $r(1+\theta)=1$. Moreover,  there is a function $\pi(x, t)$  such that $(\omega, \theta, \pi) $ satisfies  the low Mach number Navier-Stokes equations: 
\begin{equation}\label{ins} 
\left\{\begin{aligned}
&(1+\theta)^{-1}(\omega_t+\omega\cdot\nabla \omega)+ \nabla \pi=\mu\Delta \omega+(\mu+\lambda)\nabla\dv \omega &{\rm in}\ \Omega_T, \\
&(1+\theta)^{-1}(\theta_t +\omega\cdot\nabla \theta ) =\frac{\kappa}{2}\Delta\theta    &  {\rm in}\ \Omega_T, \\
&\dv \omega= \frac{\kappa}{2}\Delta\theta   &{\rm in}\ \Omega_T, \\
&\omega=0, \q\frac{\partial \theta }{\partial \mathbf{n}}=0   &{\rm on}\ \partial \Omega_T, \\ 
&(r, \omega,\theta)|_{t=0}= (r_0, \omega_0,\theta_0)(x)   &{\rm in}\ \Omega, 
\end{aligned}\right.\end{equation} 
where $(r_0, \omega_0,\theta_0)$ is the weak limit of $(\rho^\e, \mathbf{u}^\e,\theta^\e)|_{t=0}$ in $H^2(\Omega)$. 
\end{thm}
\textbf{Proof}.  Due to the uniform boundedness \ef{bound},  there exists a subsequence (still denoted by the same notation) such that the following convergence holds: 
\begin{equation} 
\label{com1}
\begin{aligned}   
& \rho^\epsilon  \rightharpoonup r \quad  \text{weakly}-* \text{ in } L^\infty(0, T; H^2(\Omega))  , \\   
& \mathbf{u}^\epsilon \rightharpoonup   \omega \quad \text{weakly}-* \text{ in } L^\infty(0, T; H^1 (\Omega))~  \text{and} ~\text{weakly}~\text{in} ~  L^2(0, T; H^3(\Omega)),  
\\     
& \theta^\epsilon \rightharpoonup \theta \quad \text{weakly}-* \text{ in } L^\infty(0, T; H^2 (\Omega))  ~  \text{and} ~\text{weakly}~ \text{in } L^2(0, T; H^4(\Omega)) , 
\\   
& \rho_t^\epsilon  \rightharpoonup  r_t  \quad \text{weakly}-* \text{ in }  L^\infty(0, T; L^2(\Omega)),\\ 
& \mathbf{u}_t^\epsilon \rightharpoonup \omega_t \quad \text{weakly}-* \text{ in } L^\infty(0, T; L^2(\Omega))  \text{~and } ~\text{weakly}~ \text{in ~} L^2(0, T; H^1(\Omega)) , \\ 
& \theta_t^\epsilon \rightharpoonup \theta_t \quad  \text{weakly}-* \text{ in } L^\infty(0, T; H^1(\Omega)) ~ \text{and} ~  \text{weakly}~\text{in } L^2(0, T; H^2(\Omega)) .
\end{aligned}
\end{equation}  
Then the Aubin-Lions-Simon compactness lemma yields 
\begin{equation} 
\label{com2} 
\begin{aligned} 
& (\rho^\epsilon, \theta^\e)  \rightarrow  (r,\theta)  \q &  \text{ in }~ C ([0, T],  H^{2-\delta} (\Omega)),  \\
&  \mathbf{u}^\epsilon \rightarrow \omega \q & \text{ in } ~ C ([0, T],  H^{1-\delta} (\Omega)),  
\end{aligned}
\end{equation} 
for any $\delta\in (0,1)$. 
Next, we aim to prove that \ef{ns1} converges to \ef{ins} in $L^2(0,T; L^2(\Omega))$ by applying the convergence results \ef{com1} and \ef{com2}.  From the definition $p^\e=\rho^\e(1+\theta^\e)=1+\e q^\e$, it holds that 
$$ r(1+\theta)=1 $$ 
when $\e\rightarrow 0$.  Passing to the limit  as $\e\rightarrow 0 $  in $\ef{ns1}_{3,4}$ yields $\ef{ins}_{2,3}$. For the momentum equation, applying $\cl$ to $\ef{ns1}_{2}$ to remove the singular term and letting $\e\rightarrow 0 $ yield that  
$$
\cl[ r(\omega_t+\omega\cdot \n \omega ) -\mu \Delta \omega-(\mu+\lambda) \n \dv\omega ]=0 .
$$ 
Therefore, we conclude that there is a function $\pi(x,t)$  with  $\n\pi\in L^2(0, T; L^2(\Omega))$ such that the equation  $\ef{ins}_{1}$ holds.
\hfill$\Box$

\begin{rmk}  
The construction of test functions from solutions to Stokes problems can also be adapted to the analysis of singular limits for other fluid systems (e.g., the compressible magnetohydrodynamic equations) with no-slip boundary conditions.
\end{rmk}

The rest of this paper is organized as follows. Section 2 discusses the difficulties and strategies and provides a detailed explanation of the main ideas in the upper half-space. Section 3 recalls some well-known lemmas. Section 4 establishes the uniform energy estimates in general bounded domains.

Finally, we introduce the notation that will be used throughout this paper. The standard Sobolev space  $W^{k,2}(\Omega)$ is denoted by $H^k(\Omega)$, equipped with the norm $\|\cdot\|_{H^k}$. We also adopt the following abbreviations:
 $$ L^p_t(H^k) \equiv L^p(0,t; H^k(\Omega)), \quad \|\cdot\|_{L^p_t(H^k)} \equiv \|\cdot\|_{L^p(0,t; H^k)}.
 $$ 
 Throughout the paper, $C$ represents a generic positive constant whose value may vary from line to line but remains independent of $\e$. Furthermore, the notation  $f \lesssim g$ indicates that $f \leq Cg$. 
From now on,  we omit the superscript $\e$ of all variables for simplicity.

\section{Main difficulties and strategies}

\q   Two main difficulties arise in the present problem. The first is caused by the boundary. In contrast to the whole-space and periodic settings considered in \cite{Ala,Bresch,KL}, the  difficulty lies in establishing higher-order estimates uniform in $\e$  on a time interval independent of $\e$. 
In a bounded domain, acoustic waves interact with the boundary. Hence, the Fourier-based methods used in domains without boundaries are no longer directly available. Moreover, normal differentiation produces boundary terms that are not controlled by the basic energy estimates. Thus, high-order normal derivatives cannot be estimated by a direct differentiation argument.

The second difficulty is essentially structural and arises from heat conduction and large temperature variations. When  $\kappa=0 $, the leading acoustic operator has the skew-symmetric form 
$\e^{-1}(\dv,\nabla)$
as exploited in \cite{Ala04,JO,KL,MS}, and the    singular terms cancel in the $L^2$ energy identity.    
For $\kappa>0$ and large temperature variations, however, the penalization operator is no longer skew-symmetric in the original variables, owing to the nonlinear coupling of the singular pressure term with both the density and the temperature.

In order to recover the skew-symmetric structure in the case of large temperature variations, one effective method is to introduce a modified velocity $\mathbf{v}:=\mathbf{u}-\frac{\kappa}{2}\n \theta$  (cf. \cite{Ala,JO2022,Sun}). 
However, this transformation does not fully decouple $\mathbf {v}$ from the temperature. 
The equation for $\mathbf {v}$ still contains third-order derivatives of temperature, while its tangential trace depends on the tangential temperature gradient under the no-slip boundary condition. 
In the vorticity-slip and Navier-slip settings considered in \cite{JO2022,Sun}, high-order estimates are closed through a div-curl decomposition of $\mathbf {v}$, while the estimate for the curl part relies crucially on the boundary condition satisfied by  $\cl\mathbf v $. This mechanism is not available under the no-slip condition. Consequently, the high-order normal estimates obtained in \cite{JO2022,Sun} cannot be transferred directly to the present setting.

To overcome  these difficulties, we   employ two  skew-symmetric structures of the full compressible Navier-Stokes equations and a divergence-free Stokes lifting of the tangential  boundary trace of  $\mathbf v_t$  to close the  $\e$-uniform high-order estimates.  On the one hand, we utilize the system \ef{ns0} satisfied by $(q,\mathbf{v},\theta)$ to estimate the high-order norms   $\|(\n^2 \dv \mathbf{v},\n^2\Delta \theta)\|_{L_t^2 (L^2)}^2$ and the low-order norms $\|  \mathbf{v}_t \|_{L_t^\infty (L^2) }^2+\|  \mathbf{v}_t \|_{ L_t^2 (H^1)}^2$.  
On the other hand, the  system to be introduced below is used for estimating  $\|  \theta_t \|_{L_t^\infty (H^1)}^2+ \|  \theta_t \|_{  L_t^2 (H^2)}^2$. Next we present our main ideas for   the  uniform estimates.  

 First, the estimates of density and the low-order estimates $(q ,\mathbf{v} ,\theta )$ can be obtained by using the standard hyperbolic regularity and   Newton-Leibniz's formula  for time variables.

Next, we introduce  the following exponential transformation (see \cite{Ala,Sun}),  
\be\label{change2} 
p =e^{\e \sigma }, \quad \Theta =e^{\vartheta }, 
\ee 
which is equivalent to the linear transformation \ef{change1} in the sense of that
\be\label{rela}
\e q =e^{\e \sigma }-1 , \quad \theta =e^{\vartheta }-1. 
\ee
Thus the equations $\ef{non-ns}_2$,  $\ef{non-ns}_3$, and \ef{pres} can be  rewritten as 
\begin{equation}\left\{
\begin{aligned} \label{1ns}
& \frac{1}{2} ( \sigma _t + \mathbf{u} \cdot \n \sigma )+\frac{1}{\e}\dv \mathbf{u}  -\frac{1}{\e}\frac{\kappa}{2} \frac{1}{p } \dv (\Theta  \n\vartheta )=\frac{\e }{2 p }  [2\mu|D(\mathbf{u} )|^2+\lambda(\dv\mathbf{u}  )^2] &{\rm in }~ \Omega_t, \\
& \frac{1}{\Theta } (\mathbf{u}_t +\mathbf{u} \cdot \n \mathbf{u} )+\frac{1}{\e}\n \sigma  - \frac{1}{p } ( \mu\Delta\mathbf{u}  +(\mu+\lambda)\n\dv \mathbf{u}   )=0 & {\rm in }~ \Omega_t, \\
&\vartheta_t  +\mathbf{u} \cdot \n \vartheta  +\dv \mathbf{u}  -\kappa  \frac{1}{p  } \dv (\Theta  \n\vartheta )=\frac{\e^2 }{p }[2\mu|D(\mathbf{u} )|^2+ \lambda(\dv\mathbf{u}  )^2]  &{\rm in }~ \Omega_t. 
\end{aligned} \right. 
\end{equation}  
Following \cite{Ala}, we set
$$
\eta:=\frac{1}{2}(\e\sigma- \vartheta)
\quad {\rm and}\quad
\mathbf u_\e:=\e\mathbf u.
$$
A direct computation shows that
$(\eta,\mathbf u_\e,\vartheta)$ satisfies the following system:
\begin{equation}\label{2ns}
\left\{
\begin{aligned} 
&  \eta_t  +\mathbf{u} \cdot \n  \eta    +\frac{1}{2 \e}\dv \mathbf{u}_\e  = 0 &{\rm in }~ \Omega_t, \\
& \frac{1}{\Theta } (\mathbf{u}_{\e t} +\mathbf{u}  \cdot \n\mathbf{u}_{\e}   ) +\frac{1}{\e}(2\n \eta  + \n \vartheta ) - \frac{1}{p } ( \mu\Delta\mathbf{u}_\e  +(\mu+\lambda)\n\dv \mathbf{u}_\e   )=0 &{\rm in }~ \Omega_t, \\ 
& \vartheta_t  +\mathbf{u} \cdot \n  \vartheta   +\frac{1}{\e} \dv \mathbf{u}_\e -\kappa  \frac{1}{p  } \dv (\Theta  \n\vartheta )=\frac{\e^2 }{p }  [2\mu|D(\mathbf{u} )|^2+\lambda(\dv\mathbf{u}  )^2]
~ &{\rm in }~ \Omega_t,  \\
&\mathbf{u}_{\e}=0, \q  \frac{\p\vartheta  }{\p\mathbf{n}}=0  
&{\rm on}~  \p\Omega_t,  \\
&(\mathbf{u}_{\e t},  \eta_{t} ,  \vartheta_{t} )|_{t=0}=(\mathbf{u}_{\e t}(0),  \eta _{t} (0),  \vartheta _{t}(0))(x)  &{\rm in }~ \Omega,   
\end{aligned} \right. 
\end{equation}  
where $\mathbf{u}_{\e t}:= \p_t\mathbf{u}_{\e}$.

We  explain the roles of the two   systems under different transformations.
The system \ef{ns0}  under the linear transformation \ef{change1} has the 
explicit skew-symmetric structure for the modified velocity and the pressure.
The   system $\ef{1ns}$ under the exponential transformation \ef{change2} has an  additional advantage that the coefficient of $\vartheta_t $   is constant. Subtracting $  \frac{1}{2}\ef{1ns}_3   $ from $\e \ef{1ns}_1$ gives  $\ef{2ns}_1$. Thereby, the principal singular part  of $\ef{2ns}$  is skew-symmetric in the variables $(\eta ,\mathbf{u}_\e,\vartheta )$ and has the same form as the   small temperature variations case (see \cite{DJO,RO,ROzamp}).  
Thus, we shall establish the uniform estimates using the skew-symmetric structures  in both systems  \ef{ns0} and $\ef{2ns}$.

Using  \ef{2ns}, we establish the   estimates for
 \be\nonumber\begin{aligned}
M_1  (t)&:=    \| ( \eta_t  ,  \mathbf{u}_{\e t} ,  \vartheta_t     ) \|^2_{L^\infty (0, t;L^2(\Omega))}
+  \|  (   \n \mathbf{u}_{\e t} ,    \n \vartheta_t   )  \|^2_{L^2 (0, t;L^2(\Omega))} \\
&\q +
\| ( \n \eta_t ,  \e\dv \mathbf{u}_t , \nabla\vartheta_t  )\|^2_{L^\infty (0, t;L^2(\Omega))}+\|(\e  \dv \mathbf{u}_t ,  \nabla\vartheta_t  )\|^2_{L^2 (0, t;H^1(\Omega))},
\end{aligned}\ee 
and 
\be\nonumber\begin{aligned}
M_2  (t)&:= \|(\e^2  \sigma_{tt} ,  \e^2 \mathbf{u}_{tt}  , \e \vartheta_{tt}  )\|^2_{L^\infty (0, t;L^2(\Omega))}+\| (\e^2    \mathbf{u}_{tt} , \e\vartheta_{tt}  ) \|^2_{L^2 (0, t;H^1(\Omega))}    \\ 
&\q  + \| \e^2 \mathbf{u}_{t}   \|^2_{L^\infty (0, t;H^1(\Omega))}+ \| \e^2 \mathbf{u}_{t}   \|^2_{L^2 (0, t;H^2 (\Omega))},
\end{aligned}\ee  
   by employing standard $L^2$ energy estimates,     interior   and boundary estimates. These quantities provide the uniform control of $\vartheta_t $ in  $L^\infty(0,t;H^1(\Omega))
\cap L^2(0,t;H^2(\Omega))$. 
Relation  \ef{rela} then yields the corresponding $\e$-uniform estimate for  $\theta_t $.

We next seek an $\e$-uniform bound for  $\|\mathbf{v}_t \|_{L_t^\infty (L^2) }^2+\|  \mathbf{v}_t \|_{ L_t^2 (H^1)}^2$. 
A direct energy estimate is not available due to the nonhomogeneous boundary trace of $\mathbf{v}_t $.  Indeed, taking 
$\mathbf{v}_t $ as a test function in $\p_t\ef{ns0}_2$ yields a boundary term that cannot be controlled: 
$$-\om \Delta \mathbf{v}_t\cdot\mathbf{v}_t dx=\om |\n\mathbf{v}_t |^2 dx-\int_{\p\Omega } (\mathbf{n}\cdot \n \mathbf{v}_t  )\cdot \mathbf{v}_t  dS .$$
To overcome this obstruction, we construct a divergence-free Stokes lifting $\ell $ of the boundary data $\mathbf{v}_t|_{\p\Omega_t}=-\frac{\kappa}{2} \n_\tau \theta_t$ by solving the Stokes problem \ef{stokes} with $g= -\frac{\kappa}{2} \n_\tau \theta_t  $  and $\ell_0=0$.  
For the trace space specified in Lemma \ref{lem2.4}, the existence, uniqueness, and the required estimates follow from Theorem 6.1 in \cite{Fursikov}. 
We then use $Z:=\mathbf{v}_t -\ell$, which satisfies 
$$Z|_{\p\Omega_t}=0,\q {\rm and}\q \dv Z=\dv\mathbf{v}_t\q {\rm in}~\Omega_t, $$ 
as a new test  function  in  $\p_t \ef{ns0}_2$. 
Moreover, since $\ell $ is divergence-free and has zero normal trace, the correction by $\ell $ does not alter the cancellation of the singular pressure terms.  
Combining the resulting energy estimate with the bounds for $\ell $   in Lemma  \ref{lem2.4}, we obtain the desired $\e$-uniform estimate for $\mathbf v_t $.

We also remark that the method of constructing special test functions  is typically applicable to other problems as well. 
For instance, test functions constructed from the Stokes problem are used in \cite{CMT,LR}, while those based on  the Bogovskii problem are used in \cite{FJLX,MS22}.

Finally, we apply the Stokes theory to reduce the estimates of $\|\mathbf{u} \|_{L_t^2(H^3)}^2$ to $\|\n^2\dv\mathbf{u} \|_{L_t^2(L^2)}^2$. Because of
the strong coupling of the modified velocity and the temperature, we  control $\|\n^2\dv\mathbf{v} \|_{L_t^2(L^2)}^2$ and $\|\n^2\Delta \theta \|_{L_t^2(L^2)}^2$ simultaneously in the interior part and boundary part. 
We shall establish the uniform {\it a priori} estimates in general bounded domains in Section \ref{ue}.

\vskip4mm       

\textbf{Simplified problem in the upper half-space:}  To clarify our main ideas, we shall outline the derivation of uniform estimates for a simplified system   in $\Omega_t:=\Omega\times(0,t]$ with $\Omega=\mathbb{R}^3_+$, which reads as follows: 
\begin{equation}\label{sim0}\left\{\begin{aligned} 
&\frac{1}{2} q_t +\frac{1}{\e}(\dv\mathbf{u}-\frac{\kappa}{2}\Delta \theta)=0  &{\rm in}\  \Omega_t, \\
& \mathbf{u}_t +\frac{1}{\e}\nabla q-\mu\Delta \mathbf{u} -(\mu+\lambda)\nabla\dv \mathbf{u}   =0 &{\rm in}\  \Omega_t, \\
& \theta_t-  {\kappa} \Delta\theta  +\dv\mathbf{u} =0  &{\rm in}\  \Omega_t, \\
&\mathbf{u}=0,\q \p_3\theta =0 &{\rm on}\  \p\Omega_t,\\
&(q ,  \mathbf{u} ,  \theta )|_{t=0}= (  q_0, \mathbf{u}_0,  \theta_0)(x)  &{\rm in}\ \Omega. 
\end{aligned}\right.\end{equation}  
 In the half-space $\mathbb{R}^3_+$,  we denote by $\partial_{i},i=1,2$  the tangential derivative and by $\partial_{3}$  the normal derivative. 
By introducing the modified velocity $\mathbf{v}:=\mathbf{u}-\frac{\kappa}{2}\n \theta$  (see \cite{Ala,JO2022,Sun}),  we recover the skew-symmetric structure for the pressure and the modified velocity:
\begin{equation}\label{sim1}\left\{\begin{aligned} 
&\frac{1}{2} q_t +\frac{1}{\e}\dv\mathbf{v}=0  &{\rm in}\  \Omega_t, \\
& \mathbf{v}_t +\frac{1}{\e}\nabla q-\mu\Delta \mathbf{v} -\xi\nabla\dv \mathbf{v} -\beta\nabla\Delta\theta  =0 &{\rm in}\  \Omega_t, \\
& \theta_t- \frac{\kappa}{2} \Delta\theta  +\dv\mathbf{v} =0  &{\rm in}\  \Omega_t, \\
&\mathbf{v}  =- \frac{\kappa}{2} \n\theta  ,\q \p_3\theta  =0 &{\rm on}\  \p\Omega,\\
&(q ,  \mathbf{v} ,  \theta )|_{t=0}= (  q_0, \mathbf{u}_0-\frac{\kappa}{2}\n\theta_0,  \theta_0)(x)  &{\rm in}\ \Omega. 
\end{aligned}\right.
\end{equation}  
In contrast to the slip-boundary setting considered in \cite{JO2022,Sun},   the modified velocity satisfies only the homogeneous normal boundary condition 
$\mathbf{v}\cdot \mathbf{n}|_{\p\Omega}=0$, 
while no boundary condition is available for its vorticity. This lack of boundary control poses a major obstacle to establishing uniform {\it a priori} estimates.

{\it Step 1}. The low-order  estimates for $(q,\mathbf{v},\theta )$ can be shown by Newton-Leibniz's formula  for time variables,
\be\label{s0}   
\|q\|_{L_t^\infty(L^2)}^2+\|\mathbf{v}\|_{L_t^\infty(H^1)}^2+\|\theta\|_{L_t^\infty(H^2)}^2\leq C_0(M_0)+ t (\|q_t\|_{L_t^\infty(L^2)}^2+\|\mathbf{v}_t\|_{L_t^2(H^1)}^2+\|\theta_t\|_{L_t^2(H^2)}^2).
\ee

{\it Step 2}.  Due to the strong coupling between the equations for $(\mathbf{v}_t,\theta_t)$ and the lack of   boundary information, the estimates of $ \|\mathbf{v}_t\|_{L^2_t(H^1)}$ and $\|\theta_t\|_{L^2_t(H^2)}$ cannot be derived simultaneously from the energy estimates associated with $ \p_t\eqref{sim1}$. Motivated by the additional dissipation provided by $\kappa>0$, we shall  first derive the uniform estimate for $\|\theta_t\|_{L_t^2(H^2)}$ by combining the dissipative mechanism with suitable skew-symmetric coupling structure, and subsequently use it to estimate $ \|\mathbf{v}_t\|_{L^2_t(H^1)}$.

To this end, we introduce the   variables  (cf. \cite{Ala}) for the simplified system,
$$(\eta_t,\mathbf{u}_{\e t},\theta_t):=( \e q_t-\theta_t,\e\mathbf{u}_t,\theta_t)$$
 which satisfy 
\begin{equation}\label{sim5}\left\{\begin{aligned} 
&\eta_{tt} +\frac{1}{\e}\dv\mathbf{u}_{\e t}=0  &{\rm in}\  \Omega_t, \\
& \mathbf{u}_{\e tt} +\frac{1}{\e}\nabla (\eta_t+\theta_t)-\mu\Delta \mathbf{u}_{\e t} -(\mu+\lambda)\nabla\dv \mathbf{u}_{\e t}    =0  \q     &{\rm in}\  \Omega_t, \\
& \theta_{tt}+\frac{1}{\e}\dv\mathbf{u}_{\e t} -  \kappa  \Delta\theta_t  =0  &{\rm in}\  \Omega_t,\\
&\mathbf{u}_{\e t}=0,\q \partial_3\theta_t=0 &{\rm on}\  \p\Omega_t,\\
&(\eta_{t},\mathbf{u}_{\e t},\theta_{t})|_{t=0}=(\eta_{t}(0,x),\mathbf{u}_{\e t}(0,x),\theta_{t}(0,x)) &{\rm in}\  \Omega. 
\end{aligned}\right.\end{equation}    
Here,  $\theta_{t}(0,x):= \kappa \Delta\theta_0-\dv\mathbf{u}_0$,  
 while $\eta_{t}(0,x) $ and $\mathbf{u}_{\e t}(0,x)$ are defined analogously by taking $t=0$ in \ef{sim0}. In addition, $(\e\p_t)  (\eta_{t},\mathbf{u}_{\e t},\theta_{t})(0,x)$ can  be defined  inductively.    
 Due to \ef{rela} and \ef{data}, we know that $\|(\e\p_t)^k (\eta_{t} ,\mathbf{u}_{\e t},\theta_{t})(0,x)\|_{L^2} ^2\leq C_0(M_0)$ for $k=0,1$.  
Employing the skew-symmetric structure of   \eqref{sim5},  and the boundary conditions   $\ef{sim5}_4$, one 
tests $(\e\p_t )^k\ef{sim5}_{1,2,3}$ by  $(\e\p_t)^k (\eta_{t},\mathbf{u}_{\e t},\theta_{t})$, respectively, to derive 
\begin{equation}\label{s11}\begin{aligned}
& \left\|(\e\p_t )^k\left( \eta_t ,  \mathbf{u}_{\e t},  \theta_t   \right)\right\|^2_{L_t^\infty(L^2)}
+ \left\| (\e\p_t )^k\left(    \mathbf{u}_{\e t},     \theta_t  \right)  \right\|^2_{L^2_t(H^1)} \leq C_0(M_0) \q {\rm for}~k=0,1. 
\end{aligned}\end{equation} 
We next obtain the higher-order spatial estimates for 
$$\left\|\left( \n\eta_t ,  \dv\mathbf{u}_{\e t},  \n\theta_t   \right)\right\|^2_{L_t^\infty(L^2)}
+ \left\| \left(     \dv \mathbf{u}_{\e t},   \n \theta_t  \right)  \right\|^2_{L^2_t(H^1)}, $$ 
by treating tangential and normal derivatives separately, without invoking the Stokes regularity.  
 Indeed, the $\e$-uniform energy norm $\|\n\theta_t\|^2_{L_t^\infty(L^2)}+\|\n\theta_t\|^2_{L^2_t(H^1)}$ above is our primary object of interest, whereas the remaining weighted   quantities arise as by-products. 
Testing $\p_i\ef{sim5}_{1,2,3}$ by  $\p_i(\eta_t,\mathbf{u}_{\e t},\theta_t)$, respectively, we deduce the tangential estimates:
\begin{equation}\label{s12}
\left\|\left( \p_i\eta_t ,  \p_i\mathbf{u}_{\e t},  \p_i\theta_t   \right)\right\|^2_{L_t^\infty(L^2)}
+ \left\| \left(    \p_i\n \mathbf{u}_{\e t},   \p_i\dv \mathbf{u}_{\e t},  \p_i\n \theta_t  \right)  \right\|^2_{L^2_t(L^2)}\leq C_0(M_0).
\end{equation}  
Then, from $\ef{sim5}_{2}$, we have
\begin{equation}\label{sim6}
(2\mu+\lambda)\partial_3\dv \mathbf{u}_{\e t}  -
\mathbf{u}_{\e tt}^3 -\frac{1}{\e}\partial_3 (\eta_t+\theta_t) =\mu(\partial_3\dv \mathbf{u}_{\e t}-\Delta \mathbf{u}_{\e t} ^3) \q    {\rm in}\  \Omega_t, 
\end{equation}  
where the right-hand side term does not contain $\partial_3^2\mathbf{u}_{\e t}$ (cf. \cite{Valli1983,Valli1986}).  
Multiplying \ef{sim6} by $\partial_3\dv \mathbf{u}_{\e t}$, and $\p_3\ef{sim5}_{1,3}$ by   $\p_3(\eta_t, \theta_t)$, respectively, then summing them together, we obtain
\begin{equation}\label{s13}\begin{aligned}
&\left\|\left( \partial_3\eta_t ,  \partial_3\mathbf{u}_{\e t}^3,  \partial_3\theta_t   \right)\right\|^2_{L_t^\infty(L^2)}
+ \left\| \left(    \partial_3\dv \mathbf{u}_{\e t},   \partial_3\n \theta_t  \right)  \right\|^2_{L^2_t(L^2)}\\
& \leq C_0(M_0)  +\omt - \partial_3 \mathbf{u}_{\e tt}^3\cdot(\p_{1}\mathbf{u}_{\e t}^1+\p_{2}\mathbf{u}_{\e t}^2 ) +\mu (\partial_3\dv \mathbf{u}_{\e t}-\Delta \mathbf{u}_{\e t}^3 )\cdot \partial_3\dv \mathbf{u}_{\e t} dxds\\
&\leq C_0(M_0) +\delta( \| \mathbf{u}_{ t}\|_{L_t^2(H^1)}^2+ \|   \partial_3\dv \mathbf{u}_{\e t}  \|^2_{L^2_t(L^2)} )+C_\delta(\|\e \mathbf{u}_{\e tt}\|_{L_t^2(H^1)}^2 +\|  \p_i\n \mathbf{u}_{\e t}\|^2_{L^2_t(L^2)} ),
\end{aligned}    
\end{equation}   
where   integration by parts and $\mathbf{u}_{\e t}|_{\p\Omega_t}=0$ were used.  
Inserting \ef{s11} and \ef{s12} into \ef{s13}, we deduce that  
\begin{equation}\label{s15}
\left\|\left( \n\eta_t ,  \dv\mathbf{u}_{\e t},  \n\theta_t   \right)\right\|^2_{L_t^\infty(L^2)}
+ \left\| \left(     \dv \mathbf{u}_{\e t},   \n \theta_t  \right)  \right\|^2_{L^2_t(H^1)}\leq C_0(M_0)+\delta \| \mathbf{u}_{ t}\|_{L_t^2(H^1)}^2, 
\end{equation}   
where $\delta>0$ is a sufficiently small constant.

{\it Step 3}. To estimate $\|\mathbf{v}_t\|_{L^2_t(H^1)}$, the test function $\mathbf{v}_t $ cannot be employed in the equation $\p_t\eqref{sim1}_2$, because $\mathbf{v}_t $ lacks the boundary conditions required to eliminate the boundary term arising from the diffusion term $-\int_\Omega \Delta \mathbf{v}_t \cdot \mathbf{v}_t  dx$.  Note that 
$$\mathbf{v}_t |_{\p\Omega_t}=-\frac{\kappa}{2} \n_\tau \theta_t := -\frac{\kappa}{2} (\n  \theta_t- (\p_3\theta_t) \mathbf{n}),$$
since $ \p_{3}\theta_t|_{\p\Omega_t}=0.$  Then we construct a Stokes-type divergence-free extension $\ell$ of the boundary data $-\frac{\kappa}{2} \n_\tau \theta_t$ by solving the Stokes problem \ef{stokes} with $g= -\frac{\kappa}{2} \n_\tau \theta_t  $ and $\ell_0=0$.  
For an appropriate choice of the trace space, the existence and uniqueness of \ef{stokes} follow from Theorem 6.1 in \cite{Fursikov} (see Lemma \ref{lem2.4}).

 Then,  we introduce a new test function $Z:=\mathbf{v}_t-\ell$, which satisfies
$$Z|_{\p\Omega_t}=0,\q {\rm and}\q \dv Z=\dv\mathbf{v}_t\q {\rm in}~\Omega_t, $$ 
where the latter follows from  $\dv\ell=0$.
Multiplying $\p_t\eqref{sim1}_2 $ by $Z$, and  $\p_t\eqref{sim1}_1 $ by $q_t$ in $L^2(\Omega\times(0,t))$ yields that 
\begin{equation}\label{s8}
\begin{aligned}
&\frac{1}{2}\|Z\|_{L_t^\infty(L^2)}^2-\frac{1}{\e}\omt q_t \dv Z dxds+\mu\|\n Z\|_{L_t^2(L^2)}^2+\xi\|\dv Z\|_{L_t^2(L^2)}^2\\
&\leq C_0(M_0)+ \delta \| Z\|_{L_t^2(H^1)}^2 + C_\delta (\|\ell\|_{L_t^2(H^1)}^2+ \|\ell_t\|_{L_t^2(H^{-1})}^2+\|\theta_t\|_{L_t^2(H^2)}^2)  , 
\end{aligned} 
\end{equation}   
and
\be \label{s9}
\frac{1}{4}\|q_t\|_{L_t^\infty(L^2)}^2+\frac{1}{\e}\omt q_t \dv \mathbf{v}_t dxds\leq \frac{1}{4}\|q_t(0)\|_{L^2}^2.
\ee 
In view of Lemmas \ref{lem2.4} and \ref{lem2.5}, and applying the equation $\p_t\ef{sim1}_3$, one has 
\be\begin{aligned}\label{s10}
&\|\ell\|^2_{L_t^\infty(L^2)}+\|\ell\|^2_{L_t^2(H^1)}+\|\ell_t\|^2_{L_t^2(H^{-1})}  \\
&\leq C(\kappa, \Omega) [ \|\n_\tau \theta_t\|^2_{L^2_t(H^{ 1/2}(\partial\Omega))} +\|\n_\tau \theta_t\|^2_{H^{1/ 2}_t(H^{- 1 /2}(\partial\Omega))} ]\\
 &\leq C(\kappa, \Omega) [ \|  \theta_t\|^2_{L^2_t(H^2(\Omega))} +\|  \theta_t\|^2_{H^{1/ 2}_t(H^{1   }(\Omega))} ]\\
&\leq C(\kappa, \Omega) [ \|  \theta_t\|^2_{L^2_t(H^2 )} +\|  \theta_t\|_{H^{1}_t(L^{2} )}\|  \theta_t\|_{L^2_t(H^2)} ]\\
&\leq C(\kappa, \Omega) [ \|  \theta_t\|^2_{L^2_t(H^2 )} + \|  \theta_t\|_{L^2_t(H^2)} (\|  \theta_t\|_{L^2_t(L^2)} +\|  \theta_{tt}\|_{L^2_t(L^2)} ) ]\\
&\leq \delta \|  \theta_{tt}\|_{L^2_t(L^2)} ^2+C_\delta C(\kappa, \Omega)\|  \theta_t\|^2_{L^2_t(H^2)}\\
&\leq \delta \|  \dv\mathbf{v}_t\|_{L^2_t(L^2)} ^2+C_\delta C(\kappa, \Omega)\|  \theta_t\|^2_{L^2_t(H^2)},
\end{aligned}   
\ee    
which is a key technical step in controlling $\|\mathbf{u}_t\|_{L^2_t(H^1)}$ by means of $ \|\theta_t\|_{L^2_t(H^2)} $.  
Adding   \ef{s8} and \ef{s9}, then applying  \ef{s10},  we have   
\be\nonumber
\|Z\|_{L_t^\infty(L^2)}^2+\|q_t\|_{L_t^\infty(L^2)}^2+\|Z\|_{L_t^2(H^1)}^2 \leq  C_0(M_0)+ \delta \|  \dv\mathbf{v}_t\|_{L^2_t(L^2)}^2 +C_\delta  \|  \theta_t\|^2_{L^2_t(H^2)}. 
\ee   
By choosing above $\delta$ sufficiently small, and  applying \ef{s11}, \ef{s15} and \ef{s10}, we obtain 
\be  \begin{aligned}\label{s10.5}
\|\mathbf{v}_t\|_{L_t^\infty(L^2) }^2 +\|q_t\|_{L_t^\infty(L^2)}^2+\|\mathbf{v}_t\|_{  {L_t^2(H^1)}}^2 +\|\mathbf{u}_t\|_{L_t^\infty(L^2) }^2 +\|\mathbf{u}_t\|_{  {L_t^2(H^1)}}^2   \leq  C_0(M_0) +C_2 \|  \theta_t\|^2_{L^2_t(H^2)},
\end{aligned}   
\ee  
since $\mathbf{u}_t=\mathbf{v}_t+\frac{\kappa}{2}\n\theta_t=Z+\ell +\frac{\kappa}{2}\n\theta_t$.
Thus,  choosing $\delta$ in \ef{s15} sufficiently small such that $\delta C_2<1$, we derive the uniform bound  from \ef{s15} and  \ef{s10.5},  
\be\label{s10.6}
\|  \theta_t\|^2_{L^\infty_t(H^1)}+\|  \theta_t\|^2_{L^2_t(H^2)}+ \|\mathbf{u}_t\|_{L_t^\infty(L^2) }^2 +\|\mathbf{u}_t\|_{  {L_t^2(H^1)}}^2   \leq  C_0(M_0).  
\ee

Combining \ef{s0} and  \ef{s10.6}, we conclude  the estimates in {\it Step 1}, {\it Step 2},  and {\it Step 3} as follows:
\be\begin{aligned}\label{low}
&\|q\|_{L_t^\infty(L^2)}^2+\|\mathbf{v}\|_{L_t^\infty(H^1)}^2+\|\theta\|_{L_t^\infty(H^2)}^2 +\|  \theta_t\|^2_{L^\infty_t(H^1)}+\|  \theta_t\|^2_{L^2_t(H^2)} \\
& + \|q_t\|_{L_t^\infty(L^2)}^2+\|\mathbf{u}_t\|_{L_t^\infty(L^2) }^2 +\|\mathbf{u}_t\|_{  {L_t^2(H^1)}}^2 \leq C_0(M_0),
\end{aligned}\ee
by choosing  $t<1$.

{\it Step 4}. The final goal is to establish the estimates for higher-order spatial derivatives. 
We regard the momentum equation  $\ef{sim0}_2$ under the no-slip boundary condition as a Stokes problem: 
\begin{equation}\label{sim2}\left\{\begin{aligned} 
& -\mu\Delta \mathbf{u}  +\frac{1}{\e}\nabla q=- \mathbf{u}_t+(\mu+\lambda )\nabla\dv \mathbf{u}  &{\rm in}\  \Omega_t, \\
& \dv\mathbf{u}=\dv\mathbf{u}   &{\rm in}\  \Omega_t, \\
& \mathbf{u} =0 &{\rm on}\  \p\Omega_t.
\end{aligned}\right.\end{equation} 
Stokes regularity yields the following estimate:
$$\|\mathbf{u}\|_{L^2_t(H^3)}^2+\left\|\frac{\n q}{\e}\right \|_{L^2_t(H^1)}^2\leq C( \|\n^2\dv \mathbf{u}\|_{L^2_t(L^2)}^2+\|   \mathbf{u}_t\|_{L^2_t(H^1)}^2+\| \mathbf{u}\|_{L^2_t(L^2)}^2).$$
Thus, it remains to estimate $\|\n^2\dv \mathbf{u}\|_{L^2_t(L^2)}^2$.
In view of the equivalence of systems \eqref{sim0} and \eqref{sim1}, we work with \eqref{sim1} to derive an estimate for
$\|\n^2\dv \mathbf{v}\|_{L^2_t(L^2)}^2$
which allows us to exploit the skew-symmetric structure of the system. Due to the higher-order coupling term $-\beta \n\Delta \theta$ in the modified velocity equation $\eqref{sim1}_2$, the estimate for $\|\n^2\Delta \theta\|_{L^2_t(L^2)}^2$ is derived simultaneously.

We next outline the strategy for estimating $\|\n^2\dv \mathbf{v}\|_{L^2_t(L^2)}^2$ and $\|\n^2\Delta \theta\|_{L^2_t(L^2)}^2$. 
The main difficulty  lies in estimating the second-order tangential derivatives since 
the viscous contribution 
$-\mu\om \p_i^2\Delta \mathbf{v}\cdot \p_i^2 \mathbf{v} dx$ cannot be handled directly via integration by parts. 
This is in sharp contrast to the zero heat conduction case \cite{JO} and the small temperature variations setting \cite{ROzamp}, where the tangential estimates are straightforward.

\textbf{Tangential estimates}: Multiplying  $\p_i^2\ef{sim1}_2$ by $\frac{1}{2}\p_i^2\mathbf{u}$  and using $\p_i^2\mathbf{u}|_{\p\Omega}=0$ yields that 
 \begin{equation}   \nonumber
\begin{aligned} 
& \frac{1}{4}\frac{d}{dt}\om |\p_i^2\mathbf{u}|^2 dx - \frac{1}{2\e} \om \p_i^2q \cdot \p_i^2\dv \mathbf{u} dx +\frac{\mu}{2} \om \p_i^2\n \mathbf{u}\cdot  \p_i^2\n \mathbf{v} dx\\
& +\frac{\xi}{2} \om \p_i^2\dv \mathbf{u}\cdot  \p_i^2\dv \mathbf{v} dx  + \frac{\beta}{2} \om \p_i^2\Delta \theta\cdot \p_i^2 \dv\mathbf{u} dx \\
&= - \frac{\kappa}{4}\om \p_i^2\theta_t \cdot \p_i^2 \dv\mathbf{u} dx.  
\end{aligned} 
\end{equation}  
Indeed, the mixed viscosity terms can be transformed into 
 \begin{equation} 
\nonumber 
\begin{aligned} 
\frac{\xi}{2} \om \p_i^2\dv \mathbf{u}\cdot  \p_i^2\dv \mathbf{v} dx = \frac{\xi}{2} \om | \p_i^2\dv \mathbf{v}|^2 dx+\frac{\xi\kappa}{4} \om \p_i^2\Delta \theta \cdot  \p_i^2\dv \mathbf{u} dx-\frac{\xi\kappa^2}{8} \om |\p_i^2\Delta \theta |^2 dx,
\end{aligned}  
\end{equation}  
and
 \begin{equation} 
\nonumber 
\begin{aligned} 
 \frac{\mu}{2} \om \p_i^2\n \mathbf{u}\cdot  \p_i^2\n \mathbf{v} dx  &=\frac{\mu}{4} \om |\p_i^2\n \mathbf{u}|^2 dx+\frac{\mu}{4}\bigg( \om |\p_i^2\n \mathbf{u}|^2 dx-   {\kappa} \om \p_i^2\n \mathbf{u}\cdot  \p_i^2\n ^2\theta dx \\
&\q +\frac{\kappa^2}{4}    \om  | \p_i^2\n ^2\theta|^2  dx   \bigg)- \frac{\mu\kappa^2}{16}    \om  | \p_i^2\n ^2\theta|^2  dx \\
&=\frac{\mu}{4} \om |\p_i^2\n \mathbf{u}|^2 dx+\frac{\mu}{4} \om |\p_i^2\n \mathbf{v}|^2 dx-\frac{\mu\kappa^2}{16}   \om  | \p_i^2\Delta\theta|^2  dx. 
\end{aligned}  
\end{equation}  
Here, we apply the fact that 
 \begin{equation} 
\nonumber 
\begin{aligned} 
-\om  | \p_i^2\n ^2\theta|^2  dx
&=-\sum_{j=1}^3\sum_{k=1}^3\om   \p_i^2 \p_j\p_k \theta : \p_i^2\p_j\p_k\theta  dx\\
&= \sum_{j=1}^3\sum_{k=1}^3\om   \p_i^2 \p_j^2\p_k \theta \cdot \p_i^2 \p_k\theta  dx
- \sum_{k=1}^3\int_{\p\Omega}  \p_i^2 \p_{3}\p_k \theta \cdot \p_i^2 \p_k\theta   dS\\
&=- \sum_{j=1}^3\sum_{k=1}^3\om   \p_i^2 \p_j^2  \theta \cdot \p_i^2 \p_k^2\theta  dx + \sum_{j=1}^3 \int_{\p\Omega}  \p_i^2 \p_j^2  \theta  \cdot \p_i^2 \p_{3}\theta   dS\\
&=-\om  | \p_i^2 \Delta  \theta |^2 dx,\q \q i=1,2,
\end{aligned}  
\end{equation}  
where the first boundary term vanishes since $\p_i^2 \p_{3}\p_k \theta |_{\p\Omega_t}=0$ if $k=1,2$, and $\p_i^2  \p_k \theta |_{\p\Omega_t}=0$ if $ k=3$.
Collecting the above equalities, we obtain  
 \begin{equation}  
\label{s1}
\begin{aligned} 
& \frac{1}{4}\frac{d}{dt}\om |\p_i^2\mathbf{u}|^2 dx - \frac{1}{2\e} \om \p_i^2q \cdot \p_i^2\dv \mathbf{u} dx +\frac{\mu}{4} \om |\p_i^2\n \mathbf{u}|^2 dx+\frac{\mu}{4} \om |\p_i^2\n \mathbf{v}|^2 dx \\
& +\frac{\xi}{2} \om | \p_i^2\dv \mathbf{v}|^2 dx+\bigg(\frac{\beta}{2} +\frac{\xi\kappa}{4} \bigg)\om \p_i^2\Delta \theta \cdot  \p_i^2\dv \mathbf{u} dx-\bigg(\frac{\xi\kappa^2}{8}+\frac{\mu\kappa^2}{16} \bigg) \om |\p_i^2\Delta \theta |^2 dx\\
&\leq \delta \|\p_i^2\dv \mathbf{u}\|_{L^2}^2+ C_\delta \|\theta_t\|_{H^2}^2. 
\end{aligned}   
\end{equation}   
Then,  multiplying  $\p_i^2\ef{sim1}_1$ by $\p_i^2 q$, and  $\p_i^2\ef{sim1}_3$ by $-\p_i^2 \Delta \theta$ gives 
 \begin{equation} 
\label{s2}
 \frac{1}{4}\frac{d}{dt}\om |\p_i^2 q|^2 dx + \frac{1}{\e} \om \p_i^2q \cdot \p_i^2\dv \mathbf{v} dx =0, 
\end{equation}  
and 
\begin{equation} 
\label{s3}
  \frac{1}{2}\frac{d}{dt}\om |\p_i^2\n\theta|^2 dx  +\kappa \om |\p_i^2\Delta \theta|^2  dx -  \om \p_i^2\dv \mathbf{u} \cdot \p_i^2 \Delta \theta dx =0,
\end{equation}   
respectively. To eliminate the singular terms in \ef{s1} and \ef{s2}, we take the inner product of $\p_i^2\ef{sim1}_3$ and $-\frac{1}{2\e}  \p_i^2 q$ in $\Omega$ to deduce 
\begin{equation} 
\label{s4}
  \frac{\kappa}{2\e}   \om \p_i^2\Delta \theta\cdot\p_i^2 q   dx -   \frac{1}{2\e}  \om \p_i^2\dv \mathbf{u} \cdot \p_i^2 q dx \leq \delta \left\|\frac{\n q}{\e}\right\|_{ H^1}^2+C_\delta \|\theta_t\|_{H^2}^2.
\end{equation}  
By computing $\ef{s1} +\ef{s2}+(\frac{\xi\kappa}{4}+\frac{\beta}{2})\times\ef{s3}+\ef{s4}$, we  remove both  the singular terms  and the high-order terms on the left-hand side to  conclude that 
\begin{equation} 
\label{s5}
\begin{aligned} 
& \frac{1}{4}\frac{d}{dt}\om |\p_i^2\mathbf{u}|^2+|\p_i^2 q|^2+ (\frac{\xi\kappa}{2}+ {\beta} )|\p_i^2\n\theta|^2 dx +\frac{\mu}{4} \om |\p_i^2\n \mathbf{u}|^2 +|\p_i^2\n \mathbf{v}|^2 dx \\
&+\frac{\xi}{2} \om  | \p_i^2\dv \mathbf{v}|^2 dx  +\frac{\kappa}{2}\left( \frac{\xi\kappa}{4}+ \beta  - \frac{\mu\kappa}{8} \right)\om |\p_i^2\Delta\theta|^2dx\\
&\leq \delta ( \|\p_i^2\dv \mathbf{u}\|_{L^2}^2+\left\|\frac{\n q}{\e}\right\|_{ H^1}^2) + C_\delta \|\theta_t\|_{H^2}^2.
\end{aligned} 
\end{equation}

\textbf{Mixed tangential-normal estimates}: 
Note that $\cl\cl \mathbf{v}=\cl\cl \mathbf{u}$ in $\Omega_t$, and  the term ``$\cl\cl \mathbf{v}\cdot \mathbf{n}$''   does not involve the second-order normal derivative $\partial_3^2 \mathbf{v}$ (cf. \cite{Valli1983,Valli1986}). 
 These observations are crucial for estimating the mixed tangential-normal and the second-order normal derivatives of $\dv\mathbf{v}$.
We derive the following equation from  $\ef{sim1}_2$:
\be\label{sim3}
(\mu+\xi)\partial_3 \dv \mathbf{v}-\frac{1}{\e}\partial_3 q+ \beta\partial_3\Delta\theta =\mathbf{v}_t ^3+\mu(\partial_3 \dv \mathbf{u}-\Delta \mathbf{u}^3 )\q {\rm in}\  \Omega_t,
\ee  
where $ \mathbf{u}^3$ is the third component of $\mathbf{u}$ and $\mathbf{v}_t ^3=\p_t (\mathbf{u} ^3-\frac{\kappa}{2}\partial_3\theta)$. 
Meanwhile, we use the equation $\p_3\ef{sim1}_3 $ to verify the  estimates of mixed tangential-normal and   second-order normal derivatives  for $\Delta \theta$,    
\be\label{sim4}   
 \partial_3 \theta_t -\frac{\kappa}{2}\partial_3\Delta \theta+ \partial_3  \dv \mathbf{v}=0 \q {\rm in}\  \Omega_t. 
\ee    
Exploiting the balanced structure of $\partial_3\ef{sim1}_1$, \ef{sim3}, and \ef{sim4}, we establish the mixed tangential-normal estimates by integrating the resulting equalities of $\partial_i\partial_3\ef{sim1}_1\cdot \partial_i\partial_3q$, $\partial_i \ef{sim3}\cdot \partial_i\partial_3\dv \mathbf{v}$, and $-\beta \partial_i\ef{sim4}\cdot \partial_i\partial_3\Delta \theta $ in $\Omega$, respectively, 
$$
\frac{1}{4}\frac{d}{dt}\om  |\p_i\partial_3q|^2 dx+\frac{1}{\e}\om\p_i \partial_3q \cdot \p_i\partial_3\dv \mathbf{v} dx=0,
$$  
\begin{equation}  
\nonumber 
\begin{aligned} 
&(\mu+\xi) \om |\p_i\partial_3\dv \mathbf{v}|^2  dx  -\frac{1}{\e}\om\p_i \partial_3q \cdot \p_i\partial_3\dv \mathbf{v} dx+\beta\om \p_i\partial_3\Delta\theta\cdot \p_i\partial_3\dv \mathbf{v} dx\\
&\leq \delta\|\p_i\partial_3\dv \mathbf{v}\|^2_{L^2}+C_\delta(\|\mathbf{v}_t\|_{H^1}^2+ \|\p_i^2\n\mathbf{u}\|^2_{L^2}),
\end{aligned} 
\end{equation}  
and     
\begin{equation}\label{s6}
\frac{\beta}{2}\frac{d}{dt}\om |\p_i\partial_3\n\theta|^2 dx +\frac{\beta\kappa}{2}\om |\p_i\partial_3\Delta\theta|^2dx-\beta\om \p_i\partial_3\Delta\theta\cdot \p_i\partial_3\dv \mathbf{v} dx=0,
\end{equation}
where the  singular terms  and the high-order terms can be eliminated directly.

\textbf{Normal estimates}: Integrating the equalities of
$\partial_3^2\ef{sim1}_1\cdot  \partial_3^2q$, $\partial_3 \ef{sim3}\cdot  \partial_3^2 \dv \mathbf{v}$, and $-\beta \partial_3\ef{sim4}\cdot  \partial_3^2\Delta \theta $ over $\Omega$, and adding them together  
yields the second-order normal estimates:
\begin{equation} \label{s7}\begin{aligned} 
&\frac{1}{4}\frac{d}{dt}\om  | \partial_3^2 q|^2 dx+(\mu+\xi) \om | \partial_3^2 \dv \mathbf{v}|^2  dx+ \frac{\beta\kappa}{2}\om | \partial_3^2\Delta\theta|^2dx\\
&\leq C(\|\mathbf{v}_t\|_{H^1}^2+\|\theta_t\|_{H^2}^2+\|\p_i^2\n\mathbf{u}\|^2_{L^2}+\|\partial_3^2\p_i\mathbf{u}\|_{L^2}^2). 
\end{aligned} 
\end{equation} 
Due to $\partial_3^2\p_i\mathbf{u}=\Delta\p_i\mathbf{u}-\sum_{j=1}^2 \p_j^2\p_i\mathbf{u}$, and the regularity theory for the auxiliary Stokes problem $\p_i\ef{sim2}$ for $\p_i\mathbf{u}$, one  gets 
\begin{equation} \nonumber\begin{aligned}
\|\partial_3^2\p_i\mathbf{u}\|_{L^2}^2
&\leq C(\|\Delta\p_i\mathbf{u}\|_{L^2}^2+\| \p_j^2\p_i\mathbf{u}\|_{L^2}^2 )\\
&\leq C(\|\p_i\dv\mathbf{u}\|_{H^1}^2+ \|\mathbf{u}_t\|_{H^1}^2+\| \p_i^2\n\mathbf{u}\|_{L^2}^2).
\end{aligned}  
 \end{equation}

To close the preceding energy estimates, the coefficients of the
dissipative terms on the left-hand sides of \ef{s5}, \ef{s6}, and \ef{s7}
must be positive.  More precisely, we require
$$ \beta>0,\q\xi>0,\q {\rm and }\q  \frac{\xi\kappa}{4}+ \beta  - \frac{\mu\kappa}{8} >0.$$ 
Indeed, under the physical assumptions
$\mu>0$ and $2\mu+3\lambda\geq0$, we have
$\xi\geq \frac{\mu}{3}+\frac{\kappa}{2}>0$
and
$\frac{\xi\kappa}{4}+\beta-\frac{\mu\kappa}{8}
=\frac{\kappa}{8}(9\mu+6\lambda-\kappa).$
Consequently, the above positivity conditions reduce to $0<\kappa<4\mu+2\lambda,$
i.e., the condition $\beta>0$ in Theorem \ref{uethm}.

Adding all estimates in this step and choosing $\delta>0$   sufficiently small,  we conclude that  
\be\label{s7.5}\begin{aligned}
&\|q\|_{L^\infty_t(H^2)}^2+\|\mathbf{u}\|_{L^2_t(H^3)}^2+\left\|\frac{\n q}{\e}\right \|_{L^2_t(H^1)}^2+\|  \theta\|_{L^2_t(H^4)}^2\\
&\leq C_0(M_0)+C( \|\theta_t\|_{L^2_t(H^2)}^2+ \|\mathbf{u}_t\|_{L^2_t(H^1)}^2).
\end{aligned} \ee

In summary, combining the low-order estimate \ef{low}   and the high-order spatial  estimate  \ef{s7.5}, and choosing $\delta>0$ sufficiently small, we obtain the following uniform estimates for the simplified system \ef{sim0}, 
\begin{equation}\begin{aligned}
&\|q\|_{L_t^\infty(H^2)}^2+\|\mathbf{u}\|_{L_t^\infty(H^1)\cap L_t^2(H^3)}^2+\|\theta\|_{L_t^\infty(H^2)\cap L_t^2(H^4)}^2+\left\|\frac{\n q}{\e}\right \|_{L^2_t(H^1)}^2
\\   
&+\|\mathbf{u}_t\|_{L_t^\infty(L^2)\cap {L_t^2(H^1)}}^2 +\|\theta_t\|_{L_t^\infty(H^1)\cap {L_t^2(H^2)}}^2 +\|\theta_{tt}\|_{ {L_t^2(L^2)}}^2 \leq C_0(M_0).
\end{aligned}\end{equation}

\section{Preliminaries}
\subsection{Some basic inequalities}
We first present some  useful  interpolation inequalities and trace inequalities.
\begin{lem}\label{interpo}{\rm (Interpolation inequality,  see
Friedman \cite{Frie},  Part 1,  Theorem 10.1)} Let $\Omega\subset {\mathbb R}^N$
be a bounded domain with a $C^\kappa$-boundary and  $\mathbf{u}$ be any function
in $W^{\kappa, r}(\Omega)\cap L^q(\Omega)$ with $1\le r, q\le\infty$. For
any integer $j$ with $0\le j<\kappa$,  and for any number $a$ in the
interval $[j/\kappa,  1]$,  set
$\frac{1}{p}=\frac{j}{N}+a(\frac{1}{r}-\frac{\kappa}{N})+(1-a)\frac{1}{q}.$
 If $\kappa-j-N/r$ is not a nonnegative integer,  then
\begin{equation}\label{inter1}
\|D^j \mathbf{u}\|_{L^p(\Omega)}\le C
\|\mathbf{u}\|_{W^{\kappa, r}(\Omega)}^a\|\mathbf{u}\|_{L^q(\Omega)}^{1-a}.
\end{equation}
 If $\kappa-j-N/r$ is a nonnegative integer,  then \eqref{inter1} only holds for $a=j/\kappa$.
 The constant $C$ depends only on $\Omega, r, q, \kappa, j, a$.
\end{lem}

We record the following interpolation inequalities for later use.
\begin{lem}\label{inter} Let $\Omega$ be a bounded,  open domain of $\mathbb{R}^3$  with a $C^3$ boundary. Then the following inequalities
are satisfied:
$$
\|\mathbf{u}\|_{L^3}\leq C\|\mathbf{u}\|_{H^1}^{\frac{1}{2}}\|\mathbf{u}\|_{L^2}^{\frac{1}{2}}, \quad\|\nabla^2\mathbf{u}\|_{L^2}\leq C\|\mathbf{u}\|_{H^3}^{\frac{1}{2}}\|\nabla
\mathbf{u}\|_{L^2}^{\frac{1}{2}}, 
$$
$$
\|\nabla^2\mathbf{u}\|_{L^3}\leq C\|\mathbf{u}\|_{H^3}^{\frac{3}{4}}\|\mathbf{u}\|_{H^1}^{\frac{1}{4}}, \q \|\nabla^2\mathbf{u}\|_{L^4}\leq
C\|\mathbf{u}\|_{H^3}^{\frac{7}{8}}\|\mathbf{u}\|_{H^1}^{\frac{1}{8}}, \q 
\|\nabla \mathbf{u}\|_{L^\infty}\leq  C\|\mathbf{u}\|_{H^3}^{\frac{7}{8}}\|\mathbf{u}\|_{H^1}^{\frac{1}{8}}.
$$
Moreover,  by the Sobolev embedding theorem,  we have 
$$
\|\mathbf{u}\|_{L^\infty} \leq C\|\mathbf{u}\|_{H^2}^{\frac{3}{4}}\|\mathbf{u}\|_{L^2}^{\frac{1}{4}}\leq C\|\mathbf{u}\|_{H^2}^{\frac{7}{8}}\|\mathbf{u}\|_{L^2}^{\frac{1}{8}}, 
$$
$$
\|\mathbf{u}\|_{L^2_t(L^3)}\leq Ct^{\frac1 4}\|\mathbf{u}\|_{L^\infty_t(L^2)}^{\frac1 2}\|\mathbf{u}\|_{L^2_t(H^1)}^{\frac1 2},  \q 
\|\mathbf{u}\|_{L^2_t(L^\infty)}\leq Ct^{\frac1 8}\|\mathbf{u}\|_{L^2_t(H^2)}^{\frac3 4}\|\mathbf{u}\|_{L^\infty_t(H^1)}^{\frac1 4}.
$$
\end{lem}

\begin{lem}\label{lem1} {\rm (see Bourguignon and Brezis \cite{Brezis})} Let $\Omega$ be a bounded domain in $\mathbb{R}^N$ with smooth boundary $\partial\Omega$ and its unit outward normal $\mathbf{n}$. Then there exists a constant $C>0$ independent of $\mathbf{u}$, such that
$$
\Vert\mathbf{u}\Vert_{H^s(\Omega)}\leq C(\Vert{\rm div}\mathbf{u}\Vert_{H^{s-1}(\Omega)}+\Vert{\rm curl}\mathbf{u}\Vert_{H^{s-1}(\Omega)}
+\Vert\mathbf{u}\cdot\mathbf{n}\Vert_{H^{s-\frac{1}{2}}(\partial\Omega)}
+\Vert\mathbf{u}\Vert_{H^{s-1}(\Omega)}) 
$$
for any $\mathbf{u}\in [H^s(\Omega)]^N$, $s\geq1$.
\end{lem}

Next, we consider the nonhomogeneous boundary value problem for the Stokes equations:
\begin{equation}\label{stokes} 
\left\{\begin{aligned}
&\p_t \ell - \Delta \ell+\n\pi_\ell=0 &{\rm in}\ \Omega_t, \\
&\dv \ell=0    &  {\rm in}\ \Omega_t, \\
&\ell = g \q   &{\rm on}\ \partial \Omega_t, \\
&\ell |_{t=0}= \ell_0   &{\rm in}\ \Omega.
\end{aligned}\right.\end{equation}

 \begin{lem}\label{lem2.4}{\rm (see Fursikov et al. \cite{Fursikov}, Theorem  6.1 with $s=1$)}
Suppose that 
$$ g\in L^2(0,t;H^{1/2}(\partial\Omega)) \cap H^{1/2}(0,t;H^{-1/2}(\partial\Omega)) $$
and  $g\cdot \mathbf{n}=0$ on the boundary $\partial \Omega_t$. Suppose further that $\ell_0\in L^2$ and  the compatibility condition $\ell_0\cdot \mathbf{n}|_{\p\Omega}=0$.
 There exists a unique 
solution $(\ell,\n \pi_\ell)\in \mathcal{V}^1(\Omega_t)\times  L^2(0,t;H^{-1}(\Omega))$ for the problem \ef{stokes} with $\mathcal{V}^1(\Omega_t):=\{ v\in L^2(0,t;H^{1}(\Omega)): v_t\in L^2(0,t;H^{-1}(\Omega)) ,\dv v=0\}$. Moreover, the solution satisfies that
\be\begin{aligned}\nonumber
&\|\ell\|^2_{L_t^\infty(L^2)}+\|\ell\|^2_{L_t^2(H^1)}+\|\ell_t\|^2_{L_t^2(H^{-1})} +\|\n\pi_\ell\|^2_{L_t^2(H^{-1})} \\
&\leq C(\Omega) \left[\|\ell_0\|^2_{ L^2} +\|g\|^2_{L^2_t(H^{ 1/2}(\partial\Omega))} +\|g\|^2_{H^{1/ 2}_t(H^{- 1 /2}(\partial\Omega))} \right].
\end{aligned}\ee  
\end{lem}

In addition, we reproduce the following  
mixed derivative estimate.  
\begin{lem} {\rm (see Boulakia et al.   \cite{BGT}, Lemma  A.5)}\label{lem2.5} 
Let $\Omega\subset\mathbb R^{3}$ be a bounded domain with smooth boundary. 
Let $s \in [0,1]$, $m_1, m_2 \in [0,\infty)$ and $m := sm_1 + (1-s)m_2$. Then, there exists a constant $C$ independent of $T$ such that
\be 
\|f\|_{H^s(0,T;H^m(\Omega))} \leqslant C \|f\|_{H^1(0,T;H^{m_1}(\Omega))}^s \|f\|_{L^2(0,T;H^{m_2}(\Omega))}^{1-s}  
\ee  
for all $f \in H^1(0,T;H^{m_1}(\Omega)) \cap L^2(0,T;H^{m_2}(\Omega))$. 
\end{lem}

\subsection{Local coordinates} 
Now we introduce the local coordinates to derive   estimates near the boundary.  We choose as local coordinates 
 the isothermal coordinates $\lambda(\psi, \phi)$  (see \cite{Valli1983, Valli1986}),  where
$$\lambda_\psi\cdot\lambda_\psi>0, \q \lambda_\phi\cdot\lambda_\phi>0,  \q \lambda_\psi\cdot\lambda_\phi=0.
$$   
We cover the boundary $\partial\Omega$ by a finite number of bounded open sets $W^k\subset\mathbb{R}^3, k=1, 2, \cdots, L, $ such that  any $x\in W^k\cap\Omega$ can be written as 
$$x=\Lambda^k(\psi, \phi, r)\equiv\lambda^k(\psi, \phi) +rn(\lambda^k(\psi, \phi)),  $$  
where $\lambda^k(\psi, \phi)$ is the isothermal coordinates of $  W^k\cap   \p\Omega$  and $n$ is the unit outer normal to $\partial\Omega$.  Below,  we will omit the superscript $k$ in each $W^k$ for simplicity. Employing  the properties of isothermal coordinates, 
we construct the orthonormal system:  
$$
e_1:=\frac{\lambda_\psi}{|\lambda_\psi|}, \quad e_2:=\frac{\lambda_\phi}{|\lambda_\phi|}, \quad e_3:=e_1\times e_2\equiv n(\lambda). 
$$          
Direct calculation shows that $J\in C^2$ and 
\begin{equation*}\begin{aligned}
J & :=det\frac{\partial x}{\partial ( \psi, \phi, r)}=(\Lambda_\psi\times\Lambda_\phi, r )\cdot e_3\\
&=|\lambda_\psi \|\lambda_\phi|+r(|\lambda_\psi |n_\phi\cdot e_2+|\lambda_\phi |n_\psi\cdot e_1 )+r^2[(n_\psi\cdot e_1)(n_\phi\cdot e_2 )-(n_\psi\cdot e_2)(n_\phi\cdot e_1 )]>0, 
\end{aligned}\end{equation*}
for sufficiently small $r>0$.
Moreover,  we can easily derive the following relations (see  \cite{Valli1983}): 
\begin{equation}\begin{aligned}\nonumber
[\nabla(\Lambda^{-1})^1]\circ\Lambda&=(1/J)(\Lambda_\phi\wedge e_3), \\
[\nabla(\Lambda^{-1})^2]\circ\Lambda&=(1/J)(e_3\wedge \Lambda_\psi), \\
[\nabla(\Lambda^{-1})^3]\circ\Lambda&=(1/J)(\Lambda_\phi\wedge \Lambda_\psi),
\end{aligned}\end{equation} 
where the notation ``$\circ$'' denotes composition of maps. Set $y:=(\psi, \phi, r)$,  $a_{ij}=((Jac\Lambda)^{-1})_{ij}=(\frac{\partial y}{\partial x})_{ij}$. Then,  the normal direction $n=(a_{31}, a_{32}, a_{33})$,  the tangential directions $\tau_i=(a_{i1}, a_{i2}, a_{i3})$, $ (i=1, 2), $ and 
\begin{equation}\begin{aligned}\label{a3j}
a_{ij}a_{3j}=0,  \quad\rm{for}~i=1, 2.
\end{aligned}\end{equation}
Let   $D_i$  denote differentiation in the local coordinates. To be precise,  $D_3$ is the normal derivative and $D_i$ for $i=1, 2$ are the tangential derivatives in the  local coordinates.  
Moreover,  we have
$$
\partial_{x_j}=\sum_{k=1}^{3}a_{kj}D_k.
$$

Denote $\widetilde{\Omega}:=\Lambda^{-1}(W\cap\Omega) $ and  
$U (t, y):=\mathbf{u} (t, \Lambda(y))$.
The following properties hold.
\begin{prop}[see \cite{Valli1983}]
Let $\chi \in C_0^\infty (\Lambda^{-1}(W))$, then
$D_i (Ja_{ij})=0$, for $ j=1, 2, 3$; $\chi D_\tau U =0 $; $\chi D_\tau D_\xi U=0 $ on $\partial\widetilde{\Omega}$ in the tangential directions $\tau, \xi=1, 2$.  
\end{prop}

The relations between the different coordinates are 
$$
\|D_y U\|_{L^p(\widetilde{\Omega})}\leq C\|\nabla_x \mathbf{u}\|_{L^p(\Omega)}, \quad \|D_y^2U\|_{L^p(\widetilde{\Omega})}\leq C\|\nabla_x \mathbf{u}\|_{W^{1, p}(\Omega)}  \quad{\rm for}~ 1\leq p\leq \infty. $$
These inequalities also hold for the other variables expressed in local coordinates.

\section{Uniform  estimates in general bounded domains}\label{ue}

In this section, suppose $(\rho,\mathbf{v},q,\theta)$ is a smooth solution to the problem \ef{ns0} on $[0,T]$. We assume that $\Theta=1+\theta\geq \frac{A}{2}$ in $\overline{\Omega}\times[0,T]$ where $A>0$ is the initial temperature lower bound in Theorem \ref{uethm}.  
For $t\in[0,T]$ , we define the auxiliary  energy functional: 
 \be\nonumber\begin{aligned}
M_1  (t)&:=    \| ( \eta_t  ,  \mathbf{u}_{\e t} ,  \vartheta_t     ) \|^2_{L^\infty (0, t;L^2(\Omega))}
+  \|  (   \n \mathbf{u}_{\e t} ,    \n \vartheta_t   )  \|^2_{L^2 (0, t;L^2(\Omega))} \\
&\q +
\| ( \n \eta_t ,  \e\dv \mathbf{u}_t , \nabla\vartheta_t  )\|^2_{L^\infty (0, t;L^2(\Omega))}+\|(\e  \dv \mathbf{u}_t ,  \nabla\vartheta_t  )\|^2_{L^2 (0, t;H^1(\Omega))},
\end{aligned}\ee  
\be\nonumber\begin{aligned}
M_2  (t)&:= \|(\e^2  \sigma_{tt} ,  \e^2 \mathbf{u}_{tt}  , \e \vartheta_{tt}  )\|^2_{L^\infty (0, t;L^2(\Omega))}+\| (\e^2    \mathbf{u}_{tt} , \e\vartheta_{tt}  ) \|^2_{L^2 (0, t;H^1(\Omega))}    \\ 
&\q  + \| \e^2 \mathbf{u}_{t}   \|^2_{L^\infty (0, t;H^1(\Omega))}+ \| \e^2 \mathbf{u}_{t}   \|^2_{L^2 (0, t;H^2 (\Omega))},
\end{aligned}\ee  
and the total energy:
$$\mathcal{E}(t):=M  (t)+M_1  (t)+M_2  (t). $$     
Due to the relation  \ef{rela} and the initial condition \ef{data}, we know that  $\mathcal{E}(0)\leq  C_0 (M_0)$ where $C_0$ is a positive non-decreasing and continuous function. 
\begin{rmk}  \label{ptt}
In view of $\ef{ns1}_{2, 3}$, the relation  \ef{rela} and the definition of $\mathcal{E} (t)$,  we deduce that for any $t\in [0,T]$, 
 $$
\|\e\mathbf{u}\|_{L_t^\infty(H^2)}^2 +\|\e\theta \|_{L_t^\infty(H^3)} ^2
+\|\e^2\mathbf{u} \|_{L_t^\infty(H^3)}^2
\leq C [ (\mathcal{E}  (t))^3+\mathcal{E}  (t) ].
 $$
Moreover,  we can derive from $\p_t\ef{ns0}_{3,4}$   that 
  $$
\|\theta _{tt}\|_{L_t^2(L^2)}^2 + \|\e q _{tt}\|_{L_t^2(L^2)}^2 \leq C[(\mathcal{E}  (t))^3+ \mathcal{E}  (t) ].   $$
\end{rmk}

Then we aim to establish the following energy inequality.  
\begin{prop} \label{lemprop} 
For $\e\in(0,1]$ and $ t\in(0,1]$, there exist positive, non-decreasing, and continuous functions $C_0,C_1$
independent of $\e$ and $t$, such that
\begin{equation}\label{prop}
\mathcal{E} (t)
\leq 
C_0(M_0 )
{\rm exp}\left\{
\bigl(t^{\frac{1}{16}}+\e^{\frac{1}{4}}\bigr)
C_1\bigl(\mathcal{E} (t)\bigr)
\right\}.
\end{equation}

\end{prop}

\subsection{Estimates of  density }
\begin{lem}\label{lemrho}   For any $t\in[0, T]$ and $\e\in(0,1]$,  it holds that 
\begin{equation}\label{lemma1}
\Vert (\rho, \rho^{-1}) \Vert^2_{L_{x, t}^\infty}+\Vert \rho\Vert^2_{L_t^\infty(H^2)}+ \| \rho_t\|^2_{L_t^\infty(L^2)}+ \|\rho_t\|^2_{L_t^2(H^1)}\leq C_0(M_0){\rm exp}\{t^{\frac 1 4} C_1(\mathcal{E} (t)) \}.
\end{equation} 
\end{lem} 
\textbf{Proof}. Define the particle path $X_t(x)=X(x, t)$ associated to $\mathbf{u}$, 
\be\nonumber
\left\{\begin{aligned}
&\p_t X(x, t)= \mathbf{u}(t, X(x, t)), \\ 
&X(x, 0)=x. 
\end{aligned} \right. 
\ee
From the characteristics method,  the  solution of    $\ef{ns1}_1$ can be expressed as
\be \nonumber
\rho (x, t)=\rho_0(X(x, 0) ){\rm exp}\left ( -\int_0^t \dv\mathbf{u}(X(x, s), s) ds \right). 
\ee
 Thus  we get 
\begin{align}\label{rholb}
\|\rho^{-1}\|_{L_{x, t}^\infty}&\leq\|\rho_0^{-1}\|_{L^\infty(\Omega)}{\rm exp}\left(C\int_0^t\|\mathbf{u}\|_{H^3}ds\right) \leq  C_0(M_0){\rm exp}\{ C t^{\frac 1 2 } \mathcal{E} (t) \}. 
\end{align}
Similarly,  one can derive the upper bound of $\rho$.  Then,  multiplying  $\partial^\alpha {\eqref{ns1}}_1$ by $\partial^\alpha\rho$ ($|\alpha|\in [0, 2]$),  integrating the resulting equality over  $ \Omega\times[0, t]$,  we obtain  that 
$$ 
 \|\rho\|^2_{L_t^\infty(H^2)}\leq \|\rho_0\|^2_{ H^2} {\rm exp} \{C t^{\frac 1 2}\|\mathbf{u}\|_{L_t^2(H^3)} \}\leq  C_0(M_0){\rm exp}\{ C t^{\frac 1 2 } \mathcal{E} (t) \}.
$$ 
By taking the product of $\p_t \eqref{ns1}_1$ and $\rho_t$ in $L^2(\Omega\times[0, t])$,  one has
\begin{equation}\label{low21}\begin{aligned} 
\|\rho_t\|_{L^\infty_t(L^2)}^2
&\leq \|\rho_t(0)\|_{ L^2}^2 + C t^{\frac 1 2}(\|\rho_t\|_{L^\infty_t(L^2)}^2\| \mathbf{u}\|_{L^2_t(H^3)} + \|\rho_t\|_{L^\infty_t(L^2)}\|\rho\|_{L^\infty_t(H^2)}\| \mathbf{u}_t\|_{L^2_t(H^1)} )\\
&\leq  C_0(M_0){\rm exp}\{ C t^{\frac 1 2 } \mathcal{E} (t) \}.
\end{aligned}\end{equation} 
Moreover,  we directly obtain from $ \eqref{ns1}_1$ that 
\begin{equation}\label{low21-1}
\begin{aligned}
\|\rho_t\|_{L^2_t(H^1)}^2
&\leq C t^{\frac 1 2} \|\rho\|_{L_t^\infty(H^2)}^2 \| \mathbf{u}\|_{L^2_t(H^3)}  \| \mathbf{u}\|_{L^\infty_t(H^1)}  
\leq C t^{\frac 1 4} \mathcal{E} (t)^2.
\end{aligned}
\end{equation} 
From \ef{rholb},  \ef{low21}, and \ef{low21-1},  we conclude \ef{lemma1} immediately. 
\hfill$\Box$

\subsection{Low-order estimates of  $(q,\mathbf{v},\theta )$}
\begin{lem}\label{lemlow}   For any $t\in[0, T]$ and $\e\in(0,1]$,  it holds that 
\begin{equation} \nonumber
\|(\rho ,  q )\|_{L^\infty( L^2 )} ^2+\| \mathbf{v} \|_{L^\infty( H^1 )}^2 +\| \theta \|_{L^\infty( H^2 )} ^2 
\leq C_0(M_0){\rm exp}\{t  C_1(\mathcal{E} (t)) \} .
\end{equation} 
\end{lem}  
\textbf{Proof}.  By Newton-Leibniz's formula  for time variables, we   check that  
\be\begin{aligned}\nonumber
\|(\rho ,  q )\|_{L^\infty(0, t;L^2(\Omega))}^2 
&\leq  \|(\rho _0,  q_0)\|_{ L^2(\Omega)} ^2 +\int_0^t \|(\rho_t,  q_t)\|_{ L^2(\Omega)}^2 ds\\
&\leq C_0(M_0)+ Ct\|(\rho_t,  q_t)\|_{L^\infty( L^2 )}^2 \\
& \leq C_0(M_0)+ Ct \mathcal{E}(t), 
\end{aligned}\ee 
and  
\be\begin{aligned}\nonumber
\| \mathbf{v} \|_{L^\infty( H^1 )}^2 +\| \theta \|_{L^\infty( H^2 )} ^2 &\leq C_0(M_0) + Ct(  \| \mathbf{v}_t  \|_{L^2( H^1 )}^2 +\| \theta_t \|_{L^2(H^2 )} ^2 )\\
&\leq C_0(M_0)+ Ct \mathcal{E}(t).
\end{aligned}\ee   
The proof is completed. 
\hfill$\Box$

\subsection{Estimates of $( \eta_t ,  \mathbf{u}_{\e t}, \vartheta_t   )$}  \label{sec3.2}  
To establish the uniform bounds for $\vartheta_t$, we use the system  \ef{2ns} to  estimate $M_1(t)$ and $M_2(t)$.  
For convenience, we give the following estimates by use of  \ef{rela} and   Remark \ref{ptt},
$$ \|\Theta^{-1}\|_{L_t^\infty(H^2)} \lesssim  \|\Theta^{-1}\|_{L_{t,x}^\infty} + \|\Theta^{-1}\|_{L_{t,x}^\infty} ^2 \|\Theta \|_{L_t^\infty(H^2)} + \|\Theta^{-1}\|_{L_{t,x}^\infty} ^3 \|\Theta \|_{L_t^\infty(H^2)}  ^2\leq C_1(\mathcal{E}(t)) , $$ 
$$
 \|p^{-1}\|_{L_t^\infty(H^2)} \lesssim \|\Theta^{-1}\|_{L_t^\infty(H^2)}  \|\rho^{-1}\|_{L_t^\infty(H^2)} \leq C_1(\mathcal{E}(t)) ,
$$ 
$$
\|\n\sigma\|_{L_t^\infty(H^1)}=\| {p}^{-1}\n q\|_{L_t^\infty(H^1)}\lesssim  \|p^{-1}\|_{L_t^\infty(H^2)}  \|\n q\|_{L_t^\infty(H^1)} \leq C_1(\mathcal{E}(t)) ,
$$
$$
\left\|\frac{\n\sigma}{\e}\right\|_{L_t^2(H^1)}=\left\|{p}^{-1} \frac{\n q}{\e}\right\|_{L_t^2(H^1)}\lesssim \|p^{-1}\|_{L_t^\infty(H^2)}\left\| \frac{\n q}{\e}\right\|_{L_t^2(H^1)} \leq C_1(\mathcal{E}(t)) , 
$$ 
$$
\|\n\vartheta\|_{L_t^\infty(H^1)}=\|\Theta^{-1}\n\theta\|_{L_t^\infty(H^1)}\lesssim \|\Theta^{-1}\|_{L_t^\infty(H^2)}\|\n\theta\|_{L_t^\infty(H^1)}\leq C_1(\mathcal{E}(t)) ,
$$ 
$$
\|\n\vartheta\|_{L_t^2(H^2)}=\|\Theta^{-1}\n\theta\|_{L_t^2(H^2)}\lesssim  \|\Theta^{-1}\|_{L_t^\infty(H^2)}\|\n \theta\|_{L_t^2(H^2)} \leq C_1(\mathcal{E}(t)) , $$ 
and 
\begin{equation}\nonumber\begin{aligned}
&\|\e\sigma_{tt}\|_{L_t^2(L^2)}+\|\vartheta_{tt}\|_{L_t^2(L^2)}+ \|\eta_{tt}\|_{L_t^2(L^2)} \\
&\lesssim   \|\e q_{tt}\|_{L_t^2(L^2)}\|p^{-1} \|_{L_t^\infty(H^2)}+\|\e\sigma_t\|^2_{L_t^\infty( H^1)} + \| \theta_{tt}\|_{L_t^2(L^2)} \|\Theta^{-1} \|_{L_t^\infty(H^2)}+\|\e\vartheta_t\|^2_{L_t^\infty( H^1)}\\
& \leq  C_1(\mathcal{E}(t)) , 
\end{aligned}\end{equation}    
which will be applied in this part frequently.

\begin{lem}\label{lem3.2} 
For any $t\in[0, T]$ and $\e\in(0, 1]$,  we have the following estimate
\begin{equation}\label{lem1t1}\begin{aligned}
& \left\|\left( \eta_t ,  \mathbf{u}_{\e t},  \vartheta_t   \right)\right\|^2_{L_t^\infty(L^2)}
+ \left\| \left(   \n \mathbf{u}_{\e t},    \n \vartheta_t  \right)  \right\|^2_{L^2_t(L^2)} \leq C_0(M_0){\rm exp}\{t^{\frac 1 4} C_1( \mathcal{E}(t)) \}. 
\end{aligned}\end{equation} 
\end{lem} 
\textbf{Proof}. Differentiating \ef{2ns} with respect to time yields 
\begin{equation}\label{2nst}
\left\{
\begin{aligned} 
&  \eta_{tt} +\mathbf{u}\cdot \n  \eta_t +\mathbf{u}_t\cdot \n  \eta   +\frac{1}{2 \e}\dv \mathbf{u}_{\e t} = 0 &{\rm in }~ \Omega_t, \\
&  \mathbf{u}_{\e tt} +\mathbf{u}\cdot \n\mathbf{u}_{\e t}    +\frac{1}{\e} \Theta (2\n \eta_t + \n \vartheta_t) - \frac{1}{\rho} ( \mu\Delta\mathbf{u}_{\e t}  + (\mu+\lambda)\n\dv \mathbf{u}_{\e t}   ) =g_1  &{\rm in }~ \Omega_t,  \\ 
& \vartheta_{tt}+\mathbf{u}\cdot \n  \vartheta_t +\frac{1}{\e} \dv \mathbf{u}_{\e t} -\kappa  \frac{1}{p} \dv (\Theta \n\vartheta_t ) =  h_1 ~ &{\rm in }~ \Omega_t,  \\
&\mathbf{u}_{\e t}=0, \q  \frac{\p\vartheta_t }{\p\mathbf{n}}=0 
&{\rm on}~  \p\Omega_t,  \\
&(\mathbf{u}_{\e tt},  \eta_{tt},  \vartheta_{tt})|_{t=0}=(\mathbf{u}_{\e tt}(0),  \eta_{tt} (0),  \vartheta_{tt}(0))(x)  &{\rm in }~ \Omega,   
\end{aligned} \right.   
\end{equation}  
where 
$$
g_1:= -    \mathbf{u}_t\cdot \n\mathbf{u}_{\e }  +   \vartheta_t (\mathbf{u}_{\e t} +\mathbf{u}\cdot \n\mathbf{u}_{\e }   ) - \frac{1}{\rho} \e \sigma_t ( \mu\Delta\mathbf{u}_{\e }  +(\mu+\lambda)\n\dv \mathbf{u}_{\e } ),   $$
$$
h_1:= -   \Theta^{-1}  \mathbf{u}_t \cdot \n  e^\vartheta +\kappa  \frac{1}{p} \dv ( \vartheta_t \n \Theta )  - \kappa  \frac{1}{p} \e \sigma_t \Delta \Theta  +\e^2  \p_t\Big [\frac{1}{p}  (2\mu|D(\mathbf{u})|^2+\lambda(\dv\mathbf{u} )^2)\Big ]. 
$$
In addition,  the initial values $(\eta_t(x, 0),  \mathbf{u}_{\e t}(x, 0),  \vartheta_{t}(x, 0) )$ are given by
\begin{equation}\nonumber
\begin{aligned}  
&  \eta_t(x, 0) := - \mathbf{u}_0\cdot \n  \eta_0   - \frac{1}{2 \e}\dv \mathbf{u}_\e(0)  , \\
& \mathbf{u}_{\e t}(x, 0):=-\mathbf{u}_0\cdot \n\mathbf{u}_{\e } (0) - \frac{\Theta_0}{\e}(2\n \eta_0 + \n \vartheta_0) + \frac{\Theta_0}{p_0} ( \mu\Delta\mathbf{u}_\e(0)  +(\mu+\lambda)\n\dv \mathbf{u}_\e(0) ),  \\
&\vartheta_{t}(x, 0) := -\mathbf{u}_0\cdot \n  \vartheta_0 - \frac{1}{\e} \dv \mathbf{u}_\e (0)+\kappa  \frac{1}{p_0} \dv (\Theta _0\n\vartheta_0)+\frac{\e^2 }{p_0}  [2\mu|D(\mathbf{u}_0)|^2+\lambda(\dv\mathbf{u}_0 )^2]. 
\end{aligned}   
\end{equation}   
We multiply $\ef{2nst}_1$ by $4\eta_t $ and integrate the resulting equality over $ \Omega\times(0, t)$ to obtain 
\begin{equation}\label{et1}
\begin{aligned} 
&2\|\eta_t \|_{L^2}^2(t)+\frac{2}{\e} \omt \eta_t  \dv  \mathbf{u}_{\e t} dxds\\
& =2\|\eta_t (0)\|_{L^2}^2 +\omt 2 \dv \mathbf{u} |\eta_t|^2-4 (\mathbf{u}_{t}\cdot \n\eta )\eta_t  dxds\\
&\leq C_0(M_0) + t^{\frac 1 2}(\|\eta_t \|_{L_t^\infty(L^2)}^2 \|\mathbf{u}\|_{L_t^2(H^3)}+\|\eta_t \|_{L_t^\infty(L^2)} \|\n\eta \|_{L_t^\infty(H^1)}\|\mathbf{u}_t\|_{L_t^2(H^1)} ).
\end{aligned}   
\end{equation} 
By integrating the inner product of $\ef{2nst}_2$ and $\Theta \mathbf{u} _{\e t} $ in $L^2(\Omega\times(0, t))$,  one has 
\begin{equation}\label{et2}
\begin{aligned} 
& \frac{1}{2}\|\sqrt{\Theta}\mathbf{u} _{\e t} \|_{L^2}^2(t)-\frac{1}{\e} \omt ( 2 \eta_t + \vartheta_t)\dv  \mathbf{u}_{\e t} dxds \\ 
& + \mu \|\sqrt{p}^{-1} \n \mathbf{u}_{\e t} \|_{L_t^2(L^2)}^2  + (\mu+\lambda) \|\sqrt{p}^{-1} \dv \mathbf{u}_{\e t} \|_{L_t^2(L^2)} ^2  \\ 
& = \frac{1}{2}\|\sqrt{\Theta_0}\mathbf{u}_{\e t}(0) \|_{L^2}^2 +  \frac{1}{2} \omt (\Theta_t+\dv(\Theta \mathbf{u} ))|\mathbf{u} _{\e t}|^2 dxds + \omt \Theta g_1\cdot \mathbf{u}_{\e t} dxds \\
&\q - \omt \mu  \n (p^{-1}) \cdot \n\mathbf{u}_{\e t} \cdot\mathbf{u}_{\e t} +(\mu+\lambda) \dv \mathbf{u}_{\e t}  \n (p^{-1}) \cdot\mathbf{u}_{\e t}  dxds\\
&\leq C_0(M_0) +t^{\frac 1 4} \|\mathbf{u}_{\e t} \|_{L_t^\infty(L^2)}\|\mathbf{u}_{\e t} \|_{L_t^2(H^1)}(\| \Theta \|_{L_t^\infty(H^2)} \|\mathbf{u}\|_{L_t^2(H^3)}^{\frac 1 2}\|\mathbf{u}\|_{L_t^\infty(H^1)}^{\frac 1 2}  \\
&\q +\| \Theta _t\|_{L_t^\infty(H^1)} )
+ t^{\frac 1 4}  \|\n (p^{-1})\|_{L_t^\infty(H^1)} \|\mathbf{u}_{\e t} \|_{L_t^2(H^1)}^{\frac 3 2}\|\mathbf{u}_{\e t} \|_{L_t^\infty(H^1)}^{\frac 1 2} 
\\
&\q + t^{\frac 1 2}  \|\mathbf{u}_{\e t} \|_{L_t^2(H^1)} \| \Theta \|_{L_t^\infty(H^2)}[\|\mathbf{u}_{t} \|_{L_t^\infty(L^2)} \|\e \mathbf{u}\|_{L_t^\infty(H^2)} +(\|\mathbf{u}_{\e t} \|_{L_t^\infty(L^2)}  \\
&\q +\|\mathbf{u} \|_{L_t^\infty(H^1)}^2  )\| \vartheta_t \|_{L_t^\infty(H^1)}  +\|\rho^{-1}\|_{L^\infty_{x, t}} \|\e \sigma_t\|_{L_t^\infty(H^1)}\|\e \mathbf{u}\|_{L_t^\infty(H^2)} ].  
\end{aligned}   
\end{equation}  
Then,  taking the product of $\ef{2nst}_3$ and $ \vartheta_{t} $ in $L^2(\Omega\times(0, t))$ yields that 
\begin{equation}\label{et3}
\begin{aligned} 
& {\frac 1 2} \|\vartheta_t\|^2_{L^2}(t) + {\frac 1 \e}\omt   \vartheta_t  \dv  \mathbf{u}_{\e t} dxds +\kappa\|\sqrt{\rho}^{-1} \n \vartheta_t\|_{L_t^2(L^2)}^2 \\
& = {\frac 1 2} \|\vartheta_t(0)\|^2_{L^2}+ {\frac 1 2} \omt  \dv \mathbf{u} |\vartheta_t|^2  dxds +\omt  -\kappa \Theta \vartheta_t \n \vartheta_t\cdot \n(p^{-1})  +  \vartheta_t h_1 dxds\\
&\leq C_0(M_0)+ t^{\frac 1 2} \|\vartheta_t \|_{L_t^\infty(L^2)}^2 \|\mathbf{u}\|_{L_t^2(H^3)} +t \|\Theta \|_{L_{x, t}^\infty } \|\vartheta_t \|_{L_t^\infty(H^1)}^2  \|\n (p^{-1})\|_{L_t^\infty(H^1)}  \\
&\q + t^{\frac 1 2}  \|\vartheta_t \|_{L_t^\infty(H^1)} [ \|\Theta^{-1} \|_{L_{x, t}^\infty } \|\mathbf{u}_t \|_{L_t^\infty(L^2)}\|\Theta \|_{L_t^\infty(H^2)}  + \|p^{-1}\|_{L_{x, t}^\infty} \|\Theta \|_{L_t^\infty(H^2)}^2(\|\vartheta_t \|_{L_t^\infty(H^1)}\\ 
&\q +\|\e \sigma_t \|_{L_t^\infty(H^1)} ) + \|p^{-1}\|_{L_{x, t}^\infty}  (\|\e \sigma_t \|_{L_t^\infty(H^1)} \|\e \mathbf{u}\|_{L_t^\infty(H^2)} ^2 + \e \|\e \mathbf{u}\|_{L_t^\infty(H^2)}  \|\mathbf{u}_t \|_{L_t^2(H^1)})] . 
\end{aligned}    
\end{equation} 
From \ef{et1},  \ef{et2},  and \ef{et3},  we immediately derive \ef{lem1t1}.
\hfill$\Box$

\vskip3mm

Next, we aim to obtain the estimate of $M_2(t)$.   
\begin{lem}\label{llemtt}
For any $t\in[0, T]$ and $\e\in(0, 1]$,  it holds that
\begin{equation}\label{lemtt} \begin{aligned}
& \left\|\e ( \eta_{tt} ,  \mathbf{u}_{\e tt},  \vartheta_{tt}   )\right\|^2_{L^2}(t)
+ \left\|\e  ( \n \mathbf{u}_{\e tt},   \dv \mathbf{u}_{\e tt},  \n \vartheta_{tt}) \right\|^2_{L^2_t(L^2)}\\
&\leq  C_0(M_0){\rm exp}\{(t^{\frac 1 4} + \e^{\frac 1 2} ) C_1(\mathcal{E} (t)) \}. 
\end{aligned}\end{equation} 
\end{lem} 
\textbf{Proof}. Differentiate  $ \ef{2ns}$  twice with respect to $t$ yields 
\begin{equation}\label{2nstt}
\left\{
\begin{aligned} 
&  \eta_{ttt} +\mathbf{u}\cdot \n  \eta_{tt} +\mathbf{u}_{tt}\cdot \n  \eta + 2 \mathbf{u}_{t}\cdot \n  \eta_t   + \frac{1}{2 \e}\dv \mathbf{u}_{\e tt} = 0 &{\rm in }~ \Omega_t, \\
&  \mathbf{u}_{\e ttt} +\mathbf{u}\cdot \n\mathbf{u}_{\e tt}    +\frac{1}{\e} \Theta (2\n \eta_{tt} + \n \vartheta_{tt}) - \frac{1}{\rho} ( \mu\Delta\mathbf{u}_{\e tt}  + (\mu+\lambda)\n\dv \mathbf{u}_{\e tt}   ) =g_2  &{\rm in }~ \Omega_t,  \\ 
& \vartheta_{ttt}+\mathbf{u}\cdot \n  \vartheta_{tt} +\frac{1}{\e} \dv \mathbf{u}_{\e tt} -\kappa  \frac{1}{p} \dv (\Theta \n\vartheta_{tt} ) =  h_2   &{\rm in }~ \Omega_t,  \\
&\mathbf{u}_{\e tt}=0, \q  \frac{\p\vartheta_{tt}}{\p\mathbf{n}}=0 
&{\rm on}~  \p\Omega_t,  \\
&(\mathbf{u}_{\e tt},  \eta_{tt},  \vartheta_{tt})|_{t=0}=(\mathbf{u}_{\e tt}(0),  \eta_{tt} (0),  \vartheta_{tt}(0))(x)  &{\rm in }~ \Omega,   
\end{aligned} \right.   
\end{equation}  
where the initial  data are defined as
 \begin{equation}\nonumber
\begin{aligned} 
\mathbf{u}_{\e tt} (x, 0)&:=  -\mathbf{u}_0\cdot \n\mathbf{u}_{\e t}  (0)  -\frac{1}{\e} \Theta_0 (2\n \eta_t(0) + \n \vartheta_t (0) ) \\
&\q +\frac{1}{\rho_0} ( \mu\Delta\mathbf{u}_{\e t} (0) + (\mu+\lambda)\n\dv \mathbf{u}_{\e t} (0)   ) +g_1(x, 0), \\
 \eta_{tt}(x, 0)&:= -\mathbf{u}_0\cdot \n  \eta_t (0)-\mathbf{u}_t(0)\cdot \n  \eta _0  -\frac{1}{2 \e}\dv \mathbf{u}_{\e t} (0), \\
  \vartheta_{tt}(x, 0)&:=  -\mathbf{u}_0\cdot \n  \vartheta_t(0) -\frac{1}{\e} \dv \mathbf{u}_{\e t}(0) +\kappa  \frac{1}{p_0} \dv (\Theta_0 \n\vartheta_t(0) ) + h_1 (x, 0), 
\end{aligned} 
\end{equation} 
and the right-hand side terms are defined as 
 \begin{equation}\nonumber
\begin{aligned} 
g_2& := - \mathbf{u}_t\cdot \n\mathbf{u}_{\e t} -\vartheta_t (\mathbf{u}_{\e tt} +\mathbf{u}\cdot \n\mathbf{u}_{\e t})-\frac{1}{\rho} \e \sigma_t ( \mu\Delta\mathbf{u}_{\e t}  + (\mu+\lambda)\n\dv \mathbf{u}_{\e t} ) + \p_t g_1 - \vartheta_t g_1\\
&=-(\mathbf{u}_{tt} \cdot \n \mathbf{u}_{\e}+2\mathbf{u}_{t}\cdot \n\mathbf{u}_{\e t} ) -2\vartheta_{t}(\mathbf{u}_{\e t t}+\mathbf{u} \cdot \n \mathbf{u}_{\e t}) +(\vartheta_t^2-\vartheta_{tt} ) (\mathbf{u}_{\e t}+\mathbf{u} \cdot \n \mathbf{u}_{\e })\\
&\quad -\frac{2}{\rho} \e \sigma_t (\mu\Delta\mathbf{u}_{\e t}  + (\mu+\lambda)\n\dv \mathbf{u}_{\e t}  ) +\frac{1}{\rho}((2\vartheta_t- \e \sigma_t) \e \sigma_t-  \e \sigma_{tt} )( \mu\Delta\mathbf{u}_{\e }  + (\mu+\lambda)\n\dv \mathbf{u}_{\e }  ), 
\end{aligned} 
\end{equation} 
 \begin{equation}\nonumber
\begin{aligned} 
h_2&:= - \mathbf{u}_t\cdot \n\vartheta_t - \kappa\frac{1}{p}\e \sigma_t \dv(\Theta\n\vartheta_t)+\kappa\frac{1}{p} \dv(\Theta \vartheta_t \n\vartheta_t)+\p_t h_1\\
&= -( \mathbf{u}_{tt}\cdot \n\vartheta+2\mathbf{u}_t\cdot \n\vartheta_t   )- \frac{2\kappa}{p}\e \sigma_t \dv(\Theta\n\vartheta_t+ \Theta \vartheta_t\n \vartheta )+\frac{\kappa}{p} ( (\e \sigma_t)^2-\e \sigma_{tt} ) \dv(\Theta\n\vartheta  )\\
&\q + \frac{\kappa}{p} \dv(\Theta\vartheta_t^2\n\vartheta+\Theta\vartheta_{tt}\n\vartheta+2\Theta\vartheta_t\n\vartheta_t) -\e^2 \frac{2}{p} \e \sigma_t (2\mu|D(\mathbf{u})|^2+\lambda(\dv\mathbf{u} )^2 )\\
&\q +\e^2 \frac{2}{p}  (2\mu D(\mathbf{u}):D(\mathbf{u}_t)+\lambda \dv\mathbf{u}\dv\mathbf{u} _t  ) .
\end{aligned} 
\end{equation} 
 Multiplying $\ef{2nstt}_1$ by $4 \e^2 \Theta \eta_{tt}$ and integrating the resulting equality over $\Omega\times(0,t)$,  we derive 
\begin{align} 
&2 \|\e\sqrt{\Theta}\eta_{tt}\|_{L^2}^2(t) +2\e \omt \Theta \eta_{tt} \dv\mathbf{u}_{\e tt} dxds \nonumber \\
&= 2 \|\e\sqrt{\Theta_0}\eta_{tt}(0)\|_{L^2}^2 + 2\omt  |\e \eta_{tt}|^2 (\dv ( \Theta \mathbf{u} )+\Theta_t) dx \nonumber\\
&\q - 4 \omt \e^2 \Theta \eta_{tt} (\mathbf{u}_{tt}\cdot \n  \eta + 2 \mathbf{u}_{t}\cdot \n  \eta_t)  dxds \nonumber \\
&\lesssim C_0(M_0)+  t^{\frac 1 2}   \|\e \eta_{tt}\| _{L_t^\infty(L^2)}^2( \|\mathbf{u}\|_{L_t^2(H^3)} \|\Theta\|_{L_t^\infty(H^2)}+\|\Theta_t\|_{L_t^2(H^2)}     \label{tt1}
 \\ 
&\q+ \| \mathbf{u}\|_{L_t^2(H^3)}^{\frac 1 2} \| \mathbf{u}\|_{L_t^\infty(H^1)}^{\frac 1 2} \|\Theta\|_{L_t^\infty(H^2)}^{\frac 1 2} \|\Theta\|_{L_t^2(H^4)}^{\frac 1 2}  ) + t^{\frac 1 4}\|\eta_{tt}\| _{L_t^2(L^2)}  (\e \|\n\sigma\|_{L_t^\infty(H^1)} \nonumber \\
&\q \cdot \|\e^2 \mathbf{u}_{tt}\|_{L_t^\infty(L^2)} ^{\frac 1 2} \|\e^2 \mathbf{u}_{tt}\|_{L_t^2 (H^1)} ^{\frac 1 2}\|\Theta\|_{L_t^\infty(H^2)}  + \|\e^2 \mathbf{u}_{tt}\|_{L_t^\infty(L^2)} \|\n\Theta\|_{L_t^\infty(H^1)} ^{\frac 1 2} \|\n\Theta\|_{L_t^2(H^3)} ^{\frac 1 2} )  \nonumber \\
&\q +\e^{\frac 1 2 } \|\n\eta_{t}\| _{L_t^\infty(L^2)} \|\eta_{tt}\| _{L_t^2(L^2)} \|\e^2 \mathbf{u}_{t}\|_{L_t^2 (H^2)} ^{\frac 3 4} \| \mathbf{u}_{t}\|_{L_t^2 (H^1)} ^{\frac 1 4} \|\Theta\|_{L_t^\infty(H^2)}\nonumber   \\
&\leq C_0(M_0){\rm exp}\{(t^{\frac 1 4} + \e^{\frac 1 2} ) C_1(\mathcal{E} (t)) \}.    \nonumber 
\end{align} 
Next we integrate the product of $\ef{2nstt}_3$ and $ \e^2 \Theta \vartheta_{tt}$ in $L^2(\Omega\times(0,t))$ to get 
\begin{equation} \label{tt2}
\begin{aligned} 
&\frac{1}{2} \|\e\sqrt{\Theta}\vartheta_{tt}\|^2_{L^2}(t) +\e \omt \Theta \vartheta_{tt} \dv\mathbf{u}_{\e tt} dxds+ \kappa\| \e \Theta \sqrt{\rho}^{-1}\n \vartheta_{tt}\| _{L_t^2(L^2)}^2 \\ 
&= \frac{1}{2} \|\e\sqrt{\Theta_0}\vartheta_{tt}(0)\|_{L^2}^2
+R_1+R_2, 
\end{aligned} 
\end{equation} 
where $R_1$ and $R_2$  are defined and estimated as follows
\begin{equation} \nonumber
\begin{aligned} 
R_1:&=\frac{1}{2} \omt  {\Theta}_t |\e\vartheta_{tt}|^2 dxds -\kappa \omt \e^2 {\Theta} \n {\vartheta}_{tt}\cdot \n (\rho^{-1}) {\vartheta}_{tt} dxds\\
&\lesssim t^\frac{1}{2}\|{\Theta}_t\|_{L_t^2(H^2)} \|\e\vartheta_{tt}\|_{L_t^\infty(L^2)} ^2 + t^\frac{1}{4} \|\Theta\|_{L_t^\infty(H^2)}  \|\e\vartheta_{tt}\|_{L_t^2(H^1)} ^{\frac 3 2}\|\e\vartheta_{tt}\|_{L_t^\infty(L^2)} ^{\frac 1 2} \| \rho^{-1}\|_{L_t^\infty(H^2)} \\
&\leq t^{\frac{1}{4}}  C_1(\mathcal{E} (t)), 
\end{aligned} 
\end{equation} 
 and
\begin{align} 
R_2:&= \omt \e^2 {\Theta} \vartheta_{tt} \cdot h_2 dxds   \nonumber \\
& \lesssim t^{\frac 1 4}   (\|\e^2 \mathbf{u}_{tt}\|_{L_t^\infty(L^2)} \|\vartheta_{tt} \|_{L_t^2(L^2)}  \|\Theta \|_{L_t^2(H^4)}^{\frac 1 2}\|\Theta \|_{L_t^\infty(H^2)}^{\frac 1 2} + \|\e \vartheta_{tt} \|_{L_t^\infty(L^2)}^{\frac 1 2}  \|\e \vartheta_{tt} \|_{L_t^2(H^1)}^{\frac 1 2} 
\nonumber \\
&\q \cdot\|\Theta\|_{L_t^\infty(H^2)} \| \mathbf{u}_{t}\|_{L_t^\infty(L^2)}\|\vartheta_{t} \|_{L_t^2(H^2)}) + t^{\frac 1 2}  \|p^{-1}\|_{L_{x, t}^\infty}    ( \|\e \sigma_t\|^2_{L_{t}^\infty(H^1)} +\|\e^2 \sigma_{tt}\|^2_{L_{t}^\infty(L^2)} )
\nonumber\\
&\q \cdot  \|\e \vartheta_{tt} \|_{L_t^2(H^1)}  \| \Theta \|_{L_t^\infty(H^2)}^{\frac 1 2} \| \e \Theta \|_{L_t^\infty(H^3)}^{\frac 1 2}+
t^{\frac 1 4} \|\Theta\|_{L_t^\infty(H^2)} ^2 \|p^{-1}\|_{L_{x, t}^\infty}[ \|\e \vartheta_{tt} \|_{L_t^\infty(L^2)}(\| \vartheta_{t} \|_{L_t^\infty(H^1)}^2
  \nonumber \\ 
&\q  + \|\e \vartheta_{tt} \|_{L_t^2(H^1)}  )\|\Theta \|_{L_t^2(H^4)}^{\frac 1 2}\|\Theta \|_{L_t^\infty(H^2)}^{\frac 1 2} +\| \vartheta_{t} \|_{L_t^2(H^2)}  \| \vartheta_{t} \|_{L_t^\infty(H^1)}  \|\e \vartheta_{tt} \|_{L_t^\infty(L^2)}^{\frac 1 2} \|\e \vartheta_{tt} \|_{L_t^2(H^1)}^{\frac 1 2}] 
\nonumber \\ 
&\q + t^{\frac 1 2} \|p^{-1}\|_{L_{x, t}^\infty} \|\Theta\|_{L_t^\infty(H^2)} \|\e \vartheta_{tt} \|_{L_t^\infty(L^2)} [ \e^3 \|\e \sigma_t\|_{L_{t}^\infty(H^1)}
\| \mathbf{u}\|_{L_t^2(H^3)} \| \mathbf{u}\|_{L_t^\infty(H^1)}  \nonumber\\
&\q + \e  \|\e \sigma_t\|_{L_{t}^\infty(H^1)}   \|\Theta\|_{L_t^\infty(H^2)}\| \vartheta_{t} \|_{L_t^2(H^2)}+ \| \e^2\mathbf{u}_t\|_{L_t^2(H^2)} \| \e \mathbf{u}\|_{L_t^\infty(H^2)} ]   \nonumber \\
&\leq (t^{\frac{1}{4}}+\e ) C_1(\mathcal{E} (t)). 
\nonumber 
\end{align} 
Then we take the inner product of  $\ef{2nstt}_2$ with $\e^2 \mathbf{u}_{\e tt}$   in $L^2(\Omega\times(0, t))$ to find 
 \begin{equation} \label{tt3}
\begin{aligned} 
&\frac{1}{2}  \|\e^2 \mathbf{u}_{tt}\|_{L^2}^2(t)+\e \omt \Theta (2\n \eta_{tt} + \n \vartheta_{tt}) \cdot \mathbf{u}_{\e tt} dxds \\
& +\mu \|\e \sqrt{\rho}^{-1} \n \mathbf{u}_{\e tt}\|_{L_t^2(L^2)}^2+(\mu+\lambda) \|\e  \sqrt{\rho}^{-1} \dv  \mathbf{u}_{\e tt} \|_{L_t^2(L^2)}^2\\
&=\frac{1}{2}  \|\e^2 \mathbf{u}_{tt}(0)\|_{L^2}^2 +R_3+R_4, 
\end{aligned}   
\end{equation} 
where $R_3$ and $R_4$  are defined and estimated as follows
 \begin{equation} \nonumber
\begin{aligned} 
R_3:&= \frac{1}{2}\omt  \dv\mathbf{u} |\e^2 \mathbf{u}_{tt} |^2 dxds - \mu \omt \e^2 \n(\rho^{-1})  \cdot \n \mathbf{u}_{\e tt}\cdot \mathbf{u}_{\e tt} dxds \\
&\q - (\mu+\lambda) \omt \e^2 \dv \mathbf{u}_{\e tt} \n(\rho^{-1})\cdot  \mathbf{u}_{\e tt} dxds\\
&\lesssim  t^ \frac{1}{2} \| \mathbf{u}\|_{L_t^2(H^3)}  \|\e^2 \mathbf{u}_{tt}\|_{L_t^\infty(L^2)} ^2 + t^ \frac{1}{4} \|\e^2 \mathbf{u}_{tt}\|_{L_t^2(H^1)}^ \frac{3}{2}\|\e^2 \mathbf{u}_{tt}\|_{L_t^\infty(L^2)} ^ \frac{1}{2}  \|\rho^{-1}\|_{L_t^\infty(H^2)}\\
&\leq t^{\frac{1}{4}}  C_1(\mathcal{E} (t)), 
\end{aligned}   
\end{equation} 
and 
\begin{align} 
R_4:&=\omt \e^2 \mathbf{u}_{\e tt}\cdot g_2 dxds \nonumber\\
&\lesssim  t^{\frac 1 2} \|\e^2 \mathbf{u}_{tt}\|_{L_t^\infty(L^2)}^2\|\mathbf{u}\|_{L_t^2(H^3)} + \e \|\e^2 \mathbf{u}_{tt}\|_{L_t^\infty(L^2)} \| \mathbf{u}_{t}\|_{L_t^2(H^1)}^{\frac 3 2}\| \e^2 \mathbf{u}_{t}\|_{L_t^2(H^1)}^{\frac 1 2} \nonumber \\
& \q +  t^{\frac 1 2}  \|\e^2 \mathbf{u}_{tt}\|_{L_t^\infty(L^2)} \| \vartheta_{t}\|_{L_t^\infty(H^1)} (\|\e^2 \mathbf{u}_{tt}\|_{L_t^2(H^1)}   + \|\mathbf{u} \|_{L_t^\infty(H^1)} \|\e^2 \mathbf{u}_{t}\|_{L_t^2(H^2)}  ) 
\nonumber \\
&\q +\e  t^{\frac 1 4} \|\e^2 \mathbf{u}_{tt}\|_{L_t^\infty(L^2)}( \| \vartheta_{t}\|_{L_t^\infty(H^1)}^2+\| \e\vartheta_{tt}\|_{L_t^\infty(L^2)}^{\frac 1 2} \| \e\vartheta_{tt}\|_{L_t^2(H^1)}^{\frac 1 2}  ) ( \| \mathbf{u}_{t}\|_{L_t^2(H^1)}  \nonumber \\
&\q + \|\mathbf{u} \|_{L_t^\infty (H^1)} \| \mathbf{u}\|_{L_t^2(H^3)}  ) + \e  t^{\frac 1 4} ( \| \vartheta_{t}\|_{L_t^\infty(H^1)}\| \e \sigma_{t}\|_{L_t^\infty(H^1)}  +\| \e \sigma_{t}\|_{L_t^\infty(H^1)} ^2 \nonumber \\
&\q+\| \e^2 \sigma_{tt}\|_{L_t^\infty(L^2)}  ) \| \mathbf{u}\|_{L_t^2(H^3)} \|\e^2 \mathbf{u}_{tt}\|_{L_t^\infty(L^2)}^{\frac 1 2} \|\e^2 \mathbf{u}_{tt}\|_{L_t^2(H^1)}^{\frac 1 2} 
\nonumber\\
&\q +  t^{\frac 1 4} \|\e^2 \mathbf{u}_{tt}\|_{L_t^\infty(L^2)}^{\frac 1 2} \|\e^2 \mathbf{u}_{tt}\|_{L_t^2(H^1)}^{\frac 1 2} \| \e \sigma_{t}\|_{L_t^\infty(H^1)} \|\rho^{-1}\|_{L_{x, t}^\infty}\|\e^2 \mathbf{u}_{t}\|_{L_t^2(H^2)} 
\nonumber\\
&\leq (t^{\frac{1}{4}}+\e ) C_1(\mathcal{E} (t)). \nonumber 
\end{align} 
Now,  we deal with the singular terms in \ef{tt1},  \ef{tt2},  and \ef{tt3}. From integration by parts and $\mathbf{u}_{\e tt} |_{\p\Omega}=0$,  one has 
 \begin{equation}  \label{tt4} 
\begin{aligned} 
&\e \omt \Theta (2\n \eta_{tt} + \n \vartheta_{tt}) \cdot \mathbf{u}_{\e tt} dxds+\e \omt \Theta (2\eta_{tt}+ \vartheta_{tt} ) \dv\mathbf{u}_{\e tt} dxds \\
& = -\e \omt  (2 \eta_{tt} + \vartheta_{tt}) \dv( \Theta \mathbf{u}_{\e tt}) dxds+\e \omt \Theta (2\eta_{tt}+ \vartheta_{tt} ) \dv\mathbf{u}_{\e tt} dxds \\
&=-\e \omt  (2 \eta_{tt} + \vartheta_{tt}) \n\Theta \cdot \mathbf{u}_{\e tt} dxds\\
&\leq Ct^{\frac 1 4} (\|\eta_{tt}\|_{L_t^2(L^2)} +\| \vartheta_{tt}\|_{L_t^2(L^2)})\|\e \mathbf{u}_{\e tt}\|_{L_t^\infty(L^2)}\|\Theta \|_{L_t^\infty(H^2)}^{\frac 1 2}\|\Theta \|_{L_t^2(H^4)}^{\frac 1 2}\\
&\leq t^{\frac{1}{4}}  C_1(\mathcal{E} (t)).
\end{aligned}   
\end{equation} 
Consequently,  combining \ef{tt4} and   above estimates,  we obtain \ef{lemtt}.
\hfill$\Box$

\vskip3mm 
With Lemma \ref{llemtt},  we further derive some useful weighted  estimates of  first-order time derivative.
Applying the elliptic theory to $\ef{2nst}_2$,   one has   
\begin{equation}
\label{uth2}
\begin{aligned} 
\|\e^2  \mathbf{u}_{t}\|_{L_t^2(H^2)}
& \lesssim  \|\e \rho(\mathbf{u}_{\e tt} +\mathbf{u}\cdot \n\mathbf{u}_{\e t} + \mathbf{u}_t\cdot \n\mathbf{u}_{\e }   ) +  \rho \Theta (2\n \eta_t + \n \theta_t) \|_{L_t^2(L^2)} \\ 
&\q + \|\e \rho \theta_t (\mathbf{u}_{\e t} +\mathbf{u}\cdot \n\mathbf{u}_{\e }   ) +   \e^2 q_t ( \mu\Delta\mathbf{u}_{\e }  +(\mu+\lambda)\n\dv \mathbf{u}_{\e } )\|_{L_t^2(L^2)}
\\
&\lesssim   \|\rho\|_{L_{x, t}^\infty} (t \|\e^2 \mathbf{u}_{tt}\|_{L_t^\infty(L^2)} + \e \|\mathbf{u}\|_{L_{t}^\infty(H^1)} \|\mathbf{u}_t\|_{L_{t}^2(H^1)}^{\frac{1}{2}} \|\e^2 \mathbf{u}_t\|_{L_{t}^2(H^2)}^{\frac{1}{2}}\\
&\q +\e^2 \|\mathbf{u}_t\|_{L_{t}^\infty(L^2)} \|\mathbf{u}\|_{L_{t}^2(H^3)}   ) + t \|\rho\|_{L_{x, t}^\infty}  \|\Theta\|_{L_{x, t}^\infty} \|(\n \eta_t,  \n \theta_t) \|_{L_{t}^\infty(L^2)}
\\ 
&\quad +\e^2 \|\rho\|_{L_{x, t}^\infty} \|\theta_t\|_{L_{t}^\infty(H^1)} (  \| \mathbf{u}_t\|_{L_{t}^2(H^1)} +\|\mathbf{u}\|_{L_{t}^\infty(H^1)} \|\mathbf{u}\|_{L_{t}^2(H^2)}  ) 
\\
&\q + \e^3\|\e \sigma_t\|_{L_t^\infty(H^1)} \|\mathbf{u}\|_{L_{t}^2(H^3)}
\\
 & \leq (t +\e )C_1(\mathcal{E}(t)). 
\end{aligned}\end{equation}

Moreover,  multiplying $ \ef{2nst}_2$ by $\e^2 \rho \mathbf{u}_{\e tt} $ and integrating the resulting equality over $ \Omega\times[0, t]$ by parts give that
\begin{equation}\label{uth1}
\begin{aligned} 
&\frac{\mu}{2} \|\e \n \mathbf{u}_{\e t}\|_{L_t^\infty(L^2)} ^2+ \frac{\mu+\lambda}{2}  \|\e \dv \mathbf{u}_{\e t}\|_{L_t^\infty(L^2)}^2+\frac{1}{2}\| \sqrt{\rho} \e \mathbf{u}_{\e tt}\|_{L_t^2(L^2)} ^2
\\
& = \frac{\mu}{2}\|\e \n \mathbf{u}_{\e t}(0)\|_{ L^2}^2 + \frac{\mu+\lambda}{2} \|\e \dv \mathbf{u}_{\e t}(0)\|_{ L^2} ^2+\omt \rho \Theta (2\n \eta_t + \n \theta_t) \cdot \e \mathbf{u}_{\e tt}  \\
&\q -   [ \rho (\mathbf{u}\cdot \n\mathbf{u}_{\e } )_t
  + \rho \theta_t (\mathbf{u}_{\e t} +\mathbf{u}\cdot \n\mathbf{u}_{\e }   ) 
+   \e \sigma_t ( \mu\Delta\mathbf{u}_{\e }  +(\mu+\lambda)\n\dv \mathbf{u}_{\e } )] \cdot \e^2 \mathbf{u}_{\e tt} dxds\\
&\leq C_0(M_0)+Ct^{\frac 1 2} \|\rho\|_{L_{x, t}^\infty}  \|\Theta\|_{L_{x, t}^\infty} \|\e \mathbf{u}_{\e tt}\|_{L_t^2(L^2)} (\|\n \eta_t\|_{L_t^\infty(L^2)}+\|\n \theta_t\|_{L_t^\infty(L^2)}  )\\
&\q
+C \e \|\e \mathbf{u}_{\e tt}\|_{L_t^2(L^2)}  [\|\rho\|_{L_{x, t}^\infty} ( \|\mathbf{u}\|_{L_t^\infty(H^1)}\|\e^2\mathbf{u}_{t}\|_{L_t^2(H^2)}^{\frac 1 2} \|\mathbf{u}_{t}\|_{L_t^2(H^1)}^{\frac 1 2} \\
&\q +\|\mathbf{u}_{t}\|_{L_t^2(H^1)} \|\e \mathbf{u}\|_{L_t^\infty(H^2)}  + \|\theta_t \|_{L_t^\infty(H^1)} ( \e \|\mathbf{u}_{t}\|_{L_t^2(H^1)}+ \|\mathbf{u}\|_{L_t^\infty(H^1)}\|\e \mathbf{u}\|_{L_t^\infty(H^2)} )\\
&\q + \e \|\e q_t\|_{L_t^\infty(H^1)} \|\mathbf{u}\|_{L_t^2(H^3)}]\\
&\leq C_0(M_0)  {\rm exp}\{ (t^{\frac 1 2}+\e)C_1(\mathcal{E}(t)) \}. 
\end{aligned}\end{equation}

Therefore,  we conclude the following lemma.
\begin{lem} \label{leme2t}
For arbitrary $\e\in(0, 1]$ and $t\in [0, T]$, we have   
\be 
 \|\e^2( \n \mathbf{u}_{t}, \dv \mathbf{u}_{t}) \|_{L_t^\infty(L^2)}^2+ \|\e^2  \mathbf{u}_{t}\|_{L_t^2(H^2)}^2\leq C_0(M_0) {\rm exp}\{ (t^{\frac 1 2}+\e)C_1(\mathcal{E}(t)) \}.
 \ee 
\end{lem}

Next, we apply the weighted estimates in Lemmas 
\ref{llemtt} and \ref{leme2t} to prove the high-order estimates of $(  \eta_t ,   \mathbf{u}_{\e t},  \vartheta_t   )$.

\begin{lem}\label{lem3.3}
For any $t\in[0, T]$ and $\e\in(0, 1]$,  the following estimates hold:
\begin{equation}\label{lem1t2}\begin{aligned}
& \left\|\left(\n  \eta_t ,  \dv \mathbf{u}_{\e t},  \n \vartheta_t   \right)\right\|^2_{L^2}(t) 
+ \left\| \left(  \n  \dv \mathbf{u}_{\e t},    \Delta \vartheta_t  \right) \right\|^2_{L^2_t(L^2)}\\
&\leq C_0(M_0){\rm exp}\{(t^{\frac 1 4} + \e^{\frac 1 4} ) C_1( \mathcal{E}(t)) \}+\delta \| \mathbf{u}_t\|_{L_t^2(H^1)}^2,
\end{aligned}\end{equation}  
where $\delta>0$ is a sufficiently small constant.
\end{lem} 
\textbf{Proof}. {\it Step 1. Interior estimate}. Let $\chi_0\in C_0^\infty(\Omega)$ be an interior cutoff function. 
Multiplying $\p_k\ef{2nst}_1$ by  $4\chi_0^2 \p_k \eta_{t}$ and integrating over $ \Omega\times (0, t)$,  we get 
\begin{equation}\label{int1}
\begin{aligned} 
&2\|\chi_0 \p_k \eta_t\|_{L^2}^2(t)+\frac{2}{\e} \omt \p_k \dv\mathbf{u}_{\e t}\cdot \chi_0^2 \p_k \eta_{t} dxds\\
&= 2\|\chi_0 \p_k \eta_t(0)\|_{L^2}^2+ 2\omt \dv(\chi_0^2 \mathbf{u} )| \p_k \eta_{t} |^2 dxds\\
&\q - 4\omt ( \p_k \mathbf{u}\cdot \n \eta_{t}+\p_k \mathbf{u}_t\cdot \n \eta+  \mathbf{u}_t \cdot \n \p_k \eta) \chi_0^2   \p_k \eta_{t} dxds \\
&\lesssim C_0(M_0) + t^{\frac 1 2} \| \p_k \eta_{t} \|_{L_t^\infty(L^2)}^2 \|\mathbf{u}\|_{L_t^2(H^3)} +  \| \p_k \eta_{t} \|_{L_t^\infty(L^2)}\\
&\q \cdot \bigg ( \e^{\frac 1 2}  \bigg\|\frac{\n \sigma}{\e }\bigg\|_{L_t^2(H^1)}  \| \e^2 \mathbf{u}_t\|_{L_t^2(H^2)}^{\frac 3 4} \|   \mathbf{u}_t\|_{L_t^2(H^1)}^{\frac 1 4}  + t^{\frac 1 4} \|   \mathbf{u}_t\|_{L_t^2(H^1)} \|\n \vartheta\|_{L_t^\infty(H^1)}^{\frac 1 2} \|\n \vartheta\|_{L_t^2(H^2)}^{\frac 1 2} \bigg), 
\end{aligned}  
\end{equation}  
by applying $ \eta = \frac{\e}{2} \sigma -\frac{\vartheta}{2} $. 
Then we take the inner product of $\p_k(\frac{1}{\Theta}\ef{2nst}_2)$ and $\chi_0^2 \p_k\mathbf{u}_{\e t}$  in $L^2(\Omega\times (0, t))$ to show
\begin{equation}\label{int2}
\begin{aligned} 
&\frac{1}{2} \|\chi_0 \sqrt{\Theta}^{-1} \p_k \mathbf{u}_{\e t}\|_{L^2}^2(t)  +\frac{1}{\e} \omt  \chi_0^2 \p_k (2\n \eta_{t} +\n \vartheta_{t}) \p_k \mathbf{u}_{\e t} dxds\\
&+ \mu \|\chi_0 \sqrt{p}^{-1}\p_k \n \mathbf{u}_{\e t}\|_{L_t^2(L^2)}^2+(\mu +\lambda) \|\chi_0 \sqrt{p}^{-1}\p_k \dv \mathbf{u}_{\e t}\|_{L_t^2(L^2)}^2\\
& = \frac{1}{2}  \|\chi_0 \sqrt{\Theta_0}^{-1} \p_k \mathbf{u}_{\e t}(0)\|_{L^2}^2+\sum_{i=1}^{3} \mathcal{I}_i, 
\end{aligned}  
\end{equation}   
where 
\begin{equation}\nonumber
\begin{aligned} 
\mathcal{I}_1& := \omt   \frac{1}{2} (  \p_t (\Theta^{-1}) \chi_0^2  +   \dv ( \Theta \mathbf{u}\chi_0^2 )  )  |\p_k  \mathbf{u}_{\e t}|^2 -(\p_k( \Theta \mathbf{u})\cdot\n \mathbf{u}_{\e t} ) \chi_0^2 \p_k  \mathbf{u}_{\e t}  dxds\\
&\lesssim t^ {\frac 1 2}  \|\chi_0\p_k  \mathbf{u}_{\e t} \|_{L_t^\infty(L^2)}^2 ( \|{\Theta}^{-1} \|_{L_{x, t}^\infty} \|\vartheta_t \|_{L_t^2(H^2)} +\|\Theta \|_{L_t^\infty(H^2)} \| \mathbf{u}\|_{L_t^2(H^3)}\\
&\q+ \|\Theta \|_{L_t^\infty(H^2)}^ {\frac 1 2} \|\Theta \|_{L_t^2(H^4)}^ {\frac 1 2} \|\mathbf{u}\|_{L_t^2(H^3)}^ {\frac 1 2}\|\mathbf{u}\|_{L_t^\infty(H^1)}^ {\frac 1 2} ), 
\end{aligned}  
\end{equation} 
\begin{equation}\nonumber
\begin{aligned} 
\mathcal{I}_2:&=  \omt\chi_0^2 \p_k  \mathbf{u}_{\e t} ( \mu \p_k(p^{-1})\Delta \mathbf{u}_{\e t}+(\mu+\lambda) \mu \p_k(p^{-1}) \n\dv  \mathbf{u}_{\e t} ) dxds\\
&\lesssim  t^{\frac 1 4}\|\chi_0  \p_k \mathbf{u}_{\e t} \|_{L_t^\infty(L^2)}^{\frac 1 2}\|\chi_0  \mathbf{u}_{\e t} \|_{L_t^2(H^2)}^{\frac 3 2} \|p^{-1}\|_{L_t^\infty(H^2)}, 
\end{aligned}  
\end{equation} 
and
\begin{equation}\nonumber
\begin{aligned} 
\mathcal{I}_3&:= \omt\chi_0^2 \p_k  \mathbf{u}_{\e t} \cdot  \p_k \bigg(\frac{1}{\Theta} g_1 \bigg) dxds\\
&\lesssim  t^{\frac 1 4} \|\chi_0\p_k  \mathbf{u}_{\e t}\|_{L_t^\infty(L^2)} ^{\frac{1}{2}}\|\chi_0\p_k  \mathbf{u}_{\e t} \|_{L_t^2(H^1)} ^{\frac{1}{2}} \|{\Theta}^{-1} \|_{L_{x, t}^\infty}(1+ \|{\Theta} \|_{L_{t}^\infty(H^2)}) [
\|\mathbf{u}\|_{L_t^2(H^3)}\\
&\q \cdot ( \|\chi_0\p_k  \mathbf{u}_{\e t}\|_{L_t^\infty(L^2)} +\|\mathbf{u}_t\|_{L_t^\infty(L^2)} ) +  \|\vartheta_{ t}\|_{L_t^\infty(H^1)} ( \e\|\mathbf{u}_t\|_{L_t^2(H^1)} \\
&\q  + \|\mathbf{u}\|_{L_t^\infty(H^1)}\|\mathbf{u}\|_{L_t^2(H^3)} ) + \|\rho^{-1}\|_{L_{t}^\infty(H^2)} \|\e \sigma_t\|_{L_t^\infty(H^1)}\|\mathbf{u}\|_{L_t^2(H^3)} 
] .
\end{aligned}  
\end{equation} 
Next,  integrating the inner product of $\p_k\ef{2nst}_3$ with $\chi^2_0 \p_k\vartheta_t$  in $L^2(\Omega\times (0, t))$ gives that 
\begin{equation}\label{int3}
\begin{aligned} 
&\frac{1}{2} \|\chi_0 \p_k \vartheta_t\|_{L^2}^2(t)+\frac{1}{\e} \omt \p_k \dv\mathbf{u}_{\e t}\cdot \chi_0^2 \p_k \vartheta_{t} dxds+\kappa\|\chi_0 \sqrt{\rho}^{-1} \p_k\n  \vartheta_t \|_{L_t^2(L^2)} ^2 \\ 
&= \frac{1}{2} \|\chi_0 \p_k \vartheta_t(0)\|_{L^2}^2 + \sum_{i=4}^{6} \mathcal{I}_i, 
\end{aligned}   
\end{equation}   
in which  
\begin{equation}\nonumber
\begin{aligned} 
\mathcal{I}_4&:= \omt \frac{1}{2} \dv(\chi_0^2 \mathbf{u} )| \p_k \vartheta_{t} |^2 - (\p_k\mathbf{u} \cdot \n \vartheta_{t}) \chi^2_0 \p_k\vartheta_t dxds\\
&\lesssim t^{\frac 1 2} \| \p_k \vartheta_{t} \|^2_{L_t^\infty(L^2)} \|\mathbf{u}\|_{L_t^2(H^3)}, 
\end{aligned}  
\end{equation}  
\begin{equation}\nonumber
\begin{aligned} 
\mathcal{I}_5& := \kappa \omt [ \p_k(p^{-1}) \dv(\Theta \n\vartheta_t)+p^{-1}\dv(\p_k \Theta \n\vartheta_t)] \chi_0^2   \p_k \vartheta_{t} - \Theta \p_k \n\vartheta_t\cdot \n(p^{-1}\chi^2 )\p_k \vartheta_t dxds \\
&\lesssim  t^{\frac 1 4}  \| \vartheta_{t} \|_{L_t^\infty(H^1)}^{\frac 1 2} \|  \vartheta_{t} \|_{L_t^2(H^2)}^{\frac 3 2}  \|p^{-1}\|_{L_t^\infty(H^2)} \|\Theta\|_{L_t^\infty(H^2)} , 
\end{aligned}  
\end{equation}  
and
\begin{equation}\nonumber 
\begin{aligned} 
\mathcal{I}_6:& = \omt\p_k h_1 \cdot \chi_0^2   \p_k \vartheta_{t}dxds \\
&\lesssim t^{\frac{1}{4}}\|\vartheta_t\|_{L_t^\infty(H^1)}^{\frac 1 2}  \|\vartheta_t\|_{L_t^2(H^2)}^{\frac 1 2}  [ \|\Theta^{-1} \|_{L_{t}^\infty (H^2) }   \|\mathbf{u}_t\|_{L_t^2(H^1)}  \|\Theta \|_{L_{t}^\infty (H^2) }   \\ 
&\q +\|p^{-1}\|_{L_t^\infty(H^2)}   \|\Theta \|_{L_t^2(H^4)} ^{\frac 1 2}  \|\Theta \|_{L_t^\infty(H^2)} ^{\frac 1 2} (\|\vartheta_t \|_{L_t^2(H^2)} + \|\e\sigma_t \|_{L_t^\infty(H^1)}  )  \\ 
&\q + \e  \|p^{-1}\|_{L_t^\infty(H^2)}  \|\e \sigma_t \|_{L_t^\infty(H^1)} \|\e\mathbf{u}\|_{L_t^\infty(H^2)}\|\mathbf{u}\|_{L_t^2(H^3)}  
\\  
&\q  +   \|p^{-1}\|_{L_t^\infty(H^2)}    \|\chi_0 \n\mathbf{u}_{\e t}\|_{L_t^2(H^1)}  \|\e\mathbf{u}\|_{L_t^\infty(H^2)}].
\end{aligned}  
\end{equation}   
Adding \ef{int1}, \ef{int2}, and \ef{int3}, we obtain
\begin{equation}\label{int4}
\begin{aligned} 
& \|\chi_0 \p_k \eta_t\|_{L^2}^2(t)+  \|\chi_0 \sqrt{\Theta}^{-1} \p_k \mathbf{u}_{\e t}\|_{L^2}^2(t)  +   \|\chi_0 \sqrt{p}^{-1}\p_k \n \mathbf{u}_{\e t}\|_{L_t^2(L^2)}^2 \\
&+ \|\chi_0 \sqrt{p}^{-1}\p_k \dv \mathbf{u}_{\e t}\|_{L_t^2(L^2)}^2 +  \|\chi_0 \p_k \vartheta_t\|_{L^2}^2(t) + \|\chi_0 \sqrt{\rho}^{-1} \p_k\n  \vartheta_t \|_{L_t^2(L^2)}^2  \\
&\leq C_0(M_0){\rm exp}\{(t^{\frac 1 4} + \e^{\frac 1 4} ) C_1( \mathcal{E}(t)) \}
\end{aligned}  
\end{equation}

{\it Step 2. Boundary  estimate}. 
Denote the unknowns in local coordinates as 
$$U (t, y):=\mathbf{u} (t, \Lambda(y)), \q  U_\e (t, y):=\mathbf{u}_\e (t, \Lambda(y)), \q 
S(t, y):=\sigma(t,  \Lambda(y)), $$
$$ L(t, y):=\eta(t, \Lambda(y)),  \q  K(t, y):=\rho(t, \Lambda(y)),  \q Y (t, y):=\vartheta(t, \Lambda(y)).$$
We rewrite the system $\eqref{2nst}$ in $[0, T]\times \widetilde{\Omega}$ as 
\begin{equation}\label{2nstb}
\left\{
\begin{aligned} 
&  L_{tt} +U^j a_{kj} D_k L_t + U^j_t a_{kj} D_k L     +\frac{1}{2 \e} a_{kj} D_k U^j_{\e t}   = 0,   \\ 
&  U^i_{\e tt} + U^j a_{kj} D_k U^i_{\e t}  +\frac{1}{\e} e^{Y}  a_{ki} D_k (2 L_t + Y_t)\\
&\quad - K^{-1} [ \mu a_{kj} D_k (a_{lj} D_l U^i_{\e t} ) 
 +(\mu+\lambda)  a_{ki} D_k (a_{lj} D_l U^j_{\e t} )  ] =   \mathcal{G}_1,   \\ 
& Y_{tt} +  U^j a_{kj} D_k Y_t +\frac{1}{\e}a_{kj} D_k U^j_{\e t} -\kappa  e^{-\e S}  a_{kj} D_k( e^{Y} a_{lj} D_l Y_t ) =  \mathcal{H}_1,
\end{aligned} \right. 
\end{equation}   
where $\mathcal{G}_1$ and $\mathcal{H}_1$ are defined as 
\begin{equation}\nonumber
\begin{aligned} 
\mathcal{G}_1:&= -    U^j_t a_{kj} D_k U^i_{\e }-    Y_t (U^i_{\e t} + U^j a_{kj} D_k U^i_{\e } ) \\
&\quad - K^{-1}  \e S_t [ \mu a_{kj} D_k (a_{lj} D_l U^i_{\e } ) +(\mu+\lambda)  a_{ki} D_k (a_{lj} D_l U^j_{\e } )  ], \\
\mathcal{H}_1:&=- U^j_t a_{kj} D_k  Y +\kappa e^{-\e S} a_{kj} D_k( e^{Y} Y_t a_{lj} D_l Y )-\kappa e^{-\e S} \e S_t  a_{kj} D_k( e^{Y} a_{lj} D_l Y )  
\\ 
&\quad  +\e^2  \p_t \left [ e^{-\e S} \left( \frac{\mu}{2} (a_{ki}D_k U^j+a_{kj}D_k U^i)^2 +  \lambda( a_{kj}D_k U^j)^2\right) \right ].
\end{aligned}  
\end{equation}  
The initial and boundary value conditions are  as follows
\be\label{bd}
\begin{aligned} 
&(U_{\e tt}, S_{tt}, L_{tt}, Y_{tt})|_{t=0}=(U_{\e tt}(0), S_{tt}(0), L_{tt}(0), Y_{tt}(0))(y)  &{\rm in}~\widetilde{\Omega}, \\
&U=0, \quad D_3 Y=0  &{\rm on}~ \partial \widetilde{\Omega}_T.
\end{aligned}   
\ee 

First,  we multiply  $D_\tau \ef{2nstb}_2 $ by $J\chi^2 D_\tau U^i_{\e t}   $ and integrate   over $  \widetilde{\Omega}\times(0, t) $  to get  
 \begin{equation}\label{ttan1}
\begin{aligned}
&\frac{1}{2}  \int_{\widetilde{\Omega}} J\chi^2 |D_\tau U_{\e t} |^2 dy+\frac{1}{\e} \iint_{\widetilde{\Omega}_t} D_\tau (e^{Y}  a_{ki} D_k (2 L_t + Y_t)) J\chi^2 D_\tau U_{\e t}^i dyds  \\ 
& +\mu \iint_{\widetilde{\Omega}_t} K^{-1} J\chi^2  a_{lj} D_{\tau l} U^i_{\e t} a_{kj} D_{\tau k} U^i_{\e t} dyds +(\mu+\lambda)  \iint_{\widetilde{\Omega}_t}   K^{-1} J\chi^2  a_{lj} D_{\tau l} U^j_{\e t} a_{ki} D_{\tau k} U^i_{\e t} dyds \\
&= \frac{1}{2}  \int_{\widetilde{\Omega}} J\chi^2 |D_\tau U_{\e t} (0)|^2 dy+ \sum_{i=1}^4 \mathcal{R}_i, 
\end{aligned} 
 \end{equation} 
where  the low-order terms $\mathcal{R}_1$ to $\mathcal{R}_4$ are defined and controlled in what follows. Integrating by parts and using $\chi^2 D_\tau U_{\e t} ^i|_{\p\tilde{\Omega}}=0$ gives
 \begin{equation}\nonumber
\begin{aligned}
  \mathcal{R}_1 & := -\iint_{\widetilde{\Omega}_t} J\chi^2 D_\tau U_{\e t} ^i D_\tau (U^j a_{kj}  D_k U^i_{\e t}) dy ds \\
&=  \iint_{\widetilde{\Omega}_t}  D_{\tau }( J\chi^2 D_\tau U_{\e t} ^i ) (U^j a_{kj}  D_k U^i_{\e t}) dyds
\\ 
&\lesssim  \e^{\frac 1 4}   \|D_{\tau } \mathbf{u}_{\e t}\|_{L_t^2(H^1)}   \|\e \mathbf{u} \|_{L_t^\infty(H^2)}^{\frac 3 4} \| \mathbf{u} \|_{L_t^\infty(H^1)}^{\frac 1 4}   \|\n \mathbf{u}_{ t}\|_{L_t^2(L^2)}. 
\end{aligned} 
\end{equation} 
And  $\mathcal{R}_2$ is defined as  
 \begin{equation}\nonumber  
\begin{aligned}
\mathcal{R}_2:&
= \mu \iint_{\widetilde{\Omega}_t} J\chi^2 D_\tau U_{\e t} ^i  D_\tau (e^{ Y-\e S }  a_{kj}  ) D_k (a_{lj} D_l  U_{\e t} ^i ) dyds \\  
&\q+\mu \iint_{\widetilde{\Omega}_t} J\chi^2 D_\tau U_{\e t} ^i  e^{ Y-\e S } a_{kj}  D_k ( D_\tau a_{lj} D_l  U_{\e t} ^i ) dyds \\
&\quad -\mu  \iint_{\widetilde{\Omega}_t} J a_{kj} D_k (e^{ Y-\e S } \chi^2  ) ( D_\tau U_{\e t}^i ) a_{lj} D_{l\tau}  U_{\e t} ^i dyds \\  
&=:\mathcal{R}_{21}+\mathcal{R}_{22}+\mathcal{R}_{23}.
\end{aligned}     
\end{equation}    
By integration by parts and $\chi^2 D_\tau U_{\e t} ^i |_{\p\widetilde{\Omega}}=0 $,  we verify  
\begin{align}
 \mathcal{R}_{21} &=  \mu \iint_{\widetilde{\Omega}_t} J\chi^2 D_\tau U_{\e t} ^i  D_\tau (e^{ Y-\e S } a_{kj}  ) D_k (a_{lj} D_l  U_{\e t} ^i ) dy ds \nonumber \\
&=-\mu \iint_{\widetilde{\Omega}_t} J\chi^2  [ D_{k\tau} U_{\e t} ^i  D_\tau (e^{ Y-\e S } a_{kj}  ) + D_{\tau} U_{\e t} ^i  D_{k\tau} (e^{ Y-\e S } a_{kj}  )] (a_{lj} D_l  U_{\e t} ^i ) dyds  \nonumber \\ 
&\lesssim 
\|D_{k\tau} U_{\e t}\|_{L_t^2(L^2)} [ \e^{\frac 1 4}      \|\e \Theta \|_{L_t^\infty(H^3)}^{\frac 3 4}\| \Theta \|_{L_t^\infty(H^2)}^{\frac 1 4}   \|\n \mathbf{u}_{ t}\|_{L_t^2(L^2)}+\e  \|\rho^{-1} \|_{L_{x, t}^\infty}  \|\n \mathbf{u}_{ t}\|_{L_t^2(L^2)} \nonumber \\ 
&\q + \e  \|\rho^{-1} \|_{L_{x, t}^\infty} \|\sigma\|_{L_t^\infty(H^2)} \| \mathbf{u}_{ t}\|_{L_t^2(H^1)}^{\frac 1 2} \|\e^2 \mathbf{u}_{ t}\|_{L_t^2(H^2)}^{\frac 1 2} ] + \e \|D_{\tau} U_{\e t}\|_{L_t^\infty(L^2)} \|\rho^{-1} \|_{L_{x, t}^\infty} \|\n \mathbf{u}_{ t}\|_{L_t^2(L^2)} \nonumber \\ 
&\q  + \|D_{\tau} U_{\e t}\|_{L_t^2(H^1)} \|\n \mathbf{u}_{ t}\|_{L_t^2(L^2)}[\e^{\frac 1 2}   \|\e \Theta \|_{L_t^\infty(H^3)}^{\frac 1 2}\| \Theta \|_{L_t^\infty(H^2)}^{\frac 1 2} +\e \| \Theta \|_{L_t^\infty(H^2)} \| p^{-1}\|_{L_t^\infty(H^2)} ] \nonumber \\  
&\q  + t^{\frac 1 4} \|D_{\tau} U_{\e t}\|_{L_t^\infty(L^2)}^{\frac 1 2}  \|D_{\tau} U_{\e t}\|_{L_t^2(H^1)}^{\frac 1 2}  \|\e^2 \mathbf{u}_{ t}\|_{L_t^2(H^2)} \|\rho^{-1} \|_{L_{x, t}^\infty} ( \|\n \sigma\|_{L_t^\infty(H^1)} +\e \| \n \sigma\|_{L_t^\infty(H^1)}^2) ,  \nonumber
\end{align}      
and  
 \begin{equation}\nonumber 
\begin{aligned} 
  \mathcal{R}_{22} & :=\mu \iint_{\widetilde{\Omega}_t} J\chi^2 D_\tau U_{\e t} ^i  e^{ Y-\e S } a_{kj}  D_k ( D_\tau a_{lj} D_l  U_{\e t} ^i ) dy ds\\
&= -\mu \iint_{\widetilde{\Omega}_t} J\chi^2 [D_{k\tau} U_{\e t} ^i  e^{ Y-\e S } a_{kj} +D_{\tau} U_{\e t} ^i D_k(  e^{ Y-\e S } a_{kj} )]( D_\tau a_{lj} D_l  U_{\e t} ^i ) dy ds\\
&\lesssim  \e \|D_{k\tau} U_{\e t}\|_{L_t^2(L^2)} \|\rho^{-1}\|_{L_{x, t}^\infty}  \|\n \mathbf{u}_{ t}\|_{L_t^2(L^2)}\\
&\q + \e t^{\frac 1 4}\|\rho^{-1}\|_{L_{t}^\infty(H^2)} \|\n \mathbf{u}_{ t}\|_{L_t^2(L^2)}\|D_{\tau} U_{\e t}\|_{L_t^\infty(L^2)}^{\frac 1 2}  \|D_{\tau} U_{\e t}\|_{L_t^2(H^1)}^{\frac 1 2}.
\end{aligned}     
\end{equation}  
In addition, we can estimate $\mathcal{R}_{23}$ in a similar way to  $\mathcal{R}_{21}$ to show that  
 \begin{equation}\nonumber 
\begin{aligned} 
  \mathcal{R}_{23} \leq (\e^{\frac 1 4}+t^{\frac 1 4}) C_1(\mathcal{E}(t)).
\end{aligned}     
\end{equation}   
Then $\mathcal{R}_3$ and  $\mathcal{R}_4$ defined and estimated by integration by parts and $\chi^2 D_\tau U_{\e t} ^i |_{\p\tilde{\Omega}}=0$, 
 \begin{equation}\nonumber 
\begin{aligned} 
 \mathcal{R}_3 & := (\mu+\lambda) \iint_{\widetilde{\Omega}_t} J\chi^2 D_\tau U_{\e t} ^i [ D_\tau (K^{-1}a_{ki}  ) D_k (a_{lj} D_l  U_{\e t} ^j )+K^{-1}a_{ki}  D_k ( D_\tau a_{lj} D_l  U_{\e t} ^j ) ] \\
&\quad -  J a_{ki} D_k (K^{-1}\chi^2  ) ( D_\tau U_{\e t}^i ) a_{lj} D_{l\tau}  U_{\e t} ^j dy ds\\ 
&= (\mu+\lambda)\iint_{\widetilde{\Omega}_t} J\chi^2 D_\tau U_{\e t} ^i  D_\tau (K^{-1}a_{ki}  ) D_k (a_{lj} D_l  U_{\e t} ^j ) \\ 
&\quad - J a_{ki}  D_k ( \chi^2 D_\tau U_{\e t} ^i  K^{-1})  D_\tau a_{lj} D_l  U_{\e t} ^j   -  J a_{ki} D_k (K^{-1}\chi^2  ) ( D_\tau U_{\e t}^i ) a_{lj} D_{l\tau}  U_{\e t} ^j dyds \\
&\lesssim  t^{\frac 1 4}\|\dv \mathbf{u}_{\e t}\|_{L_t^2(H^1) }\|\rho^{-1}\|_{L^\infty_{x, t}}^2  \|\rho\|_{L^\infty_{t}(H^2)} \|D_\tau U_{\e t}\|_{L^\infty_{t}(L^2)} ^{\frac{1}{2}}  \|D_\tau U_{\e t}\|_{L^2_{t}(H^1)} ^{\frac{1}{2}} +\e \| \mathbf{u}_{ t}\|_{L^2_{t}(H^1)}\\ 
&\q \cdot  ( \|\rho^{-1}\|_{L^\infty_{x, t}}\|\dv \mathbf{u}_{\e t}\|_{L_t^2(H^1)} +\|\rho^{-1}\|_{L^\infty_{x, t}}^2\|\rho \|_{L^\infty_{t}(H^2)}  \|D_\tau U_{\e t}\|_{L^\infty_{t}(L^2)} ^{\frac{1}{2}}  \|D_\tau U_{\e t}\|_{L^2_{t}(H^1)}  ^{\frac{1}{2}}) , 
\end{aligned}    
\end{equation} 
 \begin{equation}\nonumber  
\begin{aligned} 
 \mathcal{R}_4 & := \iint_{\widetilde{\Omega}_t} J\chi^2 D_\tau U_{\e t} ^i  D_\tau \mathcal{G}_1 dyds\\
 &\leq t^{\frac 1 4} \|D_\tau U_{\e t}\|_{L_t^\infty(L^2)} ^{\frac{1}{2}}\|D_\tau U_{\e t}\|_{L_t^2(H^1)} ^{\frac{1}{2}}[ \|\mathbf{u}\|_{L_t^2(H^3)}( \|D_\tau U_{\e t}\|_{L_t^\infty(L^2)} +\|\mathbf{u}_t\|_{L_t^\infty(L^2)} )+ \|\vartheta_{ t}\|_{L_t^\infty(H^1)}\\
&\q  \cdot( \e\|\mathbf{u}_t\|_{L_t^2(H^1)} + \|\mathbf{u}\|_{L_t^\infty(H^1)}\|\mathbf{u}\|_{L_t^2(H^3)} )+ \|\rho^{-1}\|_{L_{t}^\infty(H^2)} \|\e \sigma_t\|_{L_t^\infty(H^1)}\|\mathbf{u}\|_{L_t^2(H^3)} ] .
\end{aligned}  
\end{equation}

Taking the inner products of $D_\tau \ef{2nstb}_1 $ with $4 e^Y J\chi^2 D_\tau L_{ t}   $,   and $D_\tau \ef{2nstb}_3 $ with $e^Y J\chi^2 D_\tau Y_{ t}  $  in $L^2( \widetilde{\Omega}\times(0, t)) $, respectively,  and adding the resulting equalities, we obtain 
 \begin{equation}\label{ttan2} 
\begin{aligned} 
& \int_{\widetilde{\Omega}} 2 J\chi^2 e^Y |D_\tau L_{ t} |^2 +\frac{1}{2}   J\chi^2 e^Y  |D_\tau Y_{ t} |^2 dy + \kappa \iint_{\widetilde{\Omega}_t} e^{2Y-\e S} J\chi^2 a_{lj} D_{l\tau} Y_t  a_{kj} D_{k\tau} Y_t dyds \\  
& + \frac{1}{\e} \iint_{\widetilde{\Omega}_t} e^Y J\chi^2 D_\tau (2 L_t + Y_t ) D_\tau (  a_{kj} D_k U_{\e t}^j) 
 dyds \\
&= \int_{\widetilde{\Omega}} 2 J\chi^2 e^{Y_0} |D_\tau L_{t} (0)|^2 +\frac{1}{2}   J\chi^2 e^{Y_0}  |D_\tau Y_{t} (0)|^2 dy + \sum_{i=5}^8 \mathcal{R}_i,  
\end{aligned} 
\end{equation}  
where   
 \begin{equation}\nonumber  
\begin{aligned}  
\mathcal{R}_5:& = 2 \iint_{\widetilde{\Omega}_t} (\p_t(e^Y J\chi^2 +D_k ( e^Y J\chi^2  U^j a_{kj} )) 
|D_\tau L_t|^2 dyds - 4 \iint_{\widetilde{\Omega}} e^Y J\chi^2 D_\tau L_t   \\ 
&\q  \cdot [ D_\tau( U ^j a_{kj})   D_k L_t + \frac{1}{2}D_\tau  (U_t ^j a_{kj} )D_k ( \e S-Y  )   + \frac{1}{2}  U_t ^j  a_{kj} D_{k\tau}  ( \e S-Y ) ) ] dyds \\ 
&  
 \lesssim  t^{\frac 1 2}  \|D_\tau L_t\|_{L_t^\infty(L^2)}^2 (\|\vartheta_t \|_{L_t^2(H^2)}\|\Theta \|_{L^\infty_{x, t}} +\|\Theta \|_{L_t^\infty(H^2)}^ {\frac 1 2} \|\Theta \|_{L_t^2(H^4)}^ {\frac 1 2} \|\mathbf{u}\|_{L_t^2(H^3)}^ {\frac 1 2}\|\mathbf{u}\|_{L_t^\infty(H^1)}^ {\frac 1 2}  \\
&\q  +\|\Theta \|_{L^\infty_{x, t}}\| \mathbf{u}\|_{L_t^2(H^3)}) + \|\Theta \|_{L^\infty_{x, t}} \|D_\tau L_t\|_{L_t^\infty(L^2)} [ t^{\frac 1 2} \|\mathbf{u}\|_{L_t^2(H^3)} \|D_k L_t\|_{L_t^\infty(L^2)}    \\ 
 &\q  
 +\left \|\frac{\n \sigma}{\e}\right \|_{L_t^2(H^1)} (\e   \| \mathbf{u}_t\|_{L_t^2(H^1)}^{\frac 1 2} \|\e^2  \mathbf{u}_t\|_{L_t^2(H^2)}^{\frac 1 2}+\e^{\frac 1 2}   \| \mathbf{u}_t\|_{L_t^2(H^1)}^{\frac 1 4} \|\e^2 \mathbf{u}_t\|_{L_t^2(H^2)} ^{\frac 3 4} )] \\ 
&\q  + t^{\frac 1 4} \|D_\tau L_t\|_{L_t^\infty(L^2)} \|\mathbf{u}_t\|_{L_t^2(H^1)} \|\Theta\|_{L_t^\infty(H^2)} ^{\frac 1 2}\|\Theta\|_{L_t^2(H^4)} ^{\frac 1 2}  , 
\end{aligned} 
\end{equation}  
\begin{equation}\nonumber 
\begin{aligned} 
\mathcal{R}_6:&=\frac{1}{2}\iint_{\widetilde{\Omega}_t } \p_t(e^Y J\chi^2 )|D_\tau Y_t|^2 dyds 
+\frac{1}{2}\iint_{\widetilde{\Omega}_t }  J a_{kj}  D_k  ( e^Y\chi^2  U^j ) |D_\tau Y_t|^2 dyds \\
&\quad 
- \iint_{ \widetilde{\Omega}_t } D_k (   U^j a_{kj} D_k Y_t ) e^Y J\chi^2 D_\tau Y_t dyds  \\
& \lesssim  t^ {\frac 1 2}   \|D_\tau  Y_t\|_{L_t^\infty(L^2)}^2 ( \|\Theta\|_{L_{x, t}^\infty}\|\vartheta_t \|_{L_t^2(H^2)} +\|\Theta \|_{L_t^\infty(H^2)}^ {\frac 1 2} \|\Theta \|_{L_t^2(H^4)}^ {\frac 1 2} \|\mathbf{u}\|_{L_t^2(H^3)}^ {\frac 1 2}\|\mathbf{u}\|_{L_t^\infty(H^1)}^ {\frac 1 2}  \\
&\q  +\|\Theta\|_{L_{x, t}^\infty}\| \mathbf{u}\|_{L_t^2(H^3)}) + C \|\Theta\|_{L_{x, t}^\infty}  \|D_\tau Y_t\|_{L_t^\infty(L^2)} ( t^ {\frac 1 2} \| \mathbf{u}\|_{L_t^2(H^3)} \|  \vartheta_t\|_{L_t^\infty(H^1)}\\ 
&\q + t^{\frac 1 4}\|\mathbf{u}\|_{L_t^2(H^3)}^ {\frac 1 2}\|\mathbf{u}\|_{L_t^\infty(H^1)}^ {\frac 1 2}  \|\vartheta_t\|_{L_t^2(H^2)} ),  
\end{aligned}
\end{equation}  
\begin{equation}\nonumber 
\begin{aligned} 
\mathcal{R}_7:&=\kappa \iint_{\widetilde{\Omega}_t}  e^Y a_{lj} D_{l\tau } Y_t J a_{kj} D_{k}(e^{Y-\e S} \chi^2 ) D_{\tau } Y_t dyds + \kappa \iint_{\widetilde{\Omega}_t}  e^Y J \chi^2 D_{\tau } Y_t \\ 
&\quad \cdot [D_{\tau }(e^{-\e S} a_{kj} )D_k(  e^Y a_{lj} D_{l } Y_t  )+ e^{-\e S}a_{kj} D_k ( D_{\tau }(e^Y a_{lj})D_l Y_t )] dyds 
\\  
&\lesssim  t^{\frac 1 4}    \|\vartheta_t\|_{L_t^\infty(H^1)}^{\frac 1 2}\|\vartheta_t\|_{L_t^2(H^2)}^{\frac 3 2} (\|\Theta\|_{L_t^\infty(H^2)} +\|p^{-1}\|_{L_t^\infty(H^2)}  )(1+\|\Theta\|_{L_t^\infty(H^2)}  ), 
\end{aligned} 
\end{equation}

 \begin{equation}\nonumber
\begin{aligned} 
\mathcal{R}_8:&= \iint_{\widetilde{\Omega}_t}    e^Y J \chi^2 D_{\tau } Y_t D_{\tau } \mathcal{H}_1  dyds
\\ 
&\lesssim t^{\frac{1}{4}}  \|D_{\tau } Y_t\|_{L_t^\infty(L^2)}^{\frac 1 2}  \|D_{\tau } Y_t\|_{L_t^2(H^1)}^{\frac 1 2}  [ \|\Theta^{-1} \|_{L_{t}^\infty (H^2) }   \|\mathbf{u}_t\|_{L_t^2(H^1)}  \|\Theta \|_{L_{t}^\infty (H^2) }   \\ 
&\q +\|p^{-1}\|_{L_t^\infty(H^2)}   \|\Theta \|_{L_t^2(H^4)} ^{\frac 1 2}  \|\Theta \|_{L_t^\infty(H^2)} ^{\frac 1 2} (\|\vartheta_t \|_{L_t^2(H^2)} + \|\e\sigma_t \|_{L_t^\infty(H^1)}  )  \\ 
&\q + \e  \|p^{-1}\|_{L_t^\infty(H^2)}  \|\e \sigma_t \|_{L_t^\infty(H^1)} \|\e\mathbf{u}\|_{L_t^\infty(H^2)}\|\mathbf{u}\|_{L_t^2(H^3)}  ]
\\ 
&\q + \e^{\frac{1}{2}}\|\rho^{-1}\|_{L_{x, t}^\infty }  \|D_{\tau } Y_t\|_{L_t^2(H^1)} \|D_\tau  \mathbf{u}_{\e t}\|_{L_t^2(H^1)} \|\mathbf{u}\|_{L_t^\infty(H^1)}^{\frac 1 2}\|\e \mathbf{u}\|_{L_t^\infty(H^2)}^{\frac 1 2}   .
\end{aligned}   
\end{equation}

We then combine the large terms in \eqref{ttan1} and \eqref{ttan2} so that their leading-order contributions cancel, 
\begin{align} \nonumber
&\frac{1}{\e} \iint_{\widetilde{\Omega}_t}  e^Y J\chi^2 D_\tau (2 L_t + Y_t ) D_\tau (  a_{kj} D_k U_{\e t}^j)  dyds \nonumber\\
& +\frac{1}{\e} \iint_{\widetilde{\Omega}_t}  D_\tau (e^{Y}  a_{ki} D_k (2 L_t + Y_t)) J\chi^2 D_\tau U_{\e t}^i dyds \nonumber\\ 
 & =  - \frac{1}{\e} \iint_{\widetilde{\Omega}_t}  
 Ja_{kj} D_k ( e^Y  \chi^2 D_\tau (2 L_t + Y_t )  ) 
  D_\tau   U_{\e t}^j  +   e^Y J\chi^2 D_\tau (2 L_t + Y_t )  D_\tau  a_{kj} D_k  U_{\e t}^j dyds \nonumber \\ 
 &\q+\frac{1}{\e}  \iint_{\widetilde{\Omega}_t}  D_\tau (e^{Y}  a_{kj} D_k (2 L_t + Y_t)) J\chi^2 D_\tau U_{\e t}^j dyds  \nonumber \\
 &\leq C t^{\frac 1 4} \|\n\mathbf{u}_t\|_{L_t^2(L^2)} (\|\Theta\|_{L^\infty _{x, t}}+\|\Theta\|_{L_t^\infty(H^2)} ^{\frac 1 2} \|\Theta\|_{L_t^2(H^4)}^{\frac 1 2})( \|\n \eta_t\|_{L_t^\infty(L^2)} +\|\n \vartheta_t\|_{L_t^\infty(L^2)} ) .  \nonumber
\end{align}

Since the matrix $\{\sum_k a_{ki}a_{kj}\}$ is   positive definite,  we have 
$$ \sum_{i, j}a_{lj}D_{l \tau} U_{\e t}^i a_{kj}D_{k \tau} U_{\e t}^i\geq C |D_{y\tau} U_{\e t}|^2,   $$
for some positive constant $C$. 
Therefore, adding \ef{ttan1} and \ef{ttan2},  using the estimates of $\mathcal{R}_1 $ to $\mathcal{R}_8 $ and \ef{lemma1},  we derive that 
 \begin{equation}\label{ttan3}
\begin{aligned} 
&  \int_{\widetilde{\Omega}} J\chi^2 |D_\tau U_{\e t} |^2 dy(t )+\int_{\widetilde{\Omega}}  J\chi^2  |D_\tau L_{ t} |^2 +   J\chi^2  |D_\tau Y_{ t} |^2 dy(t) \\ 
& + \iint_{\widetilde{\Omega}_t}   J\chi^2   |D_{y \tau } U_{\e t} |^2 dyds  + \iint_{\widetilde{\Omega}_t}   J\chi^2   |D_{y \tau} Y_t |^2dyds \\  
&\leq C_0(M_0)+(\e^{\frac 1 4}+t^{\frac 1 4}) C_1( \mathcal{E}(t)).
\end{aligned}  
\end{equation}

Now,  we turn to the normal derivatives estimation.  Multiplying $\ef{2nstb}_2$ by $a_{3i}$ gives 
\begin{equation} \label{d33}
\begin{aligned} 
& - a_{3i} U^i_{\e tt} +(2\mu+\lambda) K^{-1}   D_3 (a_{lj} D_l U^j_{\e t} ) 
  -\frac{1}{\e} e^{Y}   D_3 (2 L_t + Y_t)\\
& =  \mu K^{-1} [   D_3 (a_{lj} D_l U^j_{\e t} )   - a_{3i} a_{kj} D_k (a_{lj} D_l U^i_{\e t} ) 
 ] +a_{3i} U^j a_{kj} D_k U^i_{\e t} -a_{3i}   \mathcal{G}_1.  \\ 
\end{aligned}  
\end{equation}  
Notice that 
\begin{equation} \nonumber 
\begin{aligned} 
&\mu K^{-1} [   D_3 (a_{lj} D_l U^j_{\e t} )   - a_{3i} a_{kj} D_k (a_{lj} D_l U^i_{\e t} ) 
 ]\\
 &=\mu  K^{-1} [ D_3a_{3j}D_3U^i_{\e t} + D_3a_{\tau j}D_\tau U^i_{\e t} + a_{\tau j}D_{3\tau}U^i_{\e t} - a_{3j}D_3 a_{3j} a_{3i}D_3 U^i_{\e t} \\
& - a_{\tau j}a_{3i}D_\tau a_{ij}D_i U^i_{\e t} - a_{\tau j}a_{\xi j}a_{3i}D_{\tau \xi} U^i_{\e t} - a_{3j}a_{3i}D_3a_{\tau j}D_\tau U^i_{\e t} ] \q {\rm  for }\q  \xi,  \tau = 1,  2,  
\end{aligned}  
\end{equation}  
which suggests that the right-hand side of \ef{d33} does not contain $D_{33}U_{\e t}$.

By multiplying \ef{d33} by $J\chi^2 D_3( a_{lj} D_l U^i_{\e t})$ and integrating the resulting equality in $L^2(\tilde{\Omega}\times(0, t))$,  we obtain  
\begin{equation} \label{d331}
\begin{aligned} 
&\frac{1}{2} \int_{\widetilde{\Omega}} J\chi^2 |a_{3i}D_3 U^i_{\e t} |^2 dy +(2\mu+\lambda) \iint_{\widetilde{\Omega}_t} J\chi^2 K^{-1} |D_3 (a_{lj} D_l U^j_{\e t} )|^2 dyds \\
&-\frac{1}{\e} \iint_{\widetilde{\Omega}_t} e^Y D_3 (2 L_t + Y_t) J\chi^2  D_3 (a_{lj} D_l U^j_{\e t} )dyds \\
& =: \frac{1}{2} \int_{\widetilde{\Omega}} J\chi^2 |a_{3i}D_3 U^i_{\e t}(0) |^2 dy +  \sum_{i=9}^{11} \mathcal{R}_i,  
\end{aligned}  
\end{equation} 
where $\mathcal{R}_9$ to $\mathcal{R}_{11}$ are defined and estimated as follows
\begin{equation} \nonumber 
\begin{aligned} 
\mathcal{R}_9 &:= - \iint_{\widetilde{\Omega}_t}J\chi^2  (a_{3i}D_3 U^i_{\e tt} )(a_{\tau j}D_\tau U^j_{\e t})  dy ds - \iint_{\widetilde{\Omega}_t}  J a_{3i} D_3(\chi^2 )  U^i_{\e tt} a_{l j}D_l U^j_{\e t}  dyds\\ 
&\leq C \|\e\mathbf{u}_{\e tt}\|_{L_t^2( H^1)}\|\n \mathbf{u}_{t}\|_{L_t^2( L^2)} , 
\end{aligned}  
\end{equation} 
\begin{equation} \nonumber 
\begin{aligned} 
\mathcal{R}_{10} &:= \iint_{\widetilde{\Omega}_t}  J\chi^2 D_3 (a_{lj} D_l U^j_{\e t} )[K^{-1} (   O(D_{y\tau } U_{\e t})+  O(D_y U_{\e t} ) ) +a_{3i} U^j a_{kj} D_k U^i_{\e t}  ] dy ds\\
&\leq  \delta \iint_{\widetilde{\Omega}_t} J\chi^2 |D_3 (a_{lj} D_l U^j_{\e t} )|^2 dyds +C_\delta \|\rho^{-1}\|_{L^\infty_{x, t}}^2 (  \| D_{y\tau } U_{\e t}\|^2_{L_t^2(L^2)} +\e^2 \| \n \mathbf{u}_{t}\|^2_{L_t^2(L^2)} ) \\
&\q + C \e^{\frac 1 2}  \| \n \mathbf{u}_{t}\|^2_{L_t^2(L^2)} \| \e \mathbf{u} \|^{\frac 3 2}_{L_t^\infty (H^2)} \| \mathbf{u} \|^{\frac 1 2}_{L_t^\infty (H^1)}, 
\end{aligned}  
\end{equation}  
\begin{equation} \nonumber 
\begin{aligned} 
\mathcal{R}_{11}& := - \iint_{\widetilde{\Omega}_t}  J\chi^2 D_3 (a_{lj} D_l U^j_{\e t} )  a_{3i}   \mathcal{G}_1 dy ds\\
&\leq  \delta \iint_{\widetilde{\Omega}_t} J\chi^2 |D_3 (a_{lj} D_l U^j_{\e t} )|^2 dyds +C_\delta  [\e^2  \| \mathbf{u}_{t}\|^2_{L_t^\infty(L^2)}  \| \mathbf{u}\|^2_{L_t^2(H^3)}  + \| \vartheta_t\|^2_{L_t^2(H^2)}\\
&\q \cdot ( \e^2  \| \mathbf{u}_{t}\|^2_{L_t^\infty(L^2)}  + \e \| \mathbf{u}\|^3_{L_t^\infty(H^1)}  \| \e\mathbf{u}\|_{L_t^\infty(H^2)})+ \e^2\|\rho^{-1}\|_{L^\infty_{x, t}}^2 \|\e \sigma_t\| _{L_t^\infty(H^1)}^2  \| \mathbf{u}\|_{L_t^2(H^3)}^2 ] . 
\end{aligned}  
\end{equation} 
Substituting these inequalities into \ef{d331} and choosing $\delta$ sufficiently small,   we conclude that
\begin{equation} \nonumber
\begin{aligned} 
&\frac{1}{2} \int_{\widetilde{\Omega}} J\chi^2 |a_{3i}D_3 U^i_{\e t} |^2 dy +(2\mu+\lambda) \iint_{\widetilde{\Omega}_t} J\chi^2 K^{-1} |D_3 (a_{lj} D_l U^j_{\e t} )|^2 dyds \\
&-\frac{1}{\e} \iint_{\widetilde{\Omega}_t} e^Y D_3 (2 L_t + Y_t) J\chi^2  D_3 (a_{lj} D_l U^j_{\e t} )dyds \\
&\leq C_0(M_0) + \delta \|\mathbf{u}_{t}\|_{L_t^2( H^1)}^2+C_\delta\|\e\mathbf{u}_{\e tt}\|_{L_t^2( H^1)}^2     +    C \|\rho^{-1}\|_{L^\infty_{x, t}}^2 \| D_{y\tau } U_{\e t}\|^2_{L_t^2(L^2)} + t^{\frac 1 2}C_1 (\mathcal{E}(t)). 
\end{aligned}   
\end{equation}  

To remove the singular terms in \ef{d331},  we take the inner products of   $D_3 \ef{2nstb}_1 $ with   $4 e^Y J\chi^2 D_3 L_{ t}   $,   and  $D_3 \ef{2nstb}_3 $ with  $e^Y J\chi^2 D_3 Y_{ t}  $  in $L^2( \widetilde{\Omega}\times (0, t))$. 
Then adding these  two equalities yields that 
\begin{equation} \label{d332}
\begin{aligned} 
& \int_{\widetilde{\Omega}} 2 J\chi^2 e^Y | D_3 L_{t} |^2  + \frac{1}{2}  J\chi^2 e^Y | D_3 Y_{t} |^2 dy  +\kappa \iint_{\widetilde{\Omega}_t} e^{2Y-\e S } J\chi^2 a_{lj} D_{l3} Y_t a_{kj} D_{k3} Y_t dyds  \\
&+ \frac{1}{\e} \iint_{\widetilde{\Omega}_t} e^Y J\chi^2 D_3 (2 L_t + Y_t)   D_3 (a_{lj} D_l U^j_{\e t} )dyds \\
&=\int_{\widetilde{\Omega}} 2 J\chi^2 e^{Y_0} | D_3 L_{t}(0) |^2  + \frac{1}{2}  J\chi^2 e^{Y_0} | D_3 Y_{t}(0) |^2 dy +  \sum_{i=12}^{15} \mathcal{R}_i .
\end{aligned}   
\end{equation}  
Here,  we estimate the low-order terms on the right-hand side of \ef{d332} by use of  the H\"{o}lder  inequality and the interpolation inequalities 
\begin{align} 
\mathcal{R}_{12}& := 2 \iint_{\widetilde{\Omega}_t}  J(\chi^2 \p_t e^Y +a_{kj} D_k (\chi^2 e^Y U^j ) )| D_3 L_{t} |^2 dyds- 4 \iint_{\widetilde{\Omega}_t} e^Y J\chi^2 D_3 L_{ t} \nonumber \\
&\q \cdot [ D_3(U^j a_{kj}) D_k L_{t} + \frac{1}{2}D_3(U^j_{t}  a_{kj}) D_k (\e S-Y )+\frac{1}{2} U^j_{t}  a_{kj}D_{3k} (\e S-Y ) ]dyds   \nonumber \\
&\lesssim  t^{\frac 1 2} \|D_3 L_t\|_{L_t^\infty(L^2)}^2 (\|\Theta \|_{L^\infty_{x, t}} \|\vartheta_t \|_{L_t^2(H^2)} +\|\Theta \|_{L_t^\infty(H^2)}^ {\frac 1 2} \|\Theta \|_{L_t^2(H^4)}^ {\frac 1 2} \|\mathbf{u}\|_{L_t^2(H^3)}^ {\frac 1 2}\|\mathbf{u}\|_{L_t^\infty(H^1)}^ {\frac 1 2} \nonumber \\
&\q  + \|\Theta \|_{L^\infty_{x, t}} \| \mathbf{u}\|_{L_t^2(H^3)}) + \|\Theta \|_{L^\infty_{x, t}} \|D_3 L_t\|_{L_t^\infty(L^2)} [ t^{\frac 1 2} \|\mathbf{u}\|_{L_t^2(H^3)} \|D_k L_t\|_{L_t^\infty(L^2)}  \nonumber  \\
 &\q  
 +\left \|\frac{\n \sigma}{\e}\right \|_{L_t^2(H^1)} (\e   \| \mathbf{u}_t\|_{L_t^2(H^1)}^{\frac 1 2} \|\e^2  \mathbf{u}_t\|_{L_t^2(H^2)}^{\frac 1 2}+\e^{\frac 1 2}   \| \mathbf{u}_t\|_{L_t^2(H^1)}^{\frac 1 4} \|\e^2 \mathbf{u}_t\|_{L_t^2(H^2)} ^{\frac 3 4} )] \nonumber \\ 
&\q  + t^{\frac 1 4} \|\Theta \|_{L^\infty_{x, t}}  \|D_3 L_t\|_{L_t^\infty(L^2)}  \|\mathbf{u}_t\|_{L_t^2(H^1)} \|\vartheta\|_{L_t^\infty(H^2)} ^{\frac 1 2}\|\vartheta\|_{L_t^2(H^4)} ^{\frac 1 2}  ,\nonumber  
\end{align}   
 \begin{equation} \nonumber 
\begin{aligned} 
\mathcal{R}_{13} :& = - \kappa \iint_{\widetilde{\Omega}_t} 
J a_{kj} D_k (e^{Y-\e S } \chi^2 )( D_{3} Y_t ) e^Y a_{lj} D_{l3} Y_t dyds  + \kappa \iint_{\widetilde{\Omega}_t}  J \chi^2 D_{3} Y_t \\ 
&\q \cdot [  D_3 (e^{Y-\e S }a_{kj} )D_k(e^Y a_{lj} D_{l} Y_t )+e^{Y-\e S }a_{kj}D_k( D_3(e^Y a_{lj}) D_{l} Y_t) ]dyds \\
&\lesssim t^{\frac 1 4} \|\vartheta_t\|_{L_t^\infty(H^1)}^{\frac 1 2}\|\vartheta_t\|_{L_t^2(H^2)}^{\frac 3 2} (\|\Theta\|_{L_t^\infty(H^2)} + \|p^{-1}\|_{L_t^\infty(H^2)}  )(1+\|\Theta\|_{L_t^\infty(H^2)}  ), 
\end{aligned}  
\end{equation}  
\begin{equation}\nonumber
\begin{aligned} 
\mathcal{R}_{14}:&=\frac{1}{2}\iint_{\widetilde{\Omega}_t} J(\chi^2 \p_t e^Y +  a_{kj}  D_k  ( e^Y\chi^2  U^j ) )|D_3 Y_t|^2 dy ds 
  \\ 
&\quad 
-\iint_{\widetilde{\Omega}_t}  D_3 (   U^j a_{kj}) D_k Y_t  e^Y J\chi^2 D_3 Y_t dy ds \\
&\lesssim t^ {\frac 1 2} \|D_3 Y_t\|_{L_t^\infty(L^2)}^2 (\|\vartheta_t \|_{L_t^2(H^2)} \|\Theta \|_{L_{x, t}^\infty}  +\|\Theta \|_{L_t^\infty(H^2)}^ {\frac 1 2} \|\Theta \|_{L_t^2(H^4)}^ {\frac 1 2} \|\mathbf{u}\|_{L_t^2(H^3)}^ {\frac 1 2}\|\mathbf{u}\|_{L_t^\infty(H^1)}^ {\frac 1 2}  \\
&\q  +\| \mathbf{u}\|_{L_t^2(H^3)}\|\Theta \|_{L_{x, t}^\infty} ) + C \|\Theta \|_{L_{x, t}^\infty} \|D_3 Y_t\|_{L_t^\infty(L^2)} ( t^ {\frac 1 2} \| \mathbf{u}\|_{L_t^2(H^3)} \|  \vartheta_t\|_{L_t^\infty(H^1)}\\
&\q + t^{\frac 1 4}\|\mathbf{u}\|_{L_t^2(H^3)}^ {\frac 1 2}\|\mathbf{u}\|_{L_t^\infty(H^1)}^ {\frac 1 2}  \|\vartheta_t\|_{L_t^2(H^2)} ), 
\end{aligned}  
\end{equation}  
and
\begin{equation}\nonumber
\begin{aligned} 
\mathcal{R}_{15}:&= \iint_{\widetilde{\Omega}_t} e^Y J\chi^2 D_3 Y_{ t} D_3\mathcal{H}_1 dyds \\
&\lesssim t^{\frac{1}{4}}  \|D_{3 } Y_t\|_{L_t^\infty(L^2)}^{\frac 1 2}  \|D_{3 } Y_t\|_{L_t^2(H^1)}^{\frac 1 2}  [\|\Theta^{-1} \|_{L_{t}^\infty (H^2) }  \|\mathbf{u}_t\|_{L_t^2(H^1)} \|\Theta \|_{L_t^\infty(H^2)}  \\
&\q +\|p^{-1}\|_{L_t^\infty(H^2)}   \|\Theta \|_{L_t^2(H^4)} ^{\frac 1 2}  \|\Theta \|_{L_t^\infty(H^2)} ^{\frac 1 2} (\|\vartheta_t \|_{L_t^2(H^2)} + \|\e\sigma_t \|_{L_t^\infty(H^1)}  )  \\ 
&\q + \e  \|p^{-1}\|_{L_t^\infty(H^2)}  \|\e \sigma_t \|_{L_t^\infty(H^1)} \|\e\mathbf{u}\|_{L_t^\infty(H^2)}\|\mathbf{u}\|_{L_t^2(H^3)}  ]\\
&\q+  \|p^{-1}\|_{L_t^\infty(H^2)}  \|\Theta\|_{L_{x, t}^\infty } \|\vartheta_t\|_{L_t^2(H^2)}^{\frac 1 2} \|\vartheta_t\|_{L_t^\infty(H^1)}^{\frac 1 2} \|\e^2 \mathbf{u}_{t}\|_{L_t^2(H^2)}  \|\mathbf{u}\|_{L_t^\infty(H^1)}^{\frac 1 2}\|\mathbf{u}\|_{L_t^2(H^3)}^{\frac 1 2}.
\end{aligned}  
\end{equation}  
Collecting the estimates above and applying  \ef{lemtt}, \ef{uth2} and  \ef{ttan3},  we conclude that 
\begin{equation}\label{d333}
\begin{aligned} 
 & \int_{\widetilde{\Omega}} J\chi^2 |a_{3i}D_3 U^i_{\e t} |^2 dy +\int_{\widetilde{\Omega}}  J\chi^2 e^Y | D_3 L_{t} |^2  +  J\chi^2 e^Y | D_3 Y_{t} |^2 dy
\\
& +  \iint_{\widetilde{\Omega}_t} J\chi^2   |D_3 (a_{lj} D_l U^j_{\e t} )|^2 dyds + \iint_{\widetilde{\Omega}_t}   J\chi^2  | D_{y3} Y_t |^2 dyds \\
&\leq C_0(M_0){\rm exp}\{(t^{\frac 1 4} + \e^{\frac 1 4} ) C_1( \mathcal{E}(t)) \}+\delta \|\mathbf{u}_{t}\|_{L_t^2( H^1)}^2 
 \end{aligned}  
\end{equation}  
for any $\e\in(0, 1]$ and $t\in(0, T]$.  The proof is complete. 
\hfill$\Box$

\begin{rmk}   \label{rmk}
Applying Lemma \ref{lem1} to $\n\vartheta_t$, and using Lemmas \ref{lem3.2} and \ref{lem3.3}, we deduce
\begin{equation} \nonumber
\begin{aligned} 
 \|\vartheta_t\|_{L_t^2(H^2)} ^2
& =  \|\vartheta_t\|_{L_t^2(L^2)}^2 + \|\n\vartheta_t\|_{L_t^2(H^1)}^2 \\
& \lesssim  \|\vartheta_t\|_{L_t^2(H^1)}^2 +\|\Delta\vartheta_t\|_{L_t^2(L^2)}^2 +\|\cl\n \vartheta_t\|_{L_t^2(L^2)}^2  +\| \n\vartheta_t\cdot\mathbf{n}\|_{L_t^2(H^{\frac 1 2}(\p\Omega))}^2 \\
& \lesssim  \|\vartheta_t\|_{L_t^2(H^1)} ^2 +\|\Delta\vartheta_t\|_{L_t^2(L^2)}^2 \\
&\leq C_0(M_0){\rm exp}\{(t^{\frac 1 4} + \e^{\frac 1 4} ) C_1( \mathcal{E}(t)) \}+\delta \|\mathbf{u}_{t}\|_{L_t^2( H^1)}^2 . 
 \end{aligned}    
\end{equation}  
Moreover, due to the relation $\theta=e^\vartheta -1$,  we obtain that
\begin{equation}  \label{the1}
\begin{aligned} 
\|\n \theta_t\|_{L_t^\infty(L^2)}& =\|  \n e^{\vartheta}\vartheta_t +e^{\vartheta}  \n \vartheta_t\| _{L_t^\infty(L^2)}\\ 
& 
\leq C \|\vartheta_t\|_{L_t^\infty(H^1)} \|\Theta\|_{L_t^\infty(H^2)}   \\
&\leq C_0(M_0){\rm exp}\{(t^{\frac 1 4} + \e^{\frac 1 4} ) C_1( \mathcal{E}(t)) \}+\delta \|\mathbf{u}_{t}\|_{L_t^2( H^1)}^2 ,
 \end{aligned}    
\end{equation}  
and 
\begin{equation}  \label{the2}
\begin{aligned} 
\|\Delta \theta_t\|_{L_t^2(L^2)}
& =\|2 \n e^{\vartheta} \n \vartheta_t + e^{\vartheta}  \Delta \vartheta_t + \vartheta_t \Delta e^{\vartheta}  \| _{L_t^2(L^2)} \\ 
& \leq C \| \Theta\|_{L_{t}^\infty(H^2)}   \|\vartheta_t\|_{L_t^2(H^2)}\\
&\leq C_0(M_0){\rm exp}\{(t^{\frac 1 4} + \e^{\frac 1 4} ) C_1( \mathcal{E}(t)) \}+\delta \|\mathbf{u}_{t}\|_{L_t^2( H^1)}^2 .
 \end{aligned}    
\end{equation}   

\end{rmk}

\subsection{Uniform estimates of $(\mathbf{v}_t,q_t)$}
Now, we apply  the equations $\eqref{ns0}_{2, 3}$ to bound $\|\mathbf{v}_t\|_{L_t^2(H^1)}$ by use of the uniform bounds for $\theta_t$. 
To this end,  we differentiate $\eqref{ns0}_{2, 3}$ with respect to $t$ to get
\begin{equation}\label{dtns0}\left\{\begin{aligned}
&\rho(\mathbf{v}_{tt}+\mathbf{u}\cdot\nabla\mathbf{v}_t)+\frac{1}{\e}\nabla q_t-\mu\Delta \mathbf{v}_t-\xi\nabla\dv \mathbf{v}_t-\beta\nabla\Delta\theta_t\\
&\qquad =f_t-\rho_t(\mathbf{v}_t+\mathbf{u}\cdot\nabla\mathbf{v}) -\rho\mathbf{u}_t\cdot\nabla\mathbf{v}  &{\rm in}\  \Omega_t, \\
&\frac{1}{2} (q_{tt}+\mathbf{u}\cdot\nabla q_t)+\frac{1}{\e}\dv\mathbf{v}_t=g_t-\frac{1}{2}\mathbf{u}_t\cdot\nabla q  &{\rm in}\  \Omega_t, \\
&\mathbf{v}_t  =-\frac{\kappa}{2}\n_\tau \theta_t, \q \frac{\partial \theta_t}{\partial \mathbf{n}} =0 &{\rm on}\  \p\Omega_t, 
\end{aligned}\right.\end{equation} 
where $f_t,g_t$, and $h_t$ are given by 
\begin{equation}\begin{aligned} 
f_t:=\frac{\kappa}{2} 
\left\{\nabla \rho_t\cdot(\theta_t+ \mathbf{u}\cdot\nabla\theta)+\nabla \rho\cdot[\rho^{-1} (\frac{\kappa}{2}\Delta\theta-\dv\mathbf{v}+h )]_t+(\rho\nabla\mathbf{u}\nabla\theta)_t-\nabla h_t \right \}, 
\nonumber 
\end{aligned}\end{equation}
\begin{equation}\begin{aligned}
g_t:=-(q\dv\mathbf{u})_t+\frac \e 2 [2\mu|D(\mathbf{u})|^2+\lambda(\dv \mathbf{u})^2 ]_t, \nonumber
\end{aligned}\end{equation}
\begin{equation}\begin{aligned}
h_t:=-\e (q\dv\mathbf{u})_t+\e^2 [2\mu|D(\mathbf{u})|^2+\lambda(\dv \mathbf{u})^2 ]_t.\nonumber
\end{aligned}\end{equation}

\begin{lem}\label{lemut}
For any $t\in[0, T]$ and $\e\in (0,1]$,  one has
\begin{equation}\label{lem2}\begin{aligned}
\|( \mathbf{u}_t, q_t )\|^2_{L_t^\infty(L^2)} +\|\mathbf{u}_t \|^2_{L^2_t(H^1)}\leq C_0(M_0){\rm exp}\{(t^{\frac 1 8} + \e^{\frac 1 4} ) C_1(\mathcal{E}(t)) \}.
\end{aligned}\end{equation}
\end{lem} 
\textbf{Proof}.   
To prove \ef{lem2}, the test functions $\mathbf{v}_t $ cannot be employed in $\eqref{dtns0}_1$ because $\mathbf{v}_t $ lacks the boundary conditions required to eliminate the boundary term arising from the diffusion term $-\int_\Omega \Delta \mathbf{v}_t \cdot \mathbf{v}_t  dx$.

By  Lemma \ref{lem2.4}, we choose a new test function $Z:=\mathbf{v}_t-\ell $, where 
$\ell\in \mathcal{V}^1(\Omega_t)$   is the solution to   Stokes problem \ef{stokes} with   $\ell_0=0$ and $g=-\frac{\kappa}{2} \n_\tau \theta_t$. The    test function $Z$ satisfies
$$ 
Z|_{\p\Omega_t}=0 \q {\rm and } \q \dv Z=\dv \mathbf{v}_t\q {\rm in}~ \Omega_t.
$$ 
We take the inner product of  $\eqref{dtns0}_1$ and $Z$ in $L^2(\Omega\times(0, t))$,  then use integration by parts and $Z|_{\p\Omega_t}=0$ to discover 
\begin{equation}\label{low22}\begin{aligned}
&\frac{1}{2} \|\sqrt{\rho}Z \|^2_{L^2}(t)-\frac{1}{\e}\omt  q_t\dv  \mathbf{v}_t dxds  +\mu\|\nabla Z \|^2_{L_t^2(L^2)}+\xi\|\dv Z  \|^2_{L_t^2(L^2)}\\
&=\frac{1}{2} \|\sqrt{\rho_0} Z (0)\|^2_{L^2}  +Q_1+Q_2+Q_3+Q_4.
\end{aligned}\end{equation} 
Due to the duality $\langle \cdot, \cdot\rangle_{{ H_0^1(\Omega)}\times { H^{-1}(\Omega)}}$, we derive 
\begin{equation}\begin{aligned}\nonumber
Q_1:&= - \omt  \rho(\ell_{t}+\mathbf{u}\cdot\n \ell)\cdot Z dxds\\
& \lesssim \| \rho  \|_{L_t^\infty(H^2)}[\|\ell_{t}\|_{L_t^2(H^{-1})} \|Z\|_{L_t^2(H^{1}_0)} +\|\mathbf{u}\|_{L_t^2(L^\infty)}\|\n \ell\|_{L_t^2(L^2)}\|Z\|_{L_t^\infty(L^2)}  ]\\ 
& \leq \delta  \|Z\|_{L_t^2(H^{1}_0)}^2+ C_\delta \| \rho  \|_{L_t^\infty(H^{2})}^2 \|\ell_{t}\|_{L_t^2(H^{-1})}^2\\
&\q + Ct^{\frac{1}{8}} \| \rho  \|_{L_t^\infty(H^2)} \|\mathbf{u}\|_{L_t^2(H^2)}^{\frac{3}{4}}\|\mathbf{u}\|_{L_t^\infty(H^1)}^{\frac{1}{4}}\|\n \ell\|_{L_t^2(L^2)}\|Z\|_{L_t^\infty(L^2)}.
\end{aligned}\end{equation} 
Applying integration by parts and  $Z|_{\p\Omega_t}=0$,  we deduce  
\begin{equation}\begin{aligned}\nonumber
Q_2:&=\mu\omt  \Delta\ell \cdot Z dxds+
\beta\omt\n\Delta {\theta}_t\cdot  Z  dxds\\
&=-\mu\omt  \n\ell \cdot \n Z dxds-\beta\omt\Delta {\theta}_t\dv  Z  dxds \\
&\leq \delta  \| Z\|_{L_t^2(H^1)}^2+C_\delta (  \|\n \ell \|_{L_t^2(L^2)}^2 +\|  \theta_t\|^2_{L^2_t(H^2)} ).
\end{aligned}\end{equation}  
By use of Lemmas  \ref{lem2.4} and \ref{lem2.5}, we have 
\be\begin{aligned}\label{ell}
& \|\ell\|^2_{L_t^2(H^1)}+\|\ell_t\|^2_{L_t^2(H^{-1})}   \\ 
&\leq C(\kappa, \Omega) [ \|\n_\tau \theta_t\|^2_{L^2_t(H^{ 1/2}(\partial\Omega))} +\|\n_\tau \theta_t\|^2_{H^{1/ 2}_t(H^{- 1 /2}(\partial\Omega))} ]\\
 &\leq C(\kappa, \Omega) [ \|  \theta_t\|^2_{L^2_t(H^2(\Omega))} +\|  \theta_t\|^2_{H^{1/ 2}_t(H^{1   }(\Omega))} ]\\
&\leq C(\kappa, \Omega) [ \|  \theta_t\|^2_{L^2_t(H^2 )} +\|  \theta_t\|_{H^{1}_t(L^{2} )}\|  \theta_t\|_{L^2_t(H^2)} ]\\
&\leq C(\kappa, \Omega) [ \|  \theta_t\|^2_{L^2_t(H^2 )} + \|  \theta_t\|_{L^2_t(H^2)} (\|  \theta_t\|_{L^2_t(L^2)} +\|  \theta_{tt}\|_{L^2_t(L^2)} ) ]\\
&\leq \delta \|  \theta_{tt}\|_{L^2_t(L^2)} ^2+C_\delta C(\kappa, \Omega)\|  \theta_t\|^2_{L^2_t(H^2)}. 
\end{aligned}  
\ee    
Choosing $\delta<1$ sufficiently small 
 in the estimates for $Q_1$ and $Q_2$ and substituting \ef{ell}, we obtain 
\begin{equation}\begin{aligned}\nonumber
Q_1+Q_2 
&\leq \delta  \|  \theta_{tt}\|_{L^2_t(L^2)} ^2+ C_\delta[ (\| \rho  \|_{L_t^\infty(H^{2})}^2+1)\| \rho  \|_{L_t^\infty(H^{2})}^2 \|  \theta_t\|^2_{L^2_t(H^2)}+ C(\kappa, \beta,\Omega)\|  \theta_t\|^2_{L^2_t(H^2)} ]\\  
&\q + Ct^{\frac{1}{8}} \| \rho  \|_{L_t^\infty(H^2)} \|\mathbf{u}\|_{L_t^2(L^2)}^{\frac{3}{4}}\|\mathbf{u}\|_{L_t^\infty(H^1)}^{\frac{1}{4}}\|\n \ell\|_{L_t^2(L^2)}\|Z\|_{L_t^\infty(L^2)}. 
\end{aligned}\ee   
In addition,  the interpolation inequalities and the Young inequality are used to show
\begin{equation}
\begin{aligned}\nonumber
Q_3&:=\omt  f_t \cdot Z dxds \\
&\lesssim 
t^{\frac 1 4}  \|Z\|_{L_t^2(H^1)} ^{\frac 1 2}\|Z\|_{L_t^\infty (L^2)} ^{\frac 1 2} [ \|\rho_t\|_{L_t^2(H^1)} (\|\theta_t\|_{L_t^\infty (H^1)} +\|\mathbf{u}\|_{L_t^\infty (H^1)}\|\theta\|_{L_t^\infty (H^2)} ) \\
&\q+ \|\rho\|_{L_t^\infty (H^2)}  \|\rho^{-1}\|_{L_{x, t}^\infty }(\|\theta_t\|_{L_t^2(H^2)} +\|\mathbf{v}_t\|_{L_t^2(H^1)}+\|\e q_t\|_{L_t^\infty (H^1)} \|\mathbf{u}\|_{L_t^2(H^2)} \\
&\q +\e \|q\|_{L_t^\infty (H^2)} \|\mathbf{u}_t\|_{L_t^2(H^1)}+ \|\e^2\mathbf{u}_t\|_{L_t^2(H^2)}^{\frac 1 2} \|\mathbf{u}_t\|_{L_t^2(H^1)}^{\frac 1 2} \|\e \mathbf{u}\|_{L_t^\infty (H^2)} ) \\
&\q +\|\rho_t\|_{L_t^\infty (L^2)} \|\mathbf{u}\|_{L_t^2 (H^3)} \|\theta\|_{L_t^\infty (H^2)} +\|\rho\|_{L_{x, t}^\infty }(\|\mathbf{u}_t\|_{L_t^2(H^1)}\|\theta\|_{L_t^\infty (H^2)} \\
&\q + \|\mathbf{u}\|_{L_t^2(H^2)}\|\theta_t\|_{L_t^\infty (H^1)}  ) + \|\e q_t\|_{L_t^\infty (H^1)} \|\mathbf{u}\|_{L_t^2(H^2)} + \| q\|_{L_t^\infty (H^2)} (\e \|\mathbf{u}_t\|_{L_t^2(H^1)} \\
&\q +\|\e \dv\mathbf{u}_t\|_{L_t^2(H^1)})]   + \delta ( \|Z\|_{L_t^2(H^{1}_0)}^2+\|Z\|_{L_t^\infty(L^2)}^2+ \| \mathbf{u}\|_{L_t^2(H^3)}^2 )\\
&\q  +C_\delta \|\e ^2\mathbf{u}_t\|_{L_t^2(H^2)} ^4 \|\mathbf{u}\|_{L_t^\infty(H^1)}^2, 
\end{aligned}\end{equation}   
and   
\begin{equation}\begin{aligned}\nonumber
Q_4&:=-\omt  [\rho_t(\mathbf{v}_t+\mathbf{u}\cdot\nabla\mathbf{v}) -\rho\mathbf{u}_t\cdot\nabla\mathbf{v}] \cdot Z dxds\\ 
&\lesssim \|Z\|_{L_t^\infty(L^2)} ^{\frac 1 2}\|Z\|_{L_t^2(H^1)}^{\frac 1 2} 
[\|\rho_t\|_{L_t^2(H^1)}(  t^{\frac 1 8}\|\mathbf{u}\|_{L_t^\infty(H^1)} \| \mathbf{v}\|_{L_t^2(H^3)}^{\frac 1 4} \| \mathbf{v}\|_{L_t^\infty(H^1)}^{\frac 3 4}  \\
&\q +t^{\frac 1 4}\|\mathbf{v}_t\|_{L_t^\infty(L^2)})]  + t^{\frac 1 2} \|\rho\|_{L_{x, t}^\infty}\|Z\|_{L_t^\infty(L^2)}\|\mathbf{u}_t\|_{L_t^\infty(L^2)}\|\mathbf{v}\|_{L_t^2(H^3)}. 
\end{aligned} 
\end{equation}  
Next,  we multiply $\eqref{dtns0}_2$ by $q_t$ and integrate the resulting equality over $\Omega\times(0, t)$ to get 
\begin{equation}\label{low23}\begin{aligned}
& \frac{1}{4} \|q_t \|^2_{L^2}(t)+\frac{1}{\e}\omt q_t \dv\mathbf{v}_t dxds\\
&=\frac{1}{4} \|q_t(0) \|^2_{L^2}+\frac{1}{4}\omt |q_t|^2\dv\mathbf{u}dxds +\omt q_t(g_t-\frac{1}{2}\mathbf{u}_t\cdot\nabla q  )dxds\\
&\leq C_0(M_0)+C t^{\frac 1 2} \|q_t\|_{L_t^\infty(L^2)} (\|q_t\|_{L_t^\infty(L^2)} \|\mathbf{u}\|_{L_t^2(H^3)}  +\|q\|_{L_t^\infty(H^2)} \|\mathbf{u}_t\|_{L_t^2(H^1)} ).
\end{aligned}\end{equation}   
Summing up $\eqref{low22}$ and $\eqref{low23}$ together,   one can eliminate the large terms. Then letting $\delta<1$ be sufficiently small, we conclude that
\begin{equation} \begin{aligned}\label{low26}
&  \|\sqrt{\rho}Z \|^2_{L^2}(t)+  \|q_t \|^2_{L^2}(t)  + \| Z \|^2_{L_t^2(H^1)} \\
&\leq C_0(M_0)+\delta( \| \mathbf{u}\|_{L_t^2(H^3)}^2 +\|  \theta_{tt}\|_{L^2_t(L^2)} ^2)+C_\delta[ (\| \rho  \|_{L_t^\infty(H^{2})}^2+1)\| \rho  \|_{L_t^\infty(H^{2})}^2 \|  \theta_t\|^2_{L^2_t(H^2)}
\\
&\q + C(\kappa, \beta,\Omega)\|  \theta_t\|^2_{L^2_t(H^2)}  +\|\e ^2\mathbf{u}_t\|_{L_t^2(H^2)} ^4 \|\mathbf{u}\|_{L_t^\infty(H^1)}^2]+t^{\frac 1 8}  C_1(\mathcal{E}(t)).
\end{aligned}\end{equation} 
Hence, due to $ \mathbf{u}_t=\mathbf{v}_t+\frac{\kappa}{2}\n\theta_t=Z+\ell+\frac{\kappa}{2}\n\theta_t $, we derive 
 \be\nonumber\begin{aligned}
\|\mathbf{u}_t\|_{L_t^\infty(L^2)}^2+\|\mathbf{u}_t\|_{L_t^2(H^1)}^2
&\leq C( \|Z\|_{L_t^\infty(L^2)}^2+\|Z\|_{L_t^2(H^1)}^2+\|\ell\|_{L_t^\infty(L^2)}^2+\|\ell\|_{L_t^2(H^1)}^2\\
&\q +\|\n\theta_t\|_{L_t^\infty(L^2)}^2+\|\n\theta_t\|_{L_t^2(H^1)}^2),\\
&\leq C_0(M_0){\rm exp}\{(t^{\frac 1 8} + \e^{\frac 1 4} ) C_1(\mathcal{E}(t)) \}+C\delta\|\mathbf{u}_t\|_{L_t^2(H^1)}^2.
\end{aligned} \ee 
by  applying \ef{lemma1}, \ef{uth2}, \ef{lem1t2}, \ef{the1}, \ef{the2} and \ef{ell}. Choosing $\delta$ sufficiently small such that $C\delta<1$ yields \ef{lem2}. 
\hfill$\Box$

\subsection{Uniform estimates of $\|\nabla^2\dv\mathbf{u}\|_{L^2_t(L^2)}$ and $\|\nabla^2\Delta \theta\|_{L^2_t(L^2)}$}
This subsection aims to estimate $\|\mathbf{u}\|^2_{H^3}$ and $\frac{1}{\e}\|\nabla q\|^2_{H^1}$ by the usual estimates for the steady Stokes problem (Chapter \uppercase\expandafter{\romannumeral4} of Galdi's book \cite{Galdi}). We first rewrite the momentum equation $\eqref{ns1}_2$ as an inhomogeneous Stokes system with Dirichlet boundary condition:
\begin{equation}\label{Stokes}\left\{\begin{aligned}
&-\mu\Delta \mathbf{u} +\frac{1}{\e}\nabla q =
-\rho (\mathbf{u} _t+\mathbf{u} \cdot\nabla \mathbf{u} )+(\mu+\lambda)\nabla\dv \mathbf{u} :=f_s, \\
&\dv\mathbf{u}=g, \\
&\mathbf{u}|_{\partial \Omega}=0, 
\end{aligned}\right.\end{equation}
where $f_s$ was treated as a source term. Therefore,  we can obtain the following lemma. 
\begin{lem}   \label{lemSto} 
For any $t\in [0,T]$ and $\e\in (0,1]$,   we have 
\begin{equation}\label{lemStokes}\begin{aligned}
&\|\mathbf{u}\|^2_{L^2_t(H^2)}+ \left\|\frac{\nabla q}{\e} \right\|^2_{L^2_t(L^2)}\leq   \delta \|\mathbf{u}\|^2_{L^2_t(H^3)} + t C_\delta  [ \|\rho\|^2_{L^\infty_{x, t}}(\|\mathbf{u}\|^6_{L_t^\infty(H^1)}+ \|\mathbf{u}_t\|^2_{L^\infty_t(L^2)})+  \|\mathbf{u}\|^2_{L_t^\infty(H^1)} ], \\
&\|\mathbf{u}\|^2_{L^2_t(H^3)}+\left\|\frac{\nabla q}{\e} \right\|^2_{L^2_t(H^1)}
\leq  
t C  [ \|\rho\|^2_{L^\infty_{x, t}}(\|\mathbf{u}\|^6_{L_t^\infty(H^1)}+ \|\mathbf{u}_t\|^2_{L^\infty_t(L^2)})+  \|\mathbf{u}\|^2_{L_t^\infty(H^1)} ]  \\ 
&\q   +C[ \|\rho\|^2_{L_t^\infty(H^2)} (\|\mathbf{u}_t\|^2_{L^2_t(H^1)}  +t^{\frac 1 4}  \|\mathbf{u}\|^{\frac{5}{2}}_{L^\infty_t(H^1)} \|\mathbf{u}\|^{\frac 3 2}_{L_t^2(H^3)})  +\|\n^2\dv \mathbf{u}\|^2 _{L^2_t(L^2)}]. 
\end{aligned}\end{equation} 
\end{lem} 
\textbf{Proof}. We can derive the following estimates by employing the standard estimates of the steady Stokes problem and the interpolation inequalities 
\begin{equation}\label{uh2}\begin{aligned}
&\|\mathbf{u}\|^2_{L^2_t(H^2)}+\left\|\frac{\nabla q}{\e} \right\|^2_{L^2_t(L^2)}
\leq C(\|f_s\|^2_{L^2_t(L^2)}+\|g\|^2_{L^2_t(H^1)})\\
&\leq C[\|\rho\|^2_{L^\infty_{x, t}}(t\|\mathbf{u}_t\|^2_{L^\infty_t(L^2)}  +t^{\frac 1 2}\|\mathbf{u}\|^3_{L^\infty_t(H^1)} \|\mathbf{u}\|_{L^2_t(H^2)}) +\|\dv \mathbf{u}\|^2 _{L^2_t(H^1)}]\\  
&\leq \delta \|\mathbf{u}\|^2_{L^2_t(H^3)} + t C_\delta  [ \|\rho\|^2_{L^\infty_{x, t}}(\|\mathbf{u}\|^6_{L_t^\infty(H^1)}+ \|\mathbf{u}_t\|^2_{L^\infty_t(L^2)})+  \|\mathbf{u}\|^2_{L_t^\infty(H^1)} ]. 
\end{aligned}\end{equation} 
Similarly,  by use of \ef{uh2},  one gets 
\begin{equation}\begin{aligned}
&\|\mathbf{u}\|^2_{L^2_t(H^3)}+\left\|\frac{\nabla q}{\e} \right\|^2_{L^2_t(H^1)} 
\leq C(\|f_s\|^2_{L^2_t(H^1)}+\|g\|^2_{L^2_t(H^2)})\\ 
& \leq C[\|\rho\|^2_{L_t^\infty(H^2)}(\|\mathbf{u}_t\|^2_{L^2_t(H^1)}  +t^{\frac 1 4}  \|\mathbf{u}\|^{\frac{5}{2}}_{L^\infty_t(H^1)} \|\mathbf{u}\|^{\frac 3 2}_{L_t^2(H^3)})  +\|\dv \mathbf{u}\|^2 _{L^2_t(H^2)}]\\ 
& \leq \delta \|\mathbf{u}\|^2_{L_t^2(H^3)} + 
t C_\delta  [ \|\rho\|^2_{L^\infty_{x, t}}(\|\mathbf{u}\|^6_{L_t^\infty(H^1)}+ \|\mathbf{u}_t\|^2_{L^\infty_t(L^2)})+  \|\mathbf{u}\|^2_{L_t^\infty(H^1)} ]  \\
&\q   + C[\|\rho\|^2_{L_t^\infty(H^2)}(\|\mathbf{u}_t\|^2_{L^2_t(H^1)}  +t^{\frac 1 4}  \|\mathbf{u}\|^{\frac{5}{2}}_{L^\infty_t(H^1)} \|\mathbf{u}\|^{\frac 3 2}_{L_t^2(H^3)})  +\|\n^2\dv \mathbf{u}\|^2 _{L^2_t(L^2)}].  
\end{aligned}\end{equation}  
Choosing $\delta $   sufficiently small,  we obtain \eqref{lemStokes} immediately.
\hfill$\Box$

\vskip3mm

In the spirit of the ideas in  \cite{Valli1983, Valli1986} by Valli and Zajaczkowski,  we establish the estimates of $\|\nabla^2\dv\mathbf{u}\|_{L^2_t(L^2)}$ and $\|\nabla^2\Delta \theta\|_{L^2_t(L^2)}$
by dividing them into  the interior part and the boundary part. The interior estimates follow as the whole space case. For the boundary estimates, we apply the isothermal coordinates   near the boundary $\p\Omega$.  
 \subsubsection{Interior estimate}
\begin{lem}\label{lemint}
For any $\e\in(0, 1]$ and $t\in [0, T]$,  the following estimates hold
\begin{equation}\label{lemin}\begin{aligned}
&\frac{1}{2} \|\chi_0 \partial_{jk}\mathbf{v}\|^2_{L_t^\infty (L^2)} +\frac{1}{4} \|\chi_0\partial_{jk}q\|^2_{L_t^\infty (L^2)} +\mu\|\chi_0\partial_{jk}\nabla\mathbf{v}\|^2_{L_t^2(L^2)} +(\mu+\lambda)\|\chi_0\partial_{jk}\dv\mathbf{v}\|^2_{L_t^2(L^2)} \\ 
&\leq C_0(M_0) {\rm exp }\{(t^{\frac{1}{16}}+\e ) C_1(\mathcal{E}(t))\},  
\end{aligned}\end{equation} 
where $\chi_0\in C_0^\infty(\Omega)$ is a cut-off function. 
\end{lem}
\textbf{Proof}. Applying $\partial_{jk}$ to $\eqref{ns0}_2$,  multiplying by $\chi_0^2\partial_{jk}\mathbf{v}$ and integrating the resulting equation in $\Omega$ to derive
\begin{equation}\label{lemin1}\begin{aligned}
& \frac{1}{2}\frac{d}{dt}\int_\Omega  \chi_0^2 \rho|\partial_{jk}\mathbf{v}|^2 dx +\mu\int_\Omega  \chi_0^2 |\partial_{jk}\nabla\mathbf{v}|^2  dx +\xi \int_\Omega \chi_0|\partial_{jk}\dv\mathbf{v}|^2 dx\\
&+\frac{1}{\e} \int_{\Omega}\partial_{jk}\nabla q\cdot \chi_0^2\partial_{jk}\mathbf{v}dx-\beta \int_\Omega 
 \partial_{jk}\n \Delta \theta \cdot  \chi_0^2\partial_{jk}\mathbf{v} dx
=\sum_{i=1}^4 I_i, 
\end{aligned}\end{equation} 
where $I_1$ through $I_4$ are controlled by interpolation inequalities and Sobolev inequalities,  
\begin{equation*} 
\begin{aligned} 
I_1:=\frac{1}{2} \int_\Omega  \chi_0^2 \rho_t |\partial_{jk}\mathbf{v}|^2 dx\leq C  \|\rho_t\| _{L^2}  \|\mathbf{v}\|_{H^3}^{\frac 7 4} \|\mathbf{v}\|_{H^3}^{\frac 1 4}, 
\end{aligned} 
\end{equation*} 
\begin{equation*} 
\begin{aligned} 
I_2&:= -  \int_\Omega [\p_j \rho \p_k (\mathbf {v}_t+\mathbf{u}\cdot\n\mathbf {v}) +\p_k \rho \p_j (\mathbf {v}_t+\mathbf{u}\cdot\n\mathbf {v}) \\
&\q +\rho \p_{jk}( \mathbf{u}\cdot \n \mathbf{v}) + \p_{jk}\rho ( \mathbf{u}\cdot \n \mathbf{v})
]\chi_0^2 \partial_{jk}\mathbf{v}  dx\\
&\leq  C \int_\Omega  \{ |\n \rho |(| \n \mathbf{v}_t|+|\n \mathbf{u}\| \n \mathbf{v}|+ | \mathbf{u}\| \n^2 \mathbf{v}|)+|\rho |( |\n^2 \mathbf{u}\| \n \mathbf{v}|\\
&\quad +|\n \mathbf{u}\| \n^2 \mathbf{v}| +|\mathbf{u}\| \n^3 \mathbf{v}| ) +|\n^2\rho \|\mathbf{u}\| \n \mathbf{v}|\} | \n^2 \mathbf{v}| dx \\ 
&\leq C \|\rho\|_{H^2}\| \mathbf{v} \|_{H^1}^{\frac 1 4}\| \mathbf{v} \|_{H^3}^{\frac 3 4} ( \| \mathbf{v}_t \|_{H^1} + \| \mathbf{v} \|_{H^1}^{\frac 1 2}\| \mathbf{v} \|_{H^3}^{\frac 1 2}\| \mathbf{u} \|_{H^1}^{\frac 1 2}\| \mathbf{u} \|_{H^3}^{\frac 1 2}\\
&\q + \|\mathbf{u}\|_{H^3} \|\mathbf{v}\|_{H^1}+\|\mathbf{u}\|_{H^1}\|\mathbf{v}\|_{H^3} ), 
\end{aligned} 
\end{equation*}  
\begin{equation*}  
\begin{aligned}  
I_3 :&= -2\int_\Omega \chi_0 [\mu \p_{jk }  \n\mathbf{v}:(\n \chi_0 \p_{jk }  \mathbf{v}) +\xi \p_{jk }  \dv \mathbf{v} \n \chi_0\cdot \p_{jk } \mathbf{v} ] dx\leq C\| \mathbf{v}\|_{H^3}^{\frac  3 2}\|\mathbf{v}\|_{H^1}^{\frac  1 2},  
\end{aligned} 
\end{equation*}  
\begin{equation*} 
\begin{aligned}  
I_4  : &= \int_\Omega   \chi_0^ 2 \p_{jk } f \p_{jk }  \mathbf{v} dx=  -\int_\Omega \p_{j } f \p_k(\chi_0^ 2\p_{jk }  \mathbf{v} )dx\leq C \|\mathbf{v}\|_{H^3}\|f\|_{H^1}, 
\end{aligned} 
\end{equation*} 
and
\begin{equation}\label{fh1}\begin{aligned}
\| f\|_{L_t^2(H^1)} &\lesssim \|\rho\|_{L_t^\infty(H^2)}[ t^{\frac{1}{8}}\|\theta_t\|_{L_t^\infty(H^1)}^{\frac 1 4} \|\theta_t\|_{L_t^2(H^2)}^{\frac 3 4}  + t^{\frac{1}{4}}\|\theta_t\|_{L_t^\infty(H^1)}^{\frac 1 2} \|\theta_t\|_{L_t^2(H^2)}^{\frac 1 2}  \\ 
&
+ t^{\frac{1}{8}} \| \mathbf{u}\|^{\frac 3 8}_{L_t^2(H^3)}\| \mathbf{u}\|^{\frac 5 8}_{L_t^\infty(H^1)}\| \theta\|^{\frac 3 8}_{L_t^2(H^4)}\| \theta\|^{\frac 5 8}_{L_t^\infty(H^2)}+\e \| q\| _{L_t^\infty(H^2)}\| \mathbf{u}\|_{L_t^2(H^3)} 
\\ 
& + t^{\frac {1} {16}}\| \mathbf{u}\|_{L_t^\infty(H^1)}\| \theta\|^{\frac 7 8}_{L_t^2(H^4)}\| \theta\|^{\frac 1 8}_{L_t^\infty(H^2)} + t^{\frac{1}{8}} \| \mathbf{u}\|^{\frac 3 4}_{L_t^2(H^3)}\| \mathbf{u}\|^{\frac 1 4}_{L_t^\infty(H^1)} \| \theta\| _{L_t^\infty(H^2)}  \\
& +t^{\frac {1} {16}} \|\e^2 \mathbf{u}\|_{L_t^\infty(H^3)}\| \mathbf{u}\|_{L_t^2(H^3)}^{\frac 7 8} \| \mathbf{u}\|_{L_t^\infty(H^1)}^{\frac 1 8} ] \\
&\leq  (t^{\frac{1}{16}}+\e )C_1(\mathcal{E}(t)). 
\end{aligned}\end{equation} 
Next,  we integrate the inner product of $\partial_{jk}\eqref{ns0}_3$ and $\chi_0^2\partial_{jk}q$ in $\Omega$ to get
\begin{equation}\label{lemin2}
\begin{aligned}
&\frac{1}{4}\frac{d}{dt}\int_\Omega \chi_0|\partial_{jk}q|^2 dx
+\frac{1}{\e} \int_{\Omega}\partial_{jk} \dv\mathbf{v} \chi_0^2\partial_{jk}q dx
\\
&= {\frac 1 4} \int_\Omega  \dv (\chi_0^2 \mathbf{u}) |\partial_{jk}q|^2 dx   -\frac{1}{2}  \om \chi_0^2 \partial_{jk}q(  [\partial_{jk},  \mathbf{u}\cdot\nabla] q   +\partial_{jk} g  ) dx  \\
&\leq C(\|\mathbf{u}\|_{H^3}\|q\|_{H^2}^2  +\e \|q\|_{H^2} \|\mathbf{u}\|_{H^3}^2) . 
\end{aligned}
\end{equation} 
Then we take the inner product of $\partial_{jk}\eqref{ns0}_4$ and $-\beta \chi_0^2\partial_{jk}\Delta \theta $ to obtain
\begin{equation}\label{lemin3}\begin{aligned}
& \frac{\beta}{2} \frac{d}{dt}\int_\Omega \chi_0^2\rho| \partial_{jk}\n \theta |^2 dx +\frac{\kappa\beta}{2} \int_\Omega \chi_0^2 | \partial_{jk}\Delta \theta |^2 dx -\beta \int_\Omega 
\chi_0^2 \partial_{jk}\dv\mathbf{v} \partial_{jk}\Delta \theta dx =\sum_{i=5}^7 I_i, 
\end{aligned}\end{equation}
where $I_5$,  $I_6$ and $I_7$ are estimated in sequence, 
\begin{equation*} 
\begin{aligned} 
I_5: &=\frac{\beta }{2} \int_\Omega  \chi_0^2 \rho_t |\partial_{jk}\n\theta |^2 dx-\beta  \int_\Omega \n( \chi_0^2 \rho) \partial_{jk}\theta_t \cdot \partial_{jk}\n\theta dx\\
& \leq \delta \|\chi_0\partial_{jk}\n\theta\|_{L^2}^2+C_\delta \|\rho_t\| _{L^2}  \|\rho_t\| _{H^1} \|\theta\|_4^2 +C \|\rho\|_2\|\theta_t\|_2\|\theta\|_{H^2}^{\frac 1 2}\|\theta\|_{H^4}^{\frac 1 2},  
\end{aligned} 
\end{equation*}  
\begin{equation*} 
\begin{aligned} 
I_6 &= \beta \int_\Omega \chi_0^2 \partial_{jk}\Delta \theta [ \p_{jk }\rho (\theta_t+\mathbf{u}\cdot \n \theta )+ \p_{j }\rho \p_k(\theta_t+\mathbf{u}\cdot \n \theta )  \\
&\quad +\p_{k}\rho \p_j(\theta_t+\mathbf{u}\cdot \n \theta )+\rho \partial_{jk}(\mathbf{u}\cdot \n \theta)  ] dx \\
&\leq C \|\theta\|_{H^4}\|\rho\|_{H^2} [ \|\theta_t\|_{L^\infty}+ \| \mathbf{u}\|_{L^\infty} \|\n \theta\|_{L^\infty} +\|\theta_t\|_{L^3}+ \|\n\mathbf{u}\|_{L^6}  \|\n\theta\|_{L^6} \\
&\q +\|\mathbf{u}\|_{L^6} \|\n^2 \theta\|_{L^6} +\|\n\mathbf{u}\|_{L^3} \|\n \theta\|_{L^6 } + \|\mathbf{u}\|_{L^6} \|\n^3\theta\|_{L^3} +\|\n^2\mathbf{u}\|_{L^2} \|\n \theta\|_{L^\infty} 
]
\\ 
&\leq C \|\theta\|_{H^4}\|\rho\|_{H^2}[\|\theta_t\|_{H^1}^{\frac 1 4}  \|\theta_t\|_{H^2}^{\frac 3 4}+\|\theta_t\|_{H^1}^{\frac 1 2}  \|\theta_t\|_{H^2}^{\frac 1 2}+\|\theta\|_{H^2}^{\frac 5 8} \|\theta\|_{H^4}^{\frac 3 8} (\| \mathbf{u}\|_{H^1}^{\frac 5 8} \| \mathbf{u}\|_{H^3}^{\frac 3 8} \\ 
&\q  + \|\mathbf{u}\|_{H^3}^{\frac 1 2}\|\mathbf{u}\|_{H^1}^{\frac 1 2} ) +\|\theta\|_{H^2} (\|\mathbf{u}\|_{H^3}^{\frac 1 2}\|\mathbf{u}\|_{H^1}^{\frac 1 2}+ \|\mathbf{u}\|_{H^3}^{\frac 1 4}\|\mathbf{u}\|_{H^1}^{\frac 3 4}  ) \\
&\q +\|\mathbf{u}\|_{H^1} (\|\theta\|_{H^2} ^{\frac 1 2}\|\theta\|_{H^4}^{\frac 1 2}+\|\theta\|_{H^2} ^{\frac 1 4}\|\theta\|_{H^4}^{\frac 3 4} )], 
\end{aligned} 
\end{equation*} 
\begin{equation*} 
\begin{aligned} 
I_7 &= -\beta \int_\Omega \chi_0^2 \partial_{jk}\Delta \theta \partial_{jk} h dx
\leq  C\|\theta\|_{H^4} \|h\|_{H^2}.
\end{aligned}
\end{equation*} 
Combining \eqref{lemin1}, \eqref{lemin2}, and \eqref{lemin3}, we deal with the singular and higher-order terms on the left-hand side as follows:  
\begin{equation*} 
\begin{aligned} 
&\frac{1}{\e} \int_{\Omega}\partial_{jk}\nabla q\cdot \chi_0^2\partial_{jk}\mathbf{v}dx + \frac{1}{\e} \int_{\Omega}\partial_{jk} \dv\mathbf{v} \chi_0^2\partial_{jk}q dx\\
&= -\frac{1}{\e} \int_{\Omega}\partial_{jk} q\cdot \n(\chi_0^2)\partial_{jk}\mathbf{v}dx\leq C \|\mathbf{v}\|_{H^3}^{\frac 1 2}\|\mathbf{v}\|_{H^1}^{\frac 1 2} \left\|\frac{\n^2 q}{\e}\right\|_{L^2} , 
\end{aligned} 
\end{equation*} 
and
\begin{equation*} 
\begin{aligned} 
&-\beta \int_\Omega 
 \partial_{jk}\n \Delta \theta \cdot  \chi_0^2\partial_{jk}\mathbf{v} dx-\beta \int_\Omega 
\chi_0^2 \partial_{jk}\dv\mathbf{v} \partial_{jk}\Delta \theta dx \\
&= \beta \int_\Omega 
 \partial_{jk} \Delta \theta \cdot \n( \chi_0^2)\partial_{jk}\mathbf{v} dx \leq C \|\mathbf{v}\|_{H^3}^{\frac 1 2}\|\mathbf{v}\|_{H^1}^{\frac 1 2} \|\theta\|_{H^4} .
\end{aligned}
\end{equation*} 
 Thus,  integrating the equalities  \eqref{lemin1}, \eqref{lemin2}, and  \eqref{lemin3} on $[0, t]$,  utilizing \ef{fh1} and \ef{lemma1},  we conclude \ef{lemin}.
 \hfill$\Box$

\subsubsection{Boundary estimate} 
Define the unknowns in local coordinates as 
$$U(t, y):=\mathbf{u} (t, \Lambda(y)), \q
V(t, y):=\mathbf{v} (t, \Lambda(y)), \q
Q(t, y):=q(t,  \Lambda(y)), $$
$$ K(t, y):=\rho(t, \Lambda(y)),  \q H (t, y):=\theta(t, \Lambda(y)).$$
We rewrite system $\eqref{ns0}_{2, 3, 4}$ in $[0, T]\times \widetilde{\Omega}$ as

\begin{equation}\label{ns2}\left\{\begin{aligned}
&K(V^i_{t}+U^j a_{kj}D_k V^i ) +\frac{1}{\e}a_{ki}D_k Q
-\mu a_{kj}D_k(a_{lj}D_l V^i )\\
&\q\q -\xi a_{ki}D_k(a_{lj}D_l V^j ) -\beta a_{ki}D_k (a_{pj}D_p (a_{lj}D_l H )) =\mathcal{F}, \\
&\frac{1}{2}(Q_{t}+U^j a_{kj}D_k Q )+\frac{1}{\e} a_{kj}D_k V^j=\mathcal{G},  \\ 
&K (H_{t}+U^j a_{kj}D_k H )-\frac{\kappa}{2} a_{kj}D_k(a_{lj}D_l H)+a_{kj}D_k V^j=\mathcal{H},
\end{aligned}\right.\end{equation}
where  $\mathcal{F}, \mathcal{G}, $ and $\mathcal{H}$ are defined as 
$$  
\mathcal{F}:= \frac{\kappa}{2}\left [a_{ki}D_k K (H_t+U^j a_{lj}D_l H  ) +K a_{ki}D_k U^j a_{lj}D_l H-a_{ki}D_k \mathcal{H} \right ], 
$$ 
$$  
\mathcal{G}:= -Q a_{kj}D_k U^j +\frac{\e}{2} \left [\frac{\mu}{2} (a_{ki}D_k U^j+a_{kj}D_k U^i)^2 +\lambda( a_{kj}D_k U^j)^2\right ],  
$$
$$
\mathcal{H}:= -\e Q a_{kj}D_k U^j+\e^2 \left [\frac{\mu}{2} (a_{ki}D_k U^j+a_{kj}D_k U^i)^2 +\lambda( a_{kj}D_k U^j)^2\right ].
$$ 
Moreover,  $(U,  Q,  K,  H)$ satisfies the following initial and boundary value conditions: 
\be \label{boundary2}
\begin{aligned}
&(U,  Q,  K,  H)| _{t=0}
=(U_0,  Q_0,  K_0,  H_0) (y) &{\rm in}~\widetilde{\Omega}, \\ 
& U =0, \quad D_3 H =0  
&{\rm on}~ \partial \widetilde{\Omega}_T.
\end{aligned}
\ee

Now,  we give the estimates of the tangential derivatives.

\begin{lem} \label{lemtan} 
For any  $\e\in(0, 1]$ and $t\in [0, T]$,   the following estimates hold: 
 \be\begin{aligned} \label{uh3tan}
&  \frac{1}{4}\int_{\widetilde{\Omega}} J\chi^2|D_{\xi\tau}Q|^2dy
+\frac{1}{4}  \int_{\widetilde{\Omega}} J\chi^2  |D_{\xi\tau} U |^2 dy 
+ \frac{1}{2}(\frac{\xi\kappa}{4}+\frac{\beta}{2})
\int_{\wt{\Omega}}J\chi^2  |a_{kj}D_{k\xi\tau}  H|^2dy ,    \\ 
&+  c_1 
\iint_{\widetilde{\Omega}_t } J\chi^2 |D_{y\xi\tau}U|^2  dyds+ c_1
\iint_{\widetilde{\Omega}_t} J\chi^2 |D_{y\xi\tau}V|^2  dyds   \\
& + \frac{\kappa^2}{16}(9\mu+6\lambda-\kappa) \iint_{\wt{\Omega}_t}J\chi^2|D_{\xi\tau} (a_{kj}D_k(a_{lj}D_l H ) )|^2 dy ds 
\\    
&  \leq C_0(M_0) {\rm exp}\{(t^{\frac{1}{16}}+\e^{\frac 1 4}) C_1(\mathcal{E}(t)) \},
\end{aligned} \ee
where the constant $c_1$ depends only on $\mu,\xi$, and $\widetilde{\Omega}$.
\end{lem}

\noindent \textbf{Proof}. First,  integrating the inner product of $D_{\xi\tau}\eqref{ns2}_2$ and $J\chi^2 D_{\xi\tau }Q$ in $\widetilde{\Omega}$ gives the following estimates
\begin{equation}\label{tanq}
\begin{aligned} 
& \frac{1}{4}\frac{d}{dt}\int_{\widetilde{\Omega}} J\chi^2|D_{\xi\tau}Q|^2dy+\frac{1}{\e}\int_{\widetilde{\Omega}} J\chi^2   D_{\xi\tau}Q D_{\xi\tau}[a_{kj}D_k (U^j-\frac{\kappa}{2} a_{pj}D_p H)] dy\\ 
&=\frac{1}{4} \int_{\widetilde{\Omega}}D_k( J\chi^2 U^j a_{kj}) |D_{\xi\tau}Q |^2 dy -\frac{1}{2} \int_{\widetilde{\Omega}} J\chi^2 D_{\xi\tau}Q [D_{\xi\tau}(U^j a_{kj} )D_k Q\\
&\q +D_{\tau}(U^j a_{kj} )D_{k\xi} Q +D_{\xi}(U^j a_{kj} )D_{k\tau} Q ] 
 dy + \int_{\widetilde{\Omega}} J\chi^2 D_{\xi\tau}Q D_{\xi\tau}\mathcal{G}dy\\
 &\leq C (\| q\|_{H^2}^2 \|\mathbf{u}\|_{H^3}+\e \| q\|_{H^2}\|\mathbf{u}\|_{H^3}^2).
\end{aligned}
\end{equation} 
We take the inner product of $D_{\xi\tau}\eqref{ns2}_1$ and $J\chi^2D_{\xi\tau}U^i$ to get 
\begin{equation}\label{tanu1}
\begin{aligned}
 &\frac 1 2 \frac{d}{dt} \int_{\widetilde{\Omega}} J\chi^2 K |D_{\xi\tau} U |^2 dy 
 +\mu \int_{\widetilde{\Omega}} J\chi^2 a_{kj}D_{k\xi\tau}U^i a_{lj}D_{l\xi\tau}V^i dy\\
 & 
 +\xi \int_{\widetilde{\Omega}} J\chi^2 a_{ki}D_{k\xi\tau}U^i a_{lj}D_{l\xi\tau}V^j dy +\frac{1}{\e}\int_{\widetilde{\Omega}} J\chi^2 D_{\xi\tau}U^i D_{\xi\tau}(a_{ki}D_k Q) dy\\
& -\beta \int_{\widetilde{\Omega}} J\chi^2 D_{\xi\tau}U^i  D_{\xi\tau}[ a_{ki}D_k (a_{pj}D_p (a_{lj}D_l H ))] dy= \sum_{i=1}^{5} \mathcal{O}_i  , 
\end{aligned}\end{equation}
where 
\begin{equation*} 
\begin{aligned} 
\mathcal{O}_1&:= -\mu \int_{\widetilde{\Omega}} J a_{lj}D_{l\xi\tau}V^i a_{kj}D_k(\chi^2)D_{\xi\tau}U^i dy-
 \mu \int_{\widetilde{\Omega}} J\chi^2 D_{\xi\tau}U^i[D_{\xi\tau}a_{kj} D_k(a_{lj} D_l V^i ) \\
 &\q 
 +D_{\tau}a_{kj} D_{k\xi}(a_{lj} D_l V^i ) +D_{\xi}a_{kj} D_{k\tau}(a_{lj} D_l V^i )  +a_{kj}  D_k( D_{\xi\tau}a_{lj} D_l V^i\\
 &\q +D_{\xi}a_{lj} D_{l\tau} V^i +D_{\tau}a_{lj} D_{l\xi} V^i ) ] \\
 &\leq  C\| \mathbf{v}\|_{H^3}\| \mathbf{u}\|_{H^2}, 
\end{aligned} 
\end{equation*}
\begin{equation*}
\begin{aligned}
\mathcal{O}_2&:=  -\xi \int_{\widetilde{\Omega}} J a_{lj}D_{l\xi\tau}V^j  a_{ki} D_k(\chi^2)D_{\xi\tau}U^i dy
- \xi \int_{\widetilde{\Omega}} J\chi^2 D_{\xi\tau}U^i [D_{\xi\tau}a_{ki} D_k(a_{lj} D_l V^j ) \\
 &\q  
 +D_{\tau}a_{ki} D_{k\xi}(a_{lj} D_l V^j ) +D_{\xi}a_{ki} D_{k\tau}(a_{lj} D_l V^j )  +a_{ki}  D_k( D_{\xi\tau}a_{lj} D_l V^j\\
 &\q +D_{\xi}a_{lj} D_{l\tau} V^j +D_{\tau}a_{lj} D_{l\xi} V^j ) ]  \\
 &\leq  C\| \mathbf{v}\|_{H^3}\| \mathbf{u}\|_{H^2}, 
\end{aligned} 
\end{equation*} 
\begin{equation*}
\begin{aligned}
\mathcal{O}_3&:= \frac{1}{2}  \int_{\widetilde{\Omega}} J \chi^2 K_t |D_{\xi\tau} U |^2 dy- \int_{\widetilde{\Omega}}  J \chi^2 D_{\xi\tau} U^i [D_{\xi\tau}K U_t^i + D_{\xi}K D_{\tau} U_t^i\\
&\quad + D_{\tau}K D_{\xi} U_t^i  +  \frac{\kappa}{2}  D_{\xi\tau} (K a_{ki}D_k H )]\\
&\leq C \|\rho_t\|_{L^6} \|D_{\xi\tau} U\|_{L^2} \|D_{\xi\tau} U\|_{L^3} +C\|D_{\xi\tau} U\|_{L^3}\|\rho\|_{H^2} ( \| \mathbf{u}_t\|_{H^1} + \|\theta\|_{H^3} ), 
\end{aligned}  
\end{equation*}  
\begin{equation*}
\begin{aligned}
\mathcal{O}_4 &:=  - \int_{\widetilde{\Omega}}  J \chi^2 D_{\xi\tau} U^i  D_{\xi\tau}(KU^ja_{kj}D_k V^i  ) dy
\leq C \|D_{\xi\tau}U\|_{L^2} \|\rho\|_{H^2}
\|\mathbf{u}\|_{H^2}\|\n \mathbf{v}\|_{H^2}, 
\end{aligned} 
\end{equation*} 
\begin{equation*}
\begin{aligned}
\mathcal{O}_5 &:=  - \int_{\widetilde{\Omega}}  J \chi^2 D_{\xi\tau} U^i  D_{\xi\tau}\mathcal{F}dy
=  \int_{\widetilde{\Omega}}  D_\tau (J \chi^2 D_{\xi\tau} U^i ) D_{\xi}\mathcal{F}dy \leq C\|\mathbf{u}\|_{H^3} \|\mathcal{F}\|_{H^1} .
\end{aligned} 
\end{equation*}

An equivalent form of $\eqref{ns2}_3$ is given by
\be\label{theta0}
K (H_{t}+U^j a_{kj}D_k H )- {\kappa} a_{kj}D_k(a_{lj}D_l H)+a_{kj}D_k U^j=\mathcal{H}. 
\ee  
Next,  we multiply $D_{\xi\tau}\eqref{ns2}_3$ by $-J\chi^2 D_{\xi\tau }(a_{kj}D_k(a_{lj}D_l H))$ and integrate  over $\widetilde{\Omega}$ to get 
\begin{equation}\label{tantheta1}\begin{aligned}
&\frac{1}{2}\frac{d} {dt} 
\int_{\wt{\Omega}}J\chi^2 K|a_{kj}D_{k\xi\tau}  H|^2dy+\kappa \int_{\wt{\Omega}}J\chi^2|D_{\xi\tau} (a_{kj}D_k(a_{lj}D_l H ) )|^2 dy\\
&- \int_{\wt{\Omega}}J\chi^2 D_{\xi\tau}  (a_{kj}D_k(a_{lj}D_l H )) D_{\xi\tau}(a_{kj}D_k U^j) dy = \sum_{i=6}^{7} \mathcal{O}_i ,
\end{aligned}\end{equation}
where
\begin{equation*}
\begin{aligned}
\mathcal{O}_6 &:=  \frac{1}{2}\int_{\widetilde{\Omega}} K_t  J \chi^2 |a_{kj} D_{k\xi\tau} H|^2 dy -\int_{\widetilde{\Omega}} D_k( J\chi^2 K a_{kj}  )D_{\xi\tau} H_t a_{lj} D_{l\xi\tau} H dy\\
&\quad + \int_{\widetilde{\Omega}} D_{\xi\tau}(KU^j a_{kj}D_k H ) J\chi^2 D_{\xi\tau} (a_{kj}D_k(a_{lj}D_l H ) ) dy \\
&\leq  C[ \|\rho_t\|_{L^3} \| D_{k\xi\tau}H \|_{L^2} \| D_{k\xi\tau}H \| _{L^6}+ \|\rho\|_{H^2} \|\theta_t\|_{H^2} \|\n^3 \theta\|_{L^3}\\
&\q + \|\rho\|_{H^2}(\|\n^2\mathbf{u}\|_{L^3} \|\n\theta\|_{L^6} + \|\mathbf{u}\|_{L^6} \|\n^3\theta\|_{L^3} + \| \mathbf{u}\|_{L^\infty} \|\n\theta\|_{L^\infty} )  ], 
\end{aligned} 
\end{equation*} 
\begin{equation*}
\begin{aligned}
\mathcal{O}_7 &:=  \int_{\widetilde{\Omega}}  J\chi^2 D_{\xi\tau} (a_{kj}D_k(a_{lj}D_l H ) ) D_{\xi\tau}\mathcal{H} dy \\
&\leq C \| \theta\|_{H^4} ( \e \|q\|_{H^2} \|\mathbf{u}\|_{H^3}+ \e^2  \|\mathbf{u}\|_{H^3} \|\n \mathbf{u}\|_{L^\infty}+ \e^2  \|\mathbf{u}\|_{H^3} \|\n^2 \mathbf{u}\|_{L^3} ).
\end{aligned} 
\end{equation*} 
Taking the inner product of $D_{\xi\tau}\eqref{ns2}_3$ and $-\frac{1}{2\e} J\chi^2 D_{\xi\tau }Q$ in $L^2(\wt{\Omega})$ gives that
\begin{equation}\label{tantheta2}\begin{aligned}
& \frac{\kappa}{2\e}\int_{\wt{\Omega}}D_{\xi\tau}(a_{kj}D_k(a_{lj}D_l H )) J\chi^2 D_{\xi\tau} Q dy
- \frac{1}{2\e}\int_{\wt{\Omega}}D_{\xi\tau}(a_{kj}D_kU^j )J\chi^2 D_{\xi\tau} Q dy
= \sum_{i=8}^{9} \mathcal{O}_i ,
\end{aligned}\end{equation}
where 
\begin{equation*} 
\begin{aligned} 
\mathcal{O}_8 &:= \frac{1}{2\e} \int_{\widetilde{\Omega}}  J\chi^2 D_{\xi\tau }QD_{\xi\tau}[K (H_{t}+U^j a_{kj}D_k H )]dy \\
&\leq C \left\|\frac{\n^2 q}{\e}\right\|_{L^2} [ \|\rho\|_{H^2}(
 \|\theta_t\|_{H^2}+  \|D_{\xi\tau }U\|_{L^2} \|\n\theta\|_{L^\infty} + \| U\|_{L^\infty} \|D_{k\xi\tau }\theta\|_{L^2} 
) ], 
\end{aligned} 
\end{equation*} 
\begin{equation*}
\begin{aligned} 
\mathcal{O}_9 &:= -\frac{1}{2\e} \int_{\widetilde{\Omega}}  J\chi^2 D_{\xi\tau }QD_{\xi\tau}\mathcal{H} dy\\
& \leq C \left\|\frac{\n^2 q}{\e}\right\|_{L^2}( \e \|q\|_{H^2} \|\mathbf{u}\|_{H^3}+ \e^2  \|\mathbf{u}\|_{H^3} \|\n \mathbf{u}\|_{L^\infty}+ \e^2  \|\mathbf{u}\|_{H^3} \|\n^2 \mathbf{u}\|_{L^3} ).
\end{aligned} 
\end{equation*}

To  obtain the tangential  estimates of $V$,  we transform the mixed viscosity terms in the left-hand side of \ef{tanu1}, into
\begin{align} \label{tan1}
&\mu \int_{\widetilde{\Omega}} 
J\chi^2 a_{kj}D_{k\xi\tau}U^i a_{lj}D_{l\xi\tau}V^i dy   \nonumber    \\
&=\frac{\mu}{2} \int_{\widetilde{\Omega}} J\chi^2 a_{kj}D_{k\xi\tau}U^i a_{lj}D_{l\xi\tau} U^i dy 
+\frac{\mu}{2} \int_{\widetilde{\Omega}} J\chi^2 [ a_{kj}D_{k\xi\tau}U^i a_{lj}D_{l\xi\tau} U^i   \nonumber  \\
&\quad -2a_{kj}D_{k\xi\tau}U^i a_{lj}D_{l\xi\tau}( \frac{\kappa}{2}a_{pi}D_p H  )+a_{kj}D_{k\xi\tau}( \frac{\kappa}{2}a_{pi}D_p H) a_{lj}D_{l\xi\tau}( \frac{\kappa}{2}a_{qi}D_q H) ] dy   \nonumber  \\
&\quad -\frac{\mu\kappa^2 }{8} \int_{\widetilde{\Omega}} J\chi^2 a_{kj}D_{k\xi\tau}(  a_{pi}D_p H)  a_{lj}D_{l\xi\tau}(  a_{qi}D_q H)  dy    \\
&= \frac{\mu}{2} \int_{\widetilde{\Omega}} J\chi^2 a_{kj}D_{k\xi\tau}U^i a_{lj}D_{l\xi\tau} U^i dy 
+\frac{\mu}{2} \int_{\widetilde{\Omega}} J\chi^2 a_{kj}D_{k\xi\tau}V^i a_{lj}D_{l\xi\tau} V^i dy  \nonumber  \\
&\quad -\frac{\mu\kappa^2 }{8} \int_{\widetilde{\Omega}} J\chi^2 a_{kj}D_{k\xi\tau}(  a_{pj}D_p H)  a_{li}D_{l\xi\tau}(  a_{qi}D_q H)  dy,   \nonumber  
\end{align} 
and
\begin{equation}\label{tan2}
\begin{aligned} 
& \xi \int_{\widetilde{\Omega}} J\chi^2 a_{ki}D_{k\xi\tau}U^i a_{lj}D_{l\xi\tau}V^j dy \\
& = \xi \int_{\widetilde{\Omega}} J\chi^2 a_{ki}D_{k\xi\tau}V^i a_{lj}D_{l\xi\tau}V^j dy + \frac{\xi\kappa}{2} \int_{\widetilde{\Omega}} J\chi^2 a_{ki}D_{k\xi\tau}(  a_{pi}D_p H) a_{lj}D_{l\xi\tau}U^j dy\\
&\quad - \frac{\xi\kappa^2}{4} \int_{\widetilde{\Omega}} J\chi^2 a_{ki}D_{k\xi\tau}(  a_{pi}D_p H) a_{lj}D_{l\xi\tau}(  a_{pj}D_p H) dy .
\end{aligned} 
\end{equation}

Computing  $\ef{tanq}+ \frac{1}{2}\ef{tanu1}+ \ef{tantheta2}+ (\frac{\xi\kappa}{4}+\frac{\beta}{2})\cdot \ef{tantheta1}$,   utilizing \ef{tan1} and  \ef{tan2},  we conclude that the left  hand side of the resulting equality is 
\begin{align} 
&\frac{1}{4}\frac{d}{dt}\int_{\widetilde{\Omega}} J\chi^2|D_{\xi\tau}Q|^2dy+\frac{1}{\e}\int_{\widetilde{\Omega}} J\chi^2   D_{\xi\tau}Q D_{\xi\tau}[a_{kj}D_k (U^j-\frac{\kappa}{2} a_{pj}D_p H)] dy \nonumber\\ 
& +\frac{1}{4} \frac{d}{dt} \int_{\widetilde{\Omega}} J\chi^2 K |D_{\xi\tau} U |^2 dy 
 + \frac{\mu}{4} \int_{\widetilde{\Omega}} J\chi^2 a_{kj}D_{k\xi\tau}U^i a_{lj}D_{l\xi\tau} U^i dy  \nonumber\\
&
+\frac{\mu}{4} \int_{\widetilde{\Omega}} J\chi^2 a_{kj}D_{k\xi\tau}V^i a_{lj}D_{l\xi\tau} V^i dy -\frac{\mu\kappa^2 }{16} \int_{\widetilde{\Omega}} J\chi^2 |a_{kj}D_{k\xi\tau}(  a_{pj}D_p H) |^2dy\nonumber
\\
&+
\frac{\xi}{2} \int_{\widetilde{\Omega}} J\chi^2 a_{ki}D_{k\xi\tau}V^i a_{lj}D_{l\xi\tau}V^j dy
+ \frac{\xi\kappa}{4} \int_{\widetilde{\Omega}} J\chi^2 a_{ki}D_{k\xi\tau}(  a_{pi}D_p H) a_{lj}D_{l\xi\tau}U^j dy \nonumber \\
&
- \frac{\xi\kappa^2}{8} \int_{\widetilde{\Omega}} J\chi^2 |a_{ki}D_{k\xi\tau}(  a_{pi}D_p H) |^2 dy +\frac{1}{2\e}\int_{\widetilde{\Omega}} J\chi^2 D_{\xi\tau}U^i D_{\xi\tau}(a_{ki}D_k Q) dy \nonumber\\
& -\frac{\beta}{2} \int_{\widetilde{\Omega}} J\chi^2 D_{\xi\tau}U^i  D_{\xi\tau}[ a_{ki}D_k (a_{pj}D_p (a_{lj}D_l H ))] dy \nonumber\\
& +\frac{\kappa}{2\e}\int_{\wt{\Omega}}D_{\xi\tau}(a_{kj}D_k(a_{lj}D_l H )) J\chi^2 D_{\xi\tau} Q dy
- \frac{1}{2\e}\int_{\wt{\Omega}}D_{\xi\tau}(a_{kj}D_kU^j )J\chi^2 D_{\xi\tau} Q dy \nonumber\\
& + \frac{1}{2}(\frac{\xi\kappa}{4}+\frac{\beta}{2})\frac{d} {dt} \int_{\wt{\Omega}}J\chi^2 K|a_{kj}D_{k\xi\tau}  H|^2dy+\kappa  (\frac{\xi\kappa}{4} +\frac{\beta}{2})\int_{\wt{\Omega}}J\chi^2|D_{\xi\tau} (a_{kj}D_k(a_{lj}D_l H ) )|^2 dy \nonumber\\
& - (\frac{\xi\kappa}{4}+\frac{\beta}{2})\int_{\wt{\Omega}}J\chi^2 D_{\xi\tau}  (a_{kj}D_k(a_{lj}D_l H )) D_{\xi\tau}(a_{kj}D_k U^j) dy 
=: \sum_{i=1}^5\mathcal{N}_i,  \nonumber
\end{align} 
where $\mathcal{N}_1$ to $\mathcal{N}_5$ are defined and estimated as follows
\begin{equation*}
\begin{aligned}
\mathcal{N}_1&:= \frac{d}{dt}\bigg[\frac{1}{4}\int_{\widetilde{\Omega}} J\chi^2|D_{\xi\tau}Q|^2dy
+\frac{1}{4}  \int_{\widetilde{\Omega}} J\chi^2 K |D_{\xi\tau} U |^2 dy\\ 
&\q 
+ \frac{1}{2}(\frac{\xi\kappa}{4}+\frac{\beta}{2})
\int_{\wt{\Omega}}J\chi^2 K|a_{kj}D_{k\xi\tau}  H|^2dy \bigg], 
\end{aligned} 
\end{equation*} 
\begin{equation*} 
\begin{aligned}
\mathcal{N}_2&:= 
 \frac{\mu}{4} \int_{\widetilde{\Omega}} J\chi^2 a_{kj}D_{k\xi\tau}U^i a_{lj}D_{l\xi\tau} U^i dy +\frac{\mu}{4} \int_{\widetilde{\Omega}} J\chi^2 a_{kj}D_{k\xi\tau}V^i a_{lj}D_{l\xi\tau} V^i dy\\
&\q
+
\frac{\xi}{2} \int_{\widetilde{\Omega}} J\chi^2 a_{ki}D_{k\xi\tau}V^i a_{lj}D_{l\xi\tau}V^j dy, 
\end{aligned} 
\end{equation*} 
\begin{equation*}
\begin{aligned}
\mathcal{N}_3&:= \kappa  (\frac{\xi\kappa}{4}+\frac{\beta}{2})\int_{\wt{\Omega}}J\chi^2|D_{\xi\tau} (a_{kj}D_k(a_{lj}D_l H ) )|^2 dy\\
&\q -(\frac{\mu\kappa^2 }{16} +\frac{\xi\kappa^2}{8})\int_{\widetilde{\Omega}} J\chi^2 |a_{kj}D_{k\xi\tau}(  a_{pj}D_p H) |^2dy\\
&= \frac{\kappa^2}{16}(9\mu+6\lambda-\kappa) \int_{\wt{\Omega}}J\chi^2|D_{\xi\tau} (a_{kj}D_k(a_{lj}D_l H ) )|^2 dy\\
&\q -(\frac{\mu\kappa^2 }{16} +\frac{\xi\kappa^2}{8})\int_{\widetilde{\Omega}} J\chi^2 |[a_{kj}D_{k},  D_{\xi\tau}]  a_{pj}D_p H |^2dy , 
\end{aligned}  
\end{equation*} 
\begin{align}
\mathcal{N}_4&:=  \frac{\xi\kappa}{4} \int_{\widetilde{\Omega}} J\chi^2 a_{ki}D_{k\xi\tau}(  a_{pi}D_p H) a_{lj}D_{l\xi\tau}U^j dy\nonumber \\
&\q -\frac{\beta}{2} \int_{\widetilde{\Omega}} J\chi^2 D_{\xi\tau}U^i  D_{\xi\tau}[ a_{ki}D_k (a_{pj}D_p (a_{lj}D_l H ))] dy  \nonumber \\
&\quad - (\frac{\xi\kappa}{4}+\frac{\beta}{2})\int_{\wt{\Omega}}J\chi^2 D_{\xi\tau}  (a_{kj}D_k(a_{lj}D_l H )) D_{\xi\tau}(a_{kj}D_k U^j) dy  \nonumber  \\
&= \frac{\xi\kappa}{4} \int_{\widetilde{\Omega}} J\chi^2 a_{ki}D_{k\xi\tau}(  a_{pi}D_p H) a_{lj}D_{l\xi\tau}U^j dy  + \frac{\beta}{2} \int_{\widetilde{\Omega}}  D_{\xi\tau}[ (a_{pj}D_p (a_{lj}D_l H ))]   \nonumber   \\
&\q \cdot J\chi^2 a_{ki} D_{k\xi\tau}U^i   dy  - (\frac{\xi\kappa}{4}+\frac{\beta}{2})\int_{\wt{\Omega}}J\chi^2 D_{\xi\tau}  (a_{kj}D_k(a_{lj}D_l H )) D_{\xi\tau}(a_{kj}D_k U^j) dy   \nonumber \\
&\quad -\frac{\beta}{2} \int_{\widetilde{\Omega}} J\chi^2 D_{\xi\tau}U^i  [ D_{\xi\tau}a_{ki}D_k (a_{pj}D_p (a_{lj}D_l H ))+ D_{\xi}a_{ki}D_{k\tau} (a_{pj}D_p (a_{lj}D_l H ))  \nonumber \\
&\quad + D_{\tau}a_{ki}D_{k\xi} (a_{pj}D_p (a_{lj}D_l H ))] dy   \nonumber \\
&\leq C (\|\mathbf{u}\|_{H^3}\|\theta\|_{H^3}+ \|\mathbf{u}\|_{H^2}\|\theta\|_{H^4}) ,  \nonumber 
\end{align} 
\begin{equation*}
\begin{aligned}
\mathcal{N}_5&:= \frac{1}{\e}\int_{\widetilde{\Omega}} J\chi^2   D_{\xi\tau}Q D_{\xi\tau}(a_{kj}D_k U^j) dy +\frac{1}{2\e}\int_{\widetilde{\Omega}} J\chi^2 D_{\xi\tau}U^i D_{\xi\tau}(a_{ki}D_k Q) dy\\
&\quad - \frac{1}{2\e}\int_{\wt{\Omega}}D_{\xi\tau}(a_{kj}D_kU^j )J\chi^2 D_{\xi\tau} Q dy \\
&=\frac{1}{2\e}\int_{\widetilde{\Omega}} J\chi^2   D_{\xi\tau}Q D_{\xi\tau}(a_{kj}D_k U^j) dy - \frac{1}{2\e}\int_{\widetilde{\Omega}} J\chi^2 a_{ki}D_{k\xi\tau}U^i D_{\xi\tau} Q dy\\
&\quad +\frac{1}{2\e}\int_{\widetilde{\Omega}} J\chi^2 D_{\xi\tau}U^i [ D_{\xi\tau}a_{ki}D_k Q+ D_{\xi}a_{ki}D_{k\tau} Q + D_{\tau}a_{ki}D_{k\xi} Q] dy\\
&\leq C \left\| \mathbf{u} \right\|_{H^2} \left\|\frac{\nabla^2 q}{\e}\right\|_{L^2},
\end{aligned} 
\end{equation*} 
and $[a,b]:=ab-ba$ denotes the commutator notation.  
Due to the positive definiteness of $\{\sum_k a_{ki}a_{kj}\}$,  one has 
\begin{equation}\nonumber\begin{aligned}
 \sum_{i, j}a_{kj}D_{k\xi\tau} U^ia_{lj}D_{l\xi\tau} U^i\geq C |D_{y\xi\tau} U|^2, 
\end{aligned}\end{equation}
for some positive constant $C$. Then we derive 
\begin{equation*}
\begin{aligned}
\mathcal{N}_2&\geq  c_1
\int_{\widetilde{\Omega}} J\chi^2 |D_{y\xi\tau}U|^2  dy+ c_1
\int_{\widetilde{\Omega}} J\chi^2 |D_{y\xi\tau}V|^2  dy, 
\end{aligned} 
\end{equation*} 
for some positive constant $c_1$ depending on $\mu, \xi$ and $\wt{\Omega}$. 
Integrating $\ef{tanq}+ \frac{1}{2}\ef{tanu1}+ \ef{tantheta2}+ (\frac{\xi\kappa}{4}+\frac{\beta}{2})\cdot \ef{tantheta1}$ on  $[0, t]$,  and combining all estimates above yield that   
\begin{align} 
&  \frac{1}{4}\int_{\widetilde{\Omega}} J\chi^2|D_{\xi\tau}Q|^2dy
+\frac{1}{4}  \int_{\widetilde{\Omega}} J\chi^2 K |D_{\xi\tau} U |^2 dy 
+ \frac{1}{2}(\frac{\xi\kappa}{4}+\frac{\beta}{2})
\int_{\wt{\Omega}}J\chi^2 K|a_{kj}D_{k\xi\tau}  H|^2dy ,   \nonumber \\ 
&+  c_1 
\iint_{\widetilde{\Omega}_t } J\chi^2 |D_{y\xi\tau}U|^2  dyds+ c_1
\iint_{\widetilde{\Omega}_t} J\chi^2 |D_{y\xi\tau}V|^2  dyds \nonumber  \\
& + \frac{\kappa^2}{16}(9\mu+6\lambda-\kappa) \iint_{\wt{\Omega}_t}J\chi^2|D_{\xi\tau} (a_{kj}D_k(a_{lj}D_l H ) )|^2 dy ds \nonumber
\\    
&  \leq C_0(M_0) + \delta   \int_{\widetilde{\Omega}} J\chi^2 K |D_{\xi\tau} U |^2 dy 
+ \delta  \int_{\wt{\Omega}}J\chi^2 K|a_{kj}D_{k\xi\tau}  H|^2dy  \nonumber  \\
&\q +C_\delta [t^{\frac{1}{4}} \|\rho_t\|_{L_t^2(H^1)} ^2 \|\mathbf{u}\|_{L_t^2(H^3)}^{\frac 3 2} \|\mathbf{u}\|_{L_t^\infty(H^1)}^{\frac 1 2} + t \|\rho\|_{L_t^\infty(H^2)} ^2( \|\mathbf{u}_t\|_{L_t^2(H^1)}^2 +\|\theta\|_{L_t^2(H^3)}^2)  \nonumber \\
&\q + t^{\frac 1 2}\|\rho\|_{L_t^\infty(H^2)} ^2 \|\mathbf{u}\|_{L_t^2(H^3)} \|\mathbf{u}\|_{L_t^\infty(H^1)} \|\mathbf{v}\|_{L_t^2(H^3)} ^2 + t^{\frac 1 2} \|\theta\|_{L_t^2(H^4)}^2 \|\rho_t\|_{L_t^2(H^1)} \|\rho_t\|_{L_t^\infty(L^2)}  \nonumber  \\
&\q + t^{\frac 5 8}\left\|\frac{\nabla^2 q}{\e}\right\|_{L_t^2(L^2)} ^2\|\rho\|_{L_t^\infty(H^2)} ^2( \|\mathbf{u}\|_{L_t^2(H^3)}^{\frac 3 4} \|\mathbf{u}\|_{L_t^\infty(H^1)}^{\frac 5 4}  +\|\theta\|_{L_t^2(H^3)}^{\frac 3 4} \|\theta\|_{L_t^\infty(H^1)}^{\frac 5 4}    ) 
] \nonumber  \\ 
&  \q+ C \|\mathbf{u}\|_{L_t^2(H^3)} ( t^{\frac 1 2}\| q\|_{L_t^\infty(H^2)}^2 + t^{\frac 1 2}\| \theta\|_{L_t^\infty(H^2)}^{\frac 1 2}\| \theta\|_{L_t^2(H^4)}^{\frac 1 2} + \|\n f\|_{L_t^2(L^2)} )  \nonumber \\
&\q 
+ C \e \| q\|_{L_t^\infty(H^2)}\|\mathbf{u}\|_{L_t^2(H^3)} ^2 + C t^{\frac 1 4}  \| \mathbf{u}\|_{L_t^2(H^3)}^{\frac 1 2} \| \mathbf{u}\|_{L_t^\infty(H^1)}^{\frac 1 2} (\| \mathbf{v}\|_{L_t^2(H^3)}+\|\theta\|_{L_t^2(H^4)} \nonumber 
\\  
&\q  + \left\|\frac{\nabla^2 q}{\e}\right\|_{L_t^2(L^2)}) 
+ C \| \theta\|_{L_t^2(H^4)} (   t^{\frac{1}{16}}\|\e^2\mathbf{u}\|_{L_t^\infty(H^3)} \| \mathbf{u}\|_{L_t^2(H^3)}^{\frac 7 8} \| \mathbf{u}\|_{L_t^\infty (H^1)}^{\frac 1 8}   \nonumber \\ 
&\q  +\e \|q\|_{L_t^\infty (H^2)} \|\mathbf{u}\|_{L_t^2(H^3)}+    t^{\frac{1}{4}} \|\e^2\mathbf{u}\|_{L_t^\infty(H^3)} \| \mathbf{u}\|_{L_t^2(H^3)}^{\frac 1 2} \| \mathbf{u}\|_{L_t^\infty (H^1)}^{\frac 1 2}  )  \nonumber  \\ 
& \q  + C \|\rho\|_{L_t^\infty(H^2) } [ t^{\frac 1 8}\|\theta_t\|_{L_t^2(H^2)} \| \theta\|_{L_t^\infty(H^2) }^{\frac 1 4}\| \theta\|_{L_t^2(H^4) }^{\frac 3 4} + t^{\frac 1 2 } \| \theta\|_{L_t^\infty(H^2) } \| \mathbf{u}\|_{L_t^2(H^3)}  \nonumber   \\ 
&\q + t^{\frac 1 2 } \|\mathbf{u}\|_{L_t^\infty(H^1) } \|\theta\|_{L_t^2(H^4)} +t^{\frac 5 4} \| \mathbf{u}\|_{L_t^2(H^3)}^{\frac 3 8} \| \mathbf{u}\|_{L_t^\infty (H^1)}^{\frac 5 8}    \| \theta\|_{L_t^2(H^4)}^{\frac 3 8} \| \theta\|_{L_t^\infty (H^1)}^{\frac 3 8}  ]  \nonumber \\
&\q + C \left\|\frac{\n^2 q}{\e}\right\|_{L_t^2(L^2)}  [ \|\rho\|_{L_t^\infty(H^2) }
 \|\theta_t\|_{L_t^2(H^2) }+\e \|q\|_{L_t^\infty(H^2) } \|\mathbf{u}\|_{L_t^2(H^3) }  \nonumber  \\ 
&\q + t^{\frac{1}{16}}\|\e^2\mathbf{u}\|_{L_t^\infty(H^3)} \| \mathbf{u}\|_{L_t^2(H^3)}^{\frac 7 8} \| \mathbf{u}\|_{L_t^\infty (H^1)}^{\frac 1 8} +    t^{\frac{1}{4}} \|\e^2\mathbf{u}\|_{L_t^\infty(H^3)} \| \mathbf{u}\|_{L_t^2(H^3)}^{\frac 1 2} \| \mathbf{u}\|_{L_t^\infty (H^1)}^{\frac 1 2}  ] ,  \nonumber 
\end{align}   
by use of interpolation inequalities and Sobolev embedding inequalities.  Letting $\delta>0$ be sufficiently small in above inequalities,  employing estimates    \ef{lemrho}, \ef{lem1t2}, \ef{the1}, \ef{the2}, and  \ef{fh1},  we immediately derive \ef{uh3tan}.
\hfill$\Box$

\vskip3mm

Next,  we aim to establish the estimations of normal derivatives.
Multiplying $\ef{ns2}_1$ by $a_{3i}$ and summing with respect to $i$,  we get 
\begin{equation}\label{D3t}
\begin{aligned}
&(\mu+\xi) D_3(a_{lj}D_l V^j )-\frac{1}{\e} D_3 Q+\beta  D_3 (a_{pj}D_p (a_{lj}D_l H ) )\\
&=a_{3i} K(V^i_{t}+U^j a_{kj}D_k V^i ) 
- a_{3i} \mathcal{F} + \mu [D_3(a_{lj}D_l V^j )-a_{3i}a_{kj}D_k(a_{lj}D_l V^i )], 
\end{aligned} 
\end{equation} 
where we used the fact that
$$\mu [D_3(a_{lj}D_l V^j )-a_{3i}a_{kj}D_k(a_{lj}D_l V^i )]=\mu [D_3(a_{lj}D_l U^j )-a_{3i}a_{kj}D_k(a_{lj}D_l U^i )]\q {\rm in}~{\widetilde{\Omega}_t}.$$
Since $\sum_{j=1}^3 a_{3j}a_{3j} = 1$ and $\sum_{j=1}^3 a_{1j}a_{3j} = \sum_{j=1}^3 a_{2j}a_{3j} = 0$,  we discover that the last term on the right-hand side of \ef{D3t} has the form: 
\begin{equation*}
\begin{aligned}
&\mu ( D_3a_{3j}D_3U^i + D_3a_{\tau j}D_\tau U^i + a_{\tau j}D_{3\tau}U^i - a_{3j}D_3 a_{3j} a_{3i}D_3U^i \\
& - a_{\tau j}a_{3i}D_\tau a_{ij}D_i U^i - a_{\tau j}a_{\xi j}a_{3i}D_{\tau \xi} U^i - a_{3j}a_{3i}D_3a_{\tau j}D_\tau U^i ) \q  \q {\rm  for }~  \xi,  \tau = 1,  2,  
\end{aligned} 
\end{equation*}  
which indicates that the last term of  the right-hand side of \ef{D3t} does not contain second-order normal derivative $D_{33}U$.

In order to verify the normal derivative of $a_{kj}D_k(a_{lj}D_l H)$,  applying $D_3 $ to $\ef{ns2}_3$ gives 
\begin{equation}\label{D3H} 
\begin{aligned} 
D_3 ( K (H_{t}+U^j a_{kj}D_k H ))-\frac{\kappa}{2} D_3(a_{kj}D_k(a_{lj}D_l H))+ D_3(a_{kj}D_k V^j) =D_3 \mathcal{H}.
\end{aligned}  
\end{equation} 
Below, we  use \ef{D3t} and \ef{D3H} to establish the estimates of  mixed tangential-normal derivatives and normal derivatives. 
\begin{lem}\label{lemmix}
For any $\e\in(0, 1]$ and $t\in [0, T]$,  the following estimate holds 
\begin{equation}\label{mix}
\begin{aligned}
&\frac{1}{4}  \int_{\widetilde{\Omega}} J\chi^2 | D_{3\tau} Q|^2 dy (t)+\frac{\beta }{2}  \int_{\widetilde{\Omega}}  J\chi^2 |a_{kj}D_{k\tau 3} H |^2 dy(t)\\
&+(\mu+\xi) \iint_{\widetilde{\Omega}_t} J\chi^2 |D_{3\tau}(a_{lj}D_l V^j )|^2 dyds+\frac{\beta\kappa}{2} \iint_{\widetilde{\Omega}_t} J\chi^2| D_{3\tau} (a_{pj}D_p (a_{lj}D_l H ))|^2 dyds\\
 &\leq C_0(M_0) {\rm exp}\{(t^{\frac{1}{16}}+\e^{\frac 1 4}) C_1(\mathcal{E}(t)) \}.
\end{aligned} 
\end{equation}  
\end{lem} 
\textbf{Proof}. 
Integrating the product $D_\tau \ef{D3t}\cdot J\chi^2 D_{3\tau} (a_{lj}D_l V^j)$  in $L^2(\wt{\Omega}\times[0, t])$ gives that 
\begin{align}\label{mix1} 
&(\mu+\xi) \iint_{\widetilde{\Omega}_t} J\chi^2 |D_{3\tau}(a_{lj}D_l V^j )|^2 dyds -\frac{1}{\e} \iint_{\widetilde{\Omega}_t} J\chi^2 D_{3\tau}(a_{lj}D_l V^j )  D_{3\tau} Q dyds \nonumber \\ 
& 
+\beta \iint_{\widetilde{\Omega}_t} J\chi^2 D_{3\tau}(a_{lj}D_l V^j ) D_{3\tau}   (a_{pj}D_p (a_{lj}D_l H ))  dyds \nonumber  \\
&=  \iint_{\widetilde{\Omega}_t} J\chi^2 D_{3\tau}(a_{lj}D_l V^j ) D_{\tau}  [a_{3i} K(V^i_{t}+U^j a_{kj}D_k V^i ) 
- a_{3i} \mathcal{F} \nonumber  \\
&\quad + \mu (O( D_{3\tau}U) +O(D_{\xi \tau}U ) +O(D_{y}U )) ]  dyds\\
&\leq \delta \iint_{\widetilde{\Omega}_t} J\chi^2 |D_{3\tau}(a_{lj}D_l V^j )|^2 dyds+C_\delta \{  \|f\|_{L_t^2(H^1)}^2+ \|\rho\|_{L_t^\infty(H^2)}^2[ \|\mathbf{v}_t\|_{L_t^2(H^1)}^2    \nonumber  \\
&\q+  \|\mathbf{u}\|^2_{L_t^\infty(H^1)} (  t^{\frac 1 2}  \|\mathbf{v}\|_{L_t^\infty(H^1)}\|\mathbf{v}\|_{L_t^2(H^3)} +  t^{\frac{1}{8}} \|\mathbf{v}\|_{L_t^\infty(H^1)}^{\frac 1 4}\|\mathbf{v}\|_{L_t^2(H^3)}^{\frac 7 4}  \nonumber \\
&\q +  t^{\frac{1}{4}} \|\mathbf{v}\|_{L_t^\infty(H^1)}^{\frac 1 2}\|\mathbf{v}\|_{L_t^2(H^3)}^{\frac 3 2} ) 
] + \iint_{\widetilde{\Omega}_t} J\chi^2 |D_{y \xi\tau}  U |^2 dyds +t^{\frac 1 2 } \|\mathbf{v}\|_{L_t^\infty(H^1)} \|\mathbf{v}\|_{L_t^2(H^3)} \} ,  \nonumber 
\end{align}  
where  Young's inequality,  interpolation inequalities and Sobolev  inequalities are used. 
Then,  we multiply $D_{3\tau} \ef{ns2}_2$ by $ J\chi^2 D_{3\tau} Q$  in $L^2(\wt{\Omega}\times[0, t])$ to show 
\begin{align} \label{mix2} 
& \frac{1}{4}  \int_{\widetilde{\Omega}} J\chi^2 | D_{3\tau} Q|^2 dy +\frac{1}{\e} \iint_{\widetilde{\Omega}_t} J\chi^2 D_{3\tau} Q D_{3\tau}( a_{kj}D_k V^j) dyds \nonumber  \\
&= \frac{1}{4}  \int_{\widetilde{\Omega}} J\chi^2 | D_{3\tau} Q_0|^2 dy+\frac{1}{4} \iint_{\widetilde{\Omega}_t} J a_{kj}  D_k(\chi^2  U^j  ) |D_{3\tau} Q |^2  dyds \nonumber    \\
&\q + \iint_{\widetilde{\Omega}_t} J\chi^2 D_{3\tau} Q \Big  [ -\frac{1}{2}\Big (D_{3\tau}(U^j a_{kj})D_k Q +D_{3}(U^j a_{kj})D_{k\tau} Q \\
&\q +D_{\tau}(U^j a_{kj})D_{k3} Q \Big )+D_{3\tau }\mathcal{G} \Big] dyds \nonumber \\
&\leq C_0(M_0) + C (t^{\frac 1 2} \|q\|_{L_t^\infty(H^2)}^2 \|\mathbf{u}\|_{L_t^2(H^3)} + \e \|q\|_{L_t^\infty(H^2)} \|\mathbf{u}\|_{L_t^2(H^3)}^2). \nonumber 
\end{align}  
By taking the product $-\beta D_{\tau} \ef{D3H}$ with $ J\chi^2 D_{3\tau} (a_{kj}D_k(a_{lj}D_l H))$  in $L^2(\wt{\Omega}\times[0, t])$,  one obtains
\begin{align} \label{mix3}
&\frac{\beta }{2}  \int_{\widetilde{\Omega}} K J\chi^2 |a_{kj}D_{k\tau 3} H |^2 dy+\frac{\beta\kappa}{2} \iint_{\widetilde{\Omega}_t} J\chi^2| D_{3\tau} (a_{pj}D_p (a_{lj}D_l H ))|^2 dyds \nonumber
\\
&
-\beta \iint_{\widetilde{\Omega}_t} J\chi^2 D_{3\tau}(a_{lj}D_l V^j ) D_{3\tau}   (a_{pj}D_p (a_{lj}D_l H ))  dyds \nonumber \\
&= \frac{\beta }{2}  \int_{\widetilde{\Omega}} K_0 J\chi^2 |a_{kj}D_{k\tau 3} H_0 |^2 dy+\beta \iint_{\widetilde{\Omega}_t}   \frac{1 }{2} K_t J\chi^2 |a_{kj}D_{k\tau 3} H |^2  \nonumber \\
&\q - J a_{kj} D_k(K\chi^2) D_{3\tau}  H_{t} a_{lj}D_{l\tau 3} H   -   J\chi^2 D_{3\tau} (a_{kj}D_k(a_{lj}D_l H))[D_{3\tau}  K H_{t} \nonumber
\\
&\q  +D_{\tau}  K D_3H_{t}  +D_{3}  K D_\tau H_{t} +D_{3\tau}( K U^j a_{kj}D_k H ) +  D_{3\tau} \mathcal{H} ] dyds \\
&\lesssim C_0(M_0)+ t^{\frac{1}{16}} \|\rho_t\|_{L_t^2(H^1)} \|\theta\|_{L_t^2(H^4)} ^{\frac 7 8}\|\theta\|_{L_t^\infty(H^2)} ^{\frac 1 8} + \|\theta\|_{L_t^2(H^4)} [\|\rho\|_{L_t^\infty(H^2)} ( \|\theta_t\|_{L_t^2(H^2)}  \nonumber \\
&\q + t^{\frac{1}{4}}\|\theta\|_{L_t^\infty(H^2)} \|\mathbf{u}\|_{L_t^2(H^3)}^{\frac 1 2}\|\mathbf{u}\|_{L_t^\infty(H^1)}^{\frac 1 2} + t^{\frac{1}{8}}\|\mathbf{u}\|_{L_t^\infty(H^1)}\|\theta\|_{L_t^\infty(H^2)} ^{\frac 1 4}\|\theta\|_{L_t^2(H^4)} ^{\frac 3 4} \nonumber \\
&\q + t^{\frac{1}{8}}\|\mathbf{u}\|_{L_t^2(H^3)}^{\frac 3 8}\|\mathbf{u}\|_{L_t^\infty(H^1)}^{\frac 5 8}\|\theta\|_{L_t^2(H^4)}^{\frac 3 8}\|\theta \|_{L_t^\infty(H^2)}^{\frac 5 8} ) + \e \|q\|_{L_t^\infty(H^2)}\|\mathbf{u}\|_{L_t^2(H^3)} \nonumber \\
& \q +\|\e^2\mathbf{u}\|_{L_t^\infty(H^3)} ( t^{\frac{1}{16}} \| \mathbf{u}\|_{L_t^2(H^3)}^{\frac 7 8} \| \mathbf{u}\|_{L_t^\infty (H^1)}^{\frac 1 8}  +
t^{\frac{1}{4}}   \| \mathbf{u}\|_{L_t^2(H^3)}^{\frac 1 2} \| \mathbf{u}\|_{L_t^\infty (H^1)}^{\frac 1 2} )]. \nonumber
\end{align}    
Summing over \ef{mix1},   \ef{mix2}, and  \ef{mix3} to eliminate the singular terms,  combining \ef{fh1},  \ef{lem1t2},  \ef{the1},  \ef{the2}, and \ef{lem2},  we consequently obtain \ef{mix}. 
\hfill $\Box$

\begin{lem}\label{lemD33}
For any $\e\in(0, 1]$ and $t\in [0, T]$,   it holds that
\begin{equation}\label{D33} 
\begin{aligned} 
&(\mu+\xi) \iint_{\widetilde{\Omega}_t} J\chi^2 |D_{33}(a_{lj}D_l V^j )|^2 dyds
+\frac{1}{4} \int_{\widetilde{\Omega}} J\chi^2 | D_{33} Q|^2 dy \\
&+\frac{\beta\kappa}{2} \iint_{\widetilde{\Omega}_t} J\chi^2| D_{33} (a_{pj}D_p (a_{lj}D_l H ))|^2 dyds \\
 &\leq C_0(M_0) {\rm exp}\{(t^{\frac{1}{16}}+\e^{\frac 1 4}) C_1(\mathcal{E}(t)) \}  +\iint_{\widetilde{\Omega}_t} J\chi^2 |D_{3 3\tau}  U |^2 dyds .
\end{aligned}  
\end{equation} 
\end{lem} 
\textbf{Proof}. The proof of this lemma proceeds along similar lines to the proof of Lemma \ref{lemmix}.  We integrate the inner product of $D_3\ef{D3t}$ with $ J\chi^2 D_{33} (a_{lj}D_l V^j)$ in $L^2({\wt{\Omega}}\times [0, t])$ to obtain 
\begin{align}\label{D33-1} 
&(\mu+\xi) \iint_{\widetilde{\Omega}_t} J\chi^2 |D_{33}(a_{lj}D_l V^j )|^2 dyds -\frac{1}{\e} \iint_{\widetilde{\Omega}_t} J\chi^2 D_{33}(a_{lj}D_l V^j )  D_{33} Q dyds  \nonumber  \\
& 
+\beta \iint_{\widetilde{\Omega}_t} J\chi^2 D_{33}(a_{lj}D_l V^j ) D_{33}   (a_{pj}D_p (a_{lj}D_l H ))  dyds  \nonumber  \\
&=  \iint_{\widetilde{\Omega}_t} J\chi^2 D_{33}(a_{lj}D_l V^j ) D_{3}  [a_{3i} K(V^i_{t}+U^j a_{kj}D_k V^i ) 
- a_{3i} \mathcal{F} \\
&\quad + \mu (O( D_{3\tau}U) +O(D_{\xi \tau}U ) +O(D_{y}U )) ] dyds  \nonumber \\ 
&\leq \delta \iint_{\widetilde{\Omega}_t} J\chi^2 |D_{33}(a_{lj}D_l V^j )|^2 dyds+C_\delta \{   \|f\|_{L_t^2(H^1)}^2+ \|\rho\|_{L_t^\infty(H^2)}^2[ \|\mathbf{v}_t\|_{L_t^2(H^1)}^2   \nonumber  \\
&\q+  \|\mathbf{u}\|^2_{L_t^\infty(H^1)} (  t^{\frac 1 2}  \|\mathbf{v}\|_{L_t^\infty(H^1)}\|\mathbf{v}\|_{L_t^2(H^3)} +  t^{\frac{1}{8}} \|\mathbf{v}\|_{L_t^\infty(H^1)}^{\frac 1 4}\|\mathbf{v}\|_{L_t^2(H^3)}^{\frac 7 4}  \nonumber \\
&\q +  t^{\frac{1}{4}} \|\mathbf{v}\|_{L_t^\infty(H^1)}^{\frac 1 2}\|\mathbf{v}\|_{L_t^2(H^3)}^{\frac 3 2} ) 
] +t^{\frac 1 2 } \|\mathbf{v}\|_{L_t^\infty(H^1)} \|\mathbf{v}\|_{L_t^2(H^3)}     \nonumber \\
&\q    + \iint_{\widetilde{\Omega}_t} J\chi^2 |D_{3 \xi\tau}  U |^2 dyds + \iint_{\widetilde{\Omega}_t} J\chi^2 |D_{3 3\tau}  U |^2 dyds \}. \nonumber  
\end{align}  
Next,  integrating the resulting equality $D_{33} \ef{ns2}_2\cdot J\chi^2 D_{33} Q$  in $L^2(\wt{\Omega}\times[0, t])$ implies that 
\begin{align} \label{D33-2}
& \frac{1}{4}  \int_{\widetilde{\Omega}} J\chi^2 | D_{33} Q|^2 dy +\frac{1}{\e} \iint_{\widetilde{\Omega}_t} J\chi^2 D_{33} Q D_{33}( a_{kj}D_k V^j) dyds \nonumber \\
&= \frac{1}{4}  \int_{\widetilde{\Omega}} J\chi^2 | D_{33} Q_0|^2 dy 
+\frac{1}{4} \iint_{\widetilde{\Omega}_t} J a_{kj}  D_k(\chi^2  U^j  ) |D_{33} Q |^2  dyds \nonumber    \\
&\q + \iint_{\widetilde{\Omega}_t} J\chi^2 D_{33} Q \Big  [ -\frac{1}{2}\Big (D_{33}(U^j a_{kj})D_k Q +D_{3}(U^j a_{kj})D_{k\tau} Q \\
&\q +D_{3}(U^j a_{kj})D_{k3} Q \Big )+D_{33 }\mathcal{G} \Big] dyds \nonumber \\
&\leq C_0(M_0) + C (t^{\frac 1 2} \|q\|_{L_t^\infty(H^2)}^2 \|\mathbf{u}\|_{L_t^2(H^3)} + \e \|q\|_{L_t^\infty(H^2)} \|\mathbf{u}\|_{L_t^2(H^3)}^2). \nonumber 
\end{align} 
In addition,  multiplying $-\beta D_{3} \ef{D3H}$ by $ J\chi^2 D_{33} (a_{kj}D_k(a_{lj}D_l H))$ and integrating  in $L^2(\wt{\Omega}\times[0, t])$,  we verify that  
\begin{align} \label{D33-3} 
&\frac{\beta\kappa}{2} \iint_{\widetilde{\Omega}_t} J\chi^2| D_{33} (a_{pj}D_p (a_{lj}D_l H ))|^2 dyds \nonumber
\\
&-\beta \iint_{\widetilde{\Omega}_t} J\chi^2 D_{33}(a_{lj}D_l V^j ) D_{33}   (a_{pj}D_p (a_{lj}D_l H ))  dyds \nonumber \\ 
&= \beta \iint_{\widetilde{\Omega}_t} J\chi^2 D_{33} (a_{kj}D_k(a_{lj}D_l H))D_{33}  [K(H_{t}+ U^j a_{kj} D_k H)   + \mathcal{H} ] dyds\\
&\leq  C \|\theta\|_{L_t^2(H^4)} [\|\rho\|_{L_t^\infty(H^2)} (\|\theta_t\|_{L_t^2(H^2)}  + t^{\frac{1}{4}}\|\theta\|_{L_t^\infty(H^2)} \|\mathbf{u}\|_{L_t^2(H^3)}^{\frac 1 2}\|\mathbf{u}\|_{L_t^\infty(H^1)}^{\frac 1 2}  \nonumber \\
&\q + t^{\frac{1} {8}}\|\mathbf{u}\|_{L_t^\infty(H^1)}\|\theta\|_{L_t^\infty(H^2)} ^{\frac 1 4}\|\theta\|_{L_t^2(H^4)} ^{\frac 3 4} + t^{\frac{1}{8}}\|\mathbf{u}\|_{L_t^2(H^3)}^{\frac 3 8}\|\mathbf{u}\|_{L_t^\infty(H^1)}^{\frac 5 8}  \nonumber \\
&\q \cdot \|\theta\|_{L_t^2(H^4)}^{\frac 3 8}\|\theta \|_{L_t^\infty(H^2)}^{\frac 5 8} ) + \e \|q\|_{L_t^\infty(H^2)}\|\mathbf{u}\|_{L_t^2(H^3)}   \nonumber \\ 
&\q +\|\e^2\mathbf{u}\|_{L_t^\infty(H^3)} ( t^{\frac{1}{16}} \| \mathbf{u}\|_{L_t^2(H^3)}^{\frac 7 8} \| \mathbf{u}\|_{L_t^\infty (H^1)}^{\frac 1 8}  +
t^{\frac{1}{4}} \| \mathbf{u}\|_{L_t^2(H^3)}^{\frac 1 2} \| \mathbf{u}\|_{L_t^\infty (H^1)}^{\frac 1 2} )]. \nonumber 
\end{align}   
Adding \ef{D33-1},  \ef{D33-2} and \ef{D33-3} together,  we can eliminate the singular terms and highest-order terms on the left-hand side. Finally,  by use of the estimates \ef{fh1},  \ef{lem1t2}, \ef{the1}, \ef{the2}, and \ef{lem2},    we have proved \ef{D33}. 
\hfill $\Box$

\vskip3mm  

At this stage,  we require the  estimate for  $\iint_{\widetilde{\Omega}_t} J\chi^2 |D_{3 3\tau}  U |^2 dyds$.
\begin{lem}\label{lemD33t} 
For any $\e\in(0, 1]$ and $t\in [0, T]$,   we have the following estimates 
\begin{equation}\label{D33tau}
\begin{aligned}  
\iint_{\widetilde{\Omega}_t} J\chi^2 |D_{33\tau} U|^2 dyds
\leq C_0(M_0) {\rm exp}\{(t^{\frac{1}{16}}+\e^{\frac 1 4}) C_1(\mathcal{E}(t)) \}  .
\end{aligned}  
\end{equation} 
\end{lem} 
\textbf{Proof}. We apply $\ef{ns1}_2$ to introduce an auxiliary Stokes problem in original coordinates near the boundary   
\begin{equation}   \nonumber
\left\{\begin{aligned} 
& -\mu \Delta_x[ (\chi D_\tau U) \circ \Lambda^{-1} ] +\frac{1}{\e} \n_x[( \chi D_\tau Q  ) \circ \Lambda^{-1} ] = \mathcal{F}_1 &{\rm in}~ W\cap \Omega, \\
& \dv_x [ (\chi D_\tau U) \circ \Lambda^{-1}  ] =\mathcal{G}_1 &{\rm in}~ W\cap \Omega, \\
&(\chi D_\tau U) \circ \Lambda^{-1}=0 &{\rm on}~ \p(W\cap \Omega), 
\end{aligned}\right. 
\end{equation} 
where
\begin{equation} \nonumber
\begin{aligned} 
 \mathcal{F}_1^i&:=- \chi D_\tau[ K (U^i_{t}+U^j a_{kj}D_k U^i )  ] +(\mu+\lambda)\chi D_\tau[  a_{ki}D_k(a_{lj}D_l U^j ) ]\\
 &\q + O(1)( D_l U^i+  D_{kl} U^i+ \frac{1}{\e}  D_k Q), \\
  \mathcal{G}_1&:= O(1)( D_\tau U^j+  D_{k} U^j + D_{k \tau } U^j).
\end{aligned}  
\end{equation}  
According to the regularity theory of the Stokes problem,  one has
\be \nonumber
\om |\Delta_x[ (\chi D_\tau U) \circ \Lambda^{-1} ]|^2 dx\leq C( \| \mathcal{F}_1\|^2_{L^2(W\cap\Omega)} +\| \mathcal{G}_1\|^2_{H^1(W\cap\Omega)} ).
\ee
Employing \ef{a3j},  we notice that 
\be\nonumber
D_{33\tau}U = \sum_{k, l=1}^3 \bigg(\sum_{j=1}^3 a_{kj}a_{lj} \bigg) D_{kl\tau}U- \sum_{1\leq k, l\leq  2} \sum_{j=1}^3 a_{kj}a_{lj}  D_{kl\tau}U, 
\ee 
as well as
\begin{equation} \nonumber
\begin{aligned} 
& \om |\Delta_x[ (\chi D_\tau U) \circ \Lambda^{-1} ]|^2 dx \\
& = \int_{\tilde{\Omega}} J \bigg| \sum_{j=1}^3 \sum_{k=1}^3 a_{kj} D_k \bigg( \sum_{l=1}^3 a_{lj}D_l(\chi D_\tau U) \bigg )  \bigg|^2 dy\\
&= \int_{\tilde{\Omega}} J \chi^2 \bigg| \sum_{j, k, l=1}^3   a_{kj}  a_{lj}  D_{kl\tau} U     \bigg|^2 dy +O(1) \int_{\tilde{\Omega}} (  |D_\tau U|^2+|D_{y\tau} U|^2  ) dy.
\end{aligned}  
\end{equation} 
Therefore,  we deduce 
\begin{equation} \nonumber
\begin{aligned} 
 \int_{\tilde{\Omega}} J \chi^2  |D_{33\tau} U |^2 dy  &\leq  C( \| \mathcal{F}_1\|^2_{L^2(W\cap\Omega)} +\| \mathcal{G}_1\|^2_{H^1(W\cap\Omega)} )\\ 
 &\q +\int_{\tilde{\Omega}} J \chi^2  |D_{\varsigma \xi\tau} U |^2 dy +O(1) \int_{\tilde{\Omega}} (  |D_\tau U|^2+|D_{y\tau} U|^2  ) dy,
\end{aligned}  
\end{equation} 
where $ \varsigma,  \xi, \tau=1, 2$.  In view of the definitions of $\mathcal{F}_1$ and $\mathcal{G}_1$,  one shows 
\begin{equation} \nonumber
\begin{aligned} 
\| \mathcal{F}_1\|^2_{L^2(W\cap\Omega)} 
&\lesssim  \int_{\tilde{\Omega}}  J\chi^2 |D_{\tau k}(a_{lj}D_l U^j )|^2 dy+ \left\|\frac{\n q}{\e} \right\|_{L^2}^2 + \|\mathbf{u}\|_{H^1}^2 \\
& \q + \|\rho\|_{H^2}^2  (\|\mathbf{u}_t \|_{H^1} ^2   + \|\mathbf{u}\|_{H^1}^{\frac 5 2} \|\mathbf{u}\|_{H^3}^{\frac 3 2}  )   + \|\mathbf{u}\|_{H^1} \|\mathbf{u}\|_{H^3}   , 
\end{aligned}  
\end{equation}  
and 
\be\nonumber 
\n_x \mathcal{G}_1(\Lambda^{-1}(x))= O(1)( D_l U^j+  D_{kl} U^j + \chi  D_{k \tau }(a_{lj} D_l U^j )),  
\ee 
which gives that   
 \be\nonumber
\| \mathcal{G}_1\|_{H^1(W\cap\Omega)}^2 \leq C   \int_{\tilde{\Omega}}  J\chi^2 |D_{\tau k}(a_{lj}D_l U^j )|^2 dy + \|\mathbf{u}\|_{H^1} \|\mathbf{u}\|_{H^3} .
\ee   
Thus,  from Lemmas \ref{lemSto},  \ref{lemtan}, and \ref{lemmix},  we obtain \ef{D33tau}.
\hfill $\Box$

\vskip3mm  

\noindent \textbf{Proof of Proposition  \ref{lemprop}}. Combining Lemmas \ref{lemrho}-\ref{lemD33t}  and Remark \ref{rmk}, we consequently obtain the energy inequality \ef{prop}. \hfill $\Box$

\vskip3mm 

At last, we use Proposition \ref{lemprop} to  prove Theorem \ref{uethm}.  To justify the uniform estimates,  we restore the superscript  $\e$ in what follows.

\noindent  \textbf{Proof of Theorem \ref{uethm}}. 
 Notice that  
 \be\label{Theta} \Theta^\e=\Theta_0^\e+\int_0^t \Theta_s^\e ds \geq A- \sqrt{t} C(\Omega)\| \Theta_t^\e\|_{L_t^2(H^2)}
 \geq   A- C(\Omega)\sqrt{t \mathcal{E}^\e (t)}  .
\ee
Thus, to prove \ef{bound}, it suffices to prove that there exist $\e_0\in (0, 1]$ and $T \in (0, 1]$ such that $\mathcal{E}^\e(t) < C$. 
Let $ T^\epsilon < +\infty $ be the maximal existence time  for the strong solution given in Theorem \ref{le}.  From the results of Proposition \ref{lemprop}, it follows that for any $ 0 \leq t \leq \min\{T^\epsilon, 1\} $, 
\be \label{ue1}
\mathcal{E}^\epsilon(t) \leq C_0(M_0^\epsilon) \exp \{(t^{\frac{1}{16}} + \epsilon^{\frac{1}{4}})C_1(\mathcal{E}^\epsilon(t)) \},
\ee  
where $ M_0^\epsilon \leq D_0 $ for $  \epsilon \in(0, 1] $. In the sequel, we choose a finite $ D > C_0(D_0) $, then select $ \epsilon_0 \in (0, 1] $ and $ T_1 \leq 1 $ sufficiently small such that
\be\label{ue2}
C_0(D_0) \exp \{(T_1^{\frac{1}{16}} + \epsilon_0^{\frac{1}{4}})C_1(D)\} < D\q {\rm and}\q  C(\Omega)\sqrt{T_1 D}<\frac A4.
\ee
Thereby, the estimate \ef{Theta} gives 
$$\Theta^\e\geq A-C(\Omega) \sqrt{t \mathcal{E}^\e(t)}\geq A-C(\Omega) \sqrt{T_1 D} >\frac{3A}{4},$$ which recovers the assumption that  $\Theta^\e\geq \frac{A}{2}$. 
Let $ t < \min\{T^\epsilon, T_1\} $.
We claim that $ \mathcal{E}^\epsilon(t) \neq D $. Otherwise, we obtain from \ef{ue1} and  \ef{ue2} that 
$$ 
D=\mathcal{E}^\epsilon(t)\leq C_0(M^\e_0) \exp \{(t^{\frac{1}{16}} + \epsilon^{\frac{1}{4}})C_1(\mathcal{E}^\e (t))\} < D,
$$
which is a contradiction. 
Moreover, we assume without restriction that $ D_0 < D $, so that $ M_0^\epsilon < D $.
 Since  $ \mathcal{E}^\epsilon(t) $ is continuous, we have
\be\label{ue3}
\mathcal{E}^\epsilon(t) < D \quad \text{for} \quad t < \min\{T^\epsilon, T_1\} \quad \text{and} \quad \epsilon \in (0, \epsilon_0]. 
\ee

Therefore, we conclude that $ T^\epsilon > T_1 $ for $ \epsilon \in (0, \epsilon_0] $.  Otherwise, using the  uniform boundedness in \ef{ue3} and applying Theorem \ref{le} repeatedly,  we can extend the time interval of existence to $[0, T_1]$, which contradicts   the maximality of $ T^\epsilon $. 
\hfill$\Box$

\vskip3mm
 \noindent \textbf{Acknowledgments}. 
This work was partially supported  by the National Natural Science Foundation of China under Grant Nos. 12131007 and 12471216.

\vskip3mm
\noindent \textbf{Conflict of Interest.} The authors declare that they have no conflict of interest. 

\vskip3mm
\noindent \textbf{Data availability.} No data were  used   in the article.

\end{document}